\documentclass[12pt,a4paper,twoside]{article}
\usepackage[margin=25.7mm]{geometry}
\usepackage{titling}
\usepackage{titletoc}
\usepackage{url}

\usepackage[utf8]{inputenc}
\usepackage[T1]{fontenc}
\usepackage{amsmath}
\usepackage{amsfonts}
\usepackage{amssymb}
\usepackage{amsthm}

\usepackage{mathrsfs}
\usepackage{mathtools}
\usepackage{stmaryrd}
\usepackage[all]{xy}
\usepackage{makeidx}

\usepackage{natbib}
\bibpunct{(}{)}{,}{a}{}{,}
\usepackage{tocbibind}

\usepackage[french,english.british]{babel}\frenchsetup{GlobalLayoutFrench=false}
\usepackage{xspace}
\RequirePackage{etoolbox}

\makeatletter
\def\th@plain{\slshape}
\patchcmd{\th@remark}{\itshape}{\slshape}{}{}
\DeclareFontEncoding{LS1}{}{}
\DeclareFontEncoding{LS2}{}{\noaccents@}
\DeclareFontSubstitution{LS1}{stix}{m}{n}
\DeclareFontSubstitution{LS2}{stix}{m}{n}
\DeclareSymbolFont{stix-letters}       {LS1}{stix}     {m}{it}
\DeclareSymbolFont{stix-arrows1}       {LS1}{stixsf}   {m} {n}
\DeclareSymbolFont{stix-operators}     {LS1}{stix}     {m} {n}
\DeclareSymbolFont{stix-largesymbols}  {LS2}{stixex}   {m} {n}
\SetSymbolFont{stix-letters}     {bold}{LS1}{stix}     {b}{it}
\SetSymbolFont{stix-arrows1}     {bold}{LS1}{stixsf}   {b} {n}
\SetSymbolFont{stix-operators}   {bold}{LS1}{stix}     {b} {n}
\SetSymbolFont{stix-largesymbols}{bold}{LS2}{stixex}   {b} {n}

\def\stix@undefine#1{%
    \if\relax\noexpand#1\let#1=\@undefined\fi}
\def\stix@MathSymbol#1#2#3#4{%
    \stix@undefine#1%
    \DeclareMathSymbol{#1}{#2}{#3}{#4}}
\def\stix@MathDelimiter#1#2#3#4#5#6{%
    \stix@undefine#1%
    \DeclareMathDelimiter{#1}{#2}{#3}{#4}{#5}{#6}}

\stix@MathSymbol{\stixbot}            {\mathord}  {stix-letters}     {"F2}
\stix@MathSymbol{\stixexists}         {\mathord}  {stix-operators}   {"C7}
\stix@MathSymbol{\stixtop}            {\mathord}  {stix-letters}     {"F1}
\stix@MathSymbol{\stixvee}            {\mathbin}  {stix-operators}   {"E2}
\stix@MathSymbol{\stixwedge}          {\mathbin}  {stix-operators}   {"E1}
\stix@MathSymbol{\stixin}             {\mathrel}  {stix-operators}   {"CB}
\stix@MathSymbol{\stixnotin}          {\mathrel}  {stix-operators}   {"CC}
\stix@MathSymbol{\stixrightarrow}     {\mathrel}  {stix-arrows1}     {"99}
\stix@MathSymbol{\stixrightleftarrows}{\mathrel}  {stix-arrows1}     {"CB}
\stix@MathSymbol{\stixsubseteq}       {\mathrel}  {stix-letters}     {"D3}
\stix@MathSymbol{\stixvdash}          {\mathrel}  {stix-letters}     {"EF}
\stix@MathSymbol{\stixbigveeop}       {\mathop}   {stix-largesymbols}{"B5}
\stix@MathSymbol{\stixbigwedgeop}     {\mathop}   {stix-largesymbols}{"B4}
\stix@MathSymbol{\stixdotminus}       {\mathbin}  {stix-operators}   {"E9}
\stix@MathDelimiter{\stixlBrack}      {\mathopen} {stix-largesymbols}{"E0}{stix-largesymbols}{"06}
\stix@MathDelimiter{\stixrBrack}      {\mathclose}{stix-largesymbols}{"E1}{stix-largesymbols}{"07}

\newcommand\stixbigwedge{\DOTSI\stixbigwedgeop\slimits@}
\newcommand\stixbigvee{\DOTSI\stixbigveeop\slimits@}

\AtBeginDocument{%
    \@ifpackageloaded{amsmath}{%
    }{%
        \let\ilimits@=\nolimits
        \let\slimits@=\relax
        \let\DOTSI\relax
    }
}

\RequirePackage{amsbsy}

\newcommand\Exists{\boldsymbol{\stixexists}}

\newcommand\VDash{\boldsymbol{\stixvdash}}
\newcommand\Land{\boldsymbol{\stixwedge}}

\newcommand\Bot{\boldsymbol{\stixbot}}

\newcommand \vdg{\VDash}
\newcommand \Vd {\,\vdg}
\newcommand \vd {\,\,\vdg}
\newcommand \vii{\Land}

\newcommand \vdi[1] {\mathrel{\vd_{#1}}}
\newcommand \Vdi[1] {\mathrel{\Vd_{#1}}}

\makeatother

\usepackage{url}
\usepackage[bookmarksopen=false,breaklinks=true,%
      backref=page,pagebackref=true,plainpages=false,%
      hyperindex=true,pdfstartview=FitH,colorlinks=true,%
      pdfpagelabels=true,linkcolor=blue,%
      citecolor=red,urlcolor=red,
      ]%
   {hyperref}

   \usepackage{doi}

\newcounter{bidon}
\newcommand{\rdb}{\refstepcounter{bidon}}
\DeclareMathAlphabet{\mathpzc}{OT1}{pzc}{m}{it}

\newcommand \sibil[1] {#1}
\newcommand \sinotbil[1] {}

\newcommand{\di}{\,{\vert}\,}

\newcommand {\junk}[1]{}

\newcommand\ndsp{\textstyle}

\newcommand \noi {\noindent}
\renewcommand \ss {\smallskip}
\newcommand \sni {\ss\noi}

\newcommand \ms {\medskip}
\newcommand \mni {\ms\noi}

\newcommand \alb {\allowbreak}

\renewcommand \leq{\leqslant}

\renewcommand \geq{\geqslant}

\newcommand\eti{^\times}
\newcommand \epr{^\perp}

\newcommand \etoz{$^*$}

\newcommand \Rti {\gR^{\!\times}}

\newcommand\equidef{\buildrel{{\rm def}}\over{\;\Longleftrightarrow\;}}
\newcommand\eqdef{\buildrel{\rm def}\over {\;=\;}}
\newcommand\eqdefi{\buildrel{\rm def}\over {\;=\;}}

\newcommand{\pref}[1]{\textup{\hbox{\normalfont(\ref{#1})}}}

\newcommand \aqo[2] {#1\sur{\gen{#2}}\!}

\newcommand \ci[1] {{{#1}^\circ}}

\newcommand \gen[1] {\left\langle{#1}\right\rangle}

\newcommand \so[1] {\left\{ #1   \right\}}

\newcommand \sotq[2]{\so{\,#1\mathrel{;}#2\,}}

\newcommand \sur[1] {\!\left/#1\right.}

\newcommand{\mt}{\mapsto}

\renewcommand \leq{\leqslant}
\renewcommand \geq{\geqslant}

\newcommand \Prod {\prod\nolimits}
\newcommand \Ex {{\exists}}

\newcommand \vda {\,\vdash\,}

\newcommand \dar {\mathord{\downarrow}}

\newcommand\tsbf[1]{\textsf{\textbf{\textup{#1}}}}
\newcommand\lab[1]{\item[\tsbf{#1}]}
\newcommand\Lab[1]{\rdb\item[\tsbf{#1}]\label{Ax#1}}
\newcommand\fLab[1]{\rdb\item[\tsbf{#1}]\label{fAx#1}}
\newcommand\Tsbf[1]{\hyperref[Ax#1]{\tsbf{#1}}}

\newcommand\Sa[1]{\hyperref[theorie#1]{\sa{#1}}}
\newcommand\sa[1]{\hbox{\usefont{T1}{pzc}{m}{it}#1}\,}

\newcommand\sA[1]{\hbox{\small\usefont{T1}{pzc}{m}{it}#1}\,}

\newcommand \snic[1] {\sni\centerline{$#1$}

\ss}

\newcommand \eoe {\hbox{}\nobreak\hfill
\vrule width .5em height .5em depth 0mm \par \smallskip}

\newcommand \ov[1] {\overline{#1}}

\newcommand \wi[1] {\widetilde{#1} }

\renewcommand\matrix[1]{
\begin{array}{cccccccccc} 
#1
\end{array}}

\newcommand\cmatrix[1]{\left[ \matrix{#1} \right]}

\newcommand \DeuxRegles[2]{%
\vspace{-1em}\DeuxCols
{\begin{enumerate}  #1
\end{enumerate}
}
{\begin{enumerate}  #2
\end{enumerate}
}
\vspace{-.3em}
}
\usepackage{multicol}
\renewcommand\DeuxRegles[2]{%
  \begin{multicols}{2}
    \begin{itemize}
      #1
      #2
    \end{itemize}
  \end{multicols}
}

\newcommand \Regles[1]{%
\vspace{-1em}\UneCol{
\begin{enumerate}
{#1}
\end{enumerate}
}
\vspace{-.3em}
}
\renewcommand\Regles[1]{%
  \begin{itemize}
    #1
  \end{itemize}
}

\newcommand \labu {\lab{$\bullet$}}

\makeatletter
\def\revddots{\mathinner{\mkern1mu\raise\p@
\vbox{\kern7\p@\hbox{.}}\mkern2mu
\raise4\p@\hbox{.}\mkern2mu\raise7\p@\hbox{.}\mkern1mu}}
\makeatother

\newcommand \BB{\mathbb {B}}

\newcommand \NN{\mathbb {N}}
\newcommand \ZZ{\mathbb {Z}}

\newcommand \QQ{\mathbb {Q}}

\newcommand \gk {\mathbf{k}}

\newcommand \gA {\mathbf{A}}
\newcommand \gB {\mathbf{B}}
\newcommand \gC {\mathbf{C}}
\newcommand \gD {\mathbf{D}}

\newcommand \gE {\mathbf{E}}

\newcommand \gK {\mathbf{K}}

\newcommand \gR {\mathbf{R}}
\newcommand \R {\gR}

\newcommand \gT {\mathbf{T}}
\newcommand \gV {\mathbf{V}}

\newcommand \gZ {\mathbf{Z}}

\newcommand \Red {_{\mathrm{red}}}
\newcommand \qim {_{\mathrm{min}}}
\newcommand \Rred {\gR\Red}
\newcommand \Rqim {\gR\qim}
\newcommand \Rmin {\Rqim}

\newdimen\xyrowsp
\newcommand{\SCO}[6]{
\xymatrix @R = \xyrowsp {
                                  &1 \ar@{-}[dl] \ar@{-}[dr] \\
#3 \ar@{-}[ddr]                   &   & #6 \ar@{-}[ddl] \\
                                  &\bullet\ar@{-}[d] \\
                                  &\bullet   \\
#2 \ar@{-}[ddr] \ar@{-}[uur]      &   & #5 \ar@{-}[ddl] \ar@{-}[uul] \\
                                  &\bullet \ar@{-}[d] \\
                                  &\bullet  \\
#1 \ar@{-}[uur]                   &   & #4 \ar@{-}[uul] \\
                                  & 0 \ar@{-}[ul] \ar@{-}[ur] \\
}
}

\newcommand{\SCOR}[8]{
\xymatrix @R =.4em @C =3em {
                                  &\bullet \ar@{-}[ddl] \ar@{-}[ddr] \\
                                  &#7\ar@{-}[u] \\
#3 \ar@{-}[ddr]                   &   & #6 \ar@{-}[ddl] \\
                                  &\bullet\ar@{-}[d] \\
                                  &\bullet   \\
#2 \ar@{-}[ddr] \ar@{-}[uur]      &   & #5 \ar@{-}[ddl] \ar@{-}[uul] \\
                                  &\bullet \ar@{-}[d] \\
                                  &\bullet  \\
#1 \ar@{-}[uur]                   &   & #4 \ar@{-}[uul] \\
                                  &#8\ar@{-}[d] \\
                                  &\bullet \ar@{-}[uul] \ar@{-}[uur] \\
}
}

\DeclareMathOperator\Ann{Ann}
\DeclareMathOperator\Frac{Frac}
\DeclareMathOperator\Mat{Mat}
\DeclareMathOperator\Vr{Vr}
\DeclareMathOperator*\Vp{V^\prime}
\DeclareMathOperator\Nrn{Nrn}
\DeclareMathOperator\Val{\mathsf{Val}}
\DeclareMathOperator\val{\mathsf{val}}
\DeclareMathOperator\Kdim{\mathsf{Kdim}}

\DeclareMathOperator\Valp{\mathsf{Val^\prime}}
\DeclareMathOperator\Zar{\mathsf{Zar}}
\DeclareMathOperator\Vdim{\mathsf{Vdim}}
\DeclareMathOperator\vdim{\mathsf{vdim}}
\DeclareMathOperator\dimv{\mathsf{dimv}}

\newcommand\MA[1]{\mathop{#1}\nolimits}

\newcommand{\ZarR}{\Zar(\gR)}

\newcommand \cA {{\cal A}}

\newcommand \cL {{\cal L}}

\newcommand \cP {{\cal P}}

\newcommand \rD {\mathrm{D}}

\newcommand \DR {\rD_{\gR}}

\newcommand\fa{\mathfrak{a}}
\newcommand\fb{\mathfrak{b}}

\newcommand\fA{\mathfrak{A}}

\newcommand\fp{\mathfrak{p}}

\newcommand\fv{\mathfrak{v}}

\newcommand \vu {\vee} 
\newcommand \vi {\wedge} 
\newcommand \Vu {\bigvee}
\newcommand \Vi {\bigwedge}

\newcommand \Un {\mathbf{1}}
\newcommand \Deux {\mathbf{2}}

\newcommand \ua  {{\underline{a}}}

\newcommand \ux {{\underline{x}}}

\newcommand \uX {\underline{X}}
\newcommand \uy{{\underline{y}}}
\newcommand \uY  {{\underline{Y}}}

\newcommand \an {a_1,\ldots,a_n}

\newcommand \xk {x_1,\ldots,x_k}
\newcommand \Xk {X_1,\ldots,X_k}
\newcommand \xm {x_1,\ldots,x_m}

\newcommand \xn {x_1,\ldots,x_n}

\newcommand \Xn {X_1,\ldots,X_n}

\newcommand \Ym {Y_1,\ldots,Y_m}

\newcommand \RX {\gR[X]}

\newcommand \RXn {{\gR[\Xn]}}
\newcommand \Rxn {{\gR[\xn]}}

\newcommand{\SCo}[8]{
\xymatrix @R = #1 @C = #2{
                                  &1 \ar@{-}[dl] \ar@{-}[dr] \\
#5 \ar@{-}[ddr]                   &   & #8 \ar@{-}[ddl] \\
                                  &\bullet\ar@{-}[d] \\
                                  &\bullet   \\
#4 \ar@{-}[ddr] \ar@{-}[uur]      &   & #7 \ar@{-}[ddl] \ar@{-}[uul] \\
                                  &\bullet \ar@{-}[d] \\
                                  &\bullet  \\
#3 \ar@{-}[uur]                   &   & #6 \ar@{-}[uul] \\
                                  & 0 \ar@{-}[ul] \ar@{-}[ur] \\
}
}

\newcommand \Pfe {{\rm P}_{{\rm fe}}}

\begin{document}
\selectlanguage{english}

\thispagestyle{empty}
~ 
\vspace{3cm}

\noindent In this file you find the English version starting on the page  numbered \pageref{beginenglish}.

\bigskip  \noindent  {\Large \bf Valuative dimension, constructive points of view}

\medskip 
This is quasi the same text as in 
Lombardi H., Neuwirth S. and Yengui I.
Valuative dimension, constructive points of view.
{\it Journal of Algebra}, {\bf 647}, 206--229, 2024.

We have added the examples \fbox{$\vdim\leq 3\Rightarrow\Vdim\leq 3$}  and 
\fbox{$\vdim\leq 4\Rightarrow\Vdim\leq 4$} in the paragraph ``{\bf Proof of the converse inequality}''.

We have fixed a typo in Definition \ref{defivalkK}.

\bigskip \noindent  
Then the French version begins on the page numbered \pageref{beginfrench}.

\medskip\noindent   {\Large \bf Dimension valuative, points de vue constructifs}

\smallskip \noindent \foreignlanguage{french}{Le lecteur ou la lectrice sera sans doute surprise de l'alternance des sexes ainsi que de l'orthographe du mot `corolaire', avec d'autres innovations auxquelles elle n'est pas habituée. En fait, nous avons essayé de suivre au plus près les préconisations de l'orthographe nouvelle recommandée, telle qu'elle est enseignée aujourd'hui dans les écoles en France.}  

\bigskip\noindent   {\large \bf Authors}  

\smallskip \noindent Henri Lombardi, Université de Franche-Comté, CNRS, UMR 6623, LmB, 25000 Besançon, France, \url{henri.lombardi@univ-fcomte.fr}\\
email: {\tt henri.lombardi@univ-fcomte.fr}

\smallskip \noindent Stefan Neuwirth, Université de Franche-Comté, CNRS, UMR 6623, LmB, 25000 Besançon, France, \url{stefan.neuwirth@univ-fcomte.fr}\\
email: {\tt stefan.neuwirth@univ-fcomte.fr}

\smallskip \noindent  Ihsen Yengui,
Département de mathématiques,  Faculté des sciences de Sfax, Université de Sfax, 3000 Sfax, Tunisia.\\ 
email: {\tt ihsen.yengui@fss.rnu.tn}

%

\normalsize
\newpage
\thispagestyle{empty}

~

\pagestyle{headings}
\patchcmd{\sectionmark}{\MakeUppercase}{}{}{}
\setcounter{page}{0}
\renewcommand\thepage{E\arabic{page}}


\begingroup

\def\proofname{\textsl{Proof}}

\startcontents[english]



\theoremstyle{plain}
\newtheorem{theorem}{Theorem}[section]
\newtheorem{thdef}[theorem]{Theorem and definition}
\newtheorem{lemma}[theorem]{Lemma}
\newtheorem{corollary}[theorem]{Corollary}
\newtheorem{proposition}[theorem]{Proposition}
\newtheorem{propdef}[theorem]{Proposition and definition}
\newtheorem{fact}[theorem]{Fact}
\newtheorem{convention}[theorem]{Convention}

\theoremstyle{definition}
\newtheorem{note}[theorem]{Note}
\newtheorem{context}[theorem]{Context}
\newtheorem{conjecture}[theorem]{Conjecture}
\newtheorem{definition}[theorem]{Definition}
\newtheorem{definitions}[theorem]{Definitions}
\newtheorem{descri}[theorem]{Description}
\newtheorem{notation}[theorem]{Notation}
\newtheorem{definota}[theorem]{Definition and notation} 
\newtheorem{problem}[theorem]{Problem}
\newtheorem{question}[theorem]{Question}
\newtheorem{ter}[theorem]{Terminology}

\theoremstyle{remark}
\newtheorem{notes}[theorem]{Notes}
\newtheorem{remark}[theorem]{Remark}
\newtheorem{remarks}[theorem]{Remarks}
\newtheorem{comment}[theorem]{Comment}
\newtheorem{comments}[theorem]{Comments}
\newtheorem{example}[theorem]{Example}
\newtheorem{examples}[theorem]{Examples}

\normalsize

\newpage
\title{Valuative dimension, constructive points of view}
\sibil{\author{Henri Lombardi, Stefan Neuwirth, Ihsen Yengui}}
\sinotbil{
\author{%
Henri Lombardi
\thanks{Université de Franche-Comté, Laboratoire de mathématiques de Besançon, UMR, CNRS 6623,  25000
Besançon, France. {\tt henri.lombardi@univ-fcomte.fr}}
\and
Stefan Neuwirth
\thanks{Université de Franche-Comté, Laboratoire de mathématiques de Besançon, UMR CNRS 6623, 16~route de Gray, 25000
Besançon, France. {\tt stefan.neuwirth@univ-fcomte.fr}}
 \and  Ihsen Yengui
\thanks{Département de mathématiques,  Faculté des sciences de Sfax, Université de Sfax, 3000 Sfax, Tunisia. {\tt ihsen.yengui@fss.rnu.tn}}}
}

\maketitle

\rdb
\label{beginenglish}
\newcommand \gui[1] {``#1''}

\begin{abstract}There are several classical characterisations of the valuative dimension of a commutative ring. Constructive versions of this dimension have been given and proven to be equivalent to the classical notion within classical mathematics, and they can be used for the usual examples of commutative rings. To the contrary of the classical versions, the constructive versions have a clear computational content.
  This paper investigates the computational relationship between three possible constructive definitions of the valuative dimension of a commutative ring. In doing so, it proves these constructive versions to be equivalent within constructive mathematics.
\end{abstract}



\newpage
\setcounter{tocdepth}{4}
\markboth{Contents}{Contents}

\printcontents[english]{}{1}{}

\setcounter{section}{0}
\setcounter{subsection}{0}
\setcounter{theorem}{0}

\section{Introduction}
This article is written in Bishop's style of constructive mathematics \citep*{Bi67,BB85,BR1987,ACMC,MRR,Yen2015}.

The vocabulary and notation of dynamical algebraic structures will be used when necessary: see \citealt*{CLR01,CL05,Lom06,Lom2020}.

In this paper, we compare the different constructive versions of the valuative dimension found in \citealt*{Coq2009,KY2020,CACM,ACMC}, as well as a constructive version which extends that of \citealt{Coq2009} to the case of a not necessarily integral ring.

\smallskip The reader who does not know about constructive mathematics in Bishop's style may look up Chapters 1 and 2 of \citealt{Bi67}, its reviews \citealt{stolzenberg70,myhill72}, and the paper \citealt{CL05}. 

\smallskip When a classical definition or a classical theorem uses abstract notions without computational content, constructive mathematicians try to find what they call a \textsl{constructive version} of this definition or theorem. This version has to be equivalent within classical mathematics to the classical one. Moreover, it is necessary that basic classical examples can be dealt with for the constructive version. E.g., the constructive version of a local ring is simply a ring in which, each time the sum of finitely many elements is invertible, one of these elements is invertible.

\subsection{Definition and characterisations of the valuative dimension in classical mathematics}\label{subsecdimval}

In classical mathematics, the valuative dimension of an integral ring $\gR$ is the maximal length $n$ of a chain of valuation rings $\gV_0\subsetneq \dots \subsetneq \gV_n=\gK$ in the field of fractions $\gK=\Frac\gR$ which contain $\gR$.

The valuative dimension of an arbitrary ring is defined as the upper bound of the valuative dimension of its integral quotients \citep{Cah90}.

In classical mathematics, the following equivalences are well known ($\Kdim(\gR)$ denotes the Krull dimension of the ring $\gR$, i.e.\ the maximal length of a chain
of prime ideals in~$\gR$).

Let us recall that a discrete field is of Krull dimension $0$; the dimension of the trivial ring, which has no integral quotient, is by convention equal to $-1$.

\begin{theorem} \label{thclassdimval} Let $\gR$ be a nontrivial commutative integral ring with $\gK=\Frac(\gR)$. The following properties are 
equivalent.
\begin{enumerate}
\item\label{thclassdimval1} $\gR$ is of valuative dimension $\leq n$.
\item\label{thclassdimval2} For any integer $k$ and all $x_1,\dots,x_k\in \gK$, $\Kdim(\gR[\xk])\leq n$.
\item\label{thclassdimval3} For any integer $k$, $\Kdim(\gR[\Xk])\leq n+k$.
\item\label{thclassdimval4} $\Kdim(\gR[\Xn])\leq 2n$.
\end{enumerate}
Moreover, without supposing $\gR$ to be integral, but supposing it to be nontrivial, Items~\ref{thclassdimval1}, \ref{thclassdimval3}, and~\ref{thclassdimval4} are still equivalent. 
\end{theorem}

The paper \citealt{KY2020}, which follows \citealt{KV2014}, proposes a new further  characterisation of the valuative dimension of a commutative ring, inspired by the constructive characterisation of Krull dimension given in \citealt{Lom06}. The characterisation in  \citealt{KY2020} is fully constructive.

\smallskip We are thus in possession of at least three possible constructive approaches to the valuative dimension of a commutative ring defined within classical mathematics: the one corresponding to Item~\textsl{\ref{thclassdimval1}} above, the one corresponding to Items~\textsl{\ref{thclassdimval3}} and~\textsl{\ref{thclassdimval4}}, and the one proposed by Kemper and Yengui. 

\smallskip
We propose in Section~\ref{secdival2} to recall the precise constructive definitions concerning these three approaches. 

We denote by $\Vdim(\gR)$ a constructive definition corresponding to the characterisation given in Item \textsl{1} of Theorem~\ref{thclassdimval}.

We denote by $\vdim(\gR)$ a constructive definition corresponding to the characterisation given in Item \textsl{3} of Theorem~\ref{thclassdimval}.

We denote by $\dimv(\gR)$ the constructive definition given by Kemper and Yengui.
  
These constructive definitions have already been shown to be equivalent to the classical definition within classical mathematics, at least in the case of an integral domain.  

In Section~\ref{secdival3} we prove constructively the equivalence of these three definitions in all generality.

\subsection{Basic constructive terminology}\label{subsectermcov}

A subset~\(P\) of a set~\(E\) is said to be \textsl{detachable} when the property~$x\in P$ is decidable for~\(x\in E\). In other words, the following rule is satisfied:

\Regles{\labu $\vd x\in P \MA{\tsbf{ or }} x\notin P$}

\noindent In order to describe this situation, it is therefore necessary to introduce both the membership predicate and the opposite predicate.

We say that a ring is \textsl{integral} (or that it is an \textsl{integral domain}) when any element is zero or regular, and that a ring is a \textsl{discrete field} when any element is zero or invertible. This does not exclude the trivial ring.

A ring is said to be \textsl{without zerodivisor} when the rule

\Regles{\labu $\,\, xy=0\Vd x=0 \MA{\tsbf{ or }} y=0$}

\noindent is satisfied. An integral ring has no zerodivisor. The converse, valid within classical mathematics, is not guaranteed constructively.\footnote{In constructive mathematics, \gui{or} has its intuitive meaning, i.e.\ one of the two properties is explicitly valid. From the fact that the ring is without zerodivisor, with an explicit \gui{or}, there is no constructive proof that any element is zero or regular, with an explicit \gui{or}.}

Some local versions of the notions of integral ring and ring without zerodivisor are discernible even within classical mathematics.

A ring~\(\gR\) is said to be \textsl{locally without zerodivisor} (or a \textsl{pf-ring}: ``principal ideals are flat'') when the following rule is satisfied:

\Regles{\labu $\,\, ab=0\Vd \,\Exists s,t\;(sa=0\tsbf{,}\, tb=0\tsbf{,}\, s+t=1)$}

\noindent
Then in~\(\gR[1/s]\) the element~$a$ is zero, and in~\(\gR[1/t]\) the element~\(b\) is zero.\footnote{In classical mathematics, elements of a ring can be seen as \gui{functions} defined on the Zariski spectrum.  Here we have two basic open sets $\rD(s)$ and $\rD(t)$ which cover the Zariski spectrum; on the first one $a=0$, on the second one $b=0$.}

A ring~\(\gR\) is said to be a \textsl{pp-ring} (``principal ideals are projective'') if the annihilator~\(\Ann_\gR(a)\) of any element~\(a\) is generated by a (necessarily unique) idempotent, denoted by~\(1-e_a\). We have
$\gR\simeq \gR[1/e_a]\times \aqo{\gR}{e_a}$.
In the ring~\(\gR[1/e_a]\), the element~\(a\) is regular; in~\(\aqo{\gR}{e_a}\), $a$ is zero.\footnote{In classical mathematics, we have a partition of the Zariski spectrum into two basic open sets $\rD(1-e_a)$ and $\rD(e_a)$; on the first one $a=0$, on the second one $a$ is regular.}
A pp-ring is locally without zerodivisor, but the converse does not hold.
Note that we have $e_{ab}=e_a e_b$, $e_aa=a$, and $e_0=0$.
Pp-rings have a purely equational definition. Suppose indeed that a commutative ring is endowed with a unary law~\(a\mapsto
\ci{a}\) which satisfies the following three axioms:
\begin{equation}\label{eqaqis}
\ci{a}\,a=a,\quad
\ci{(ab)}=\ci{a}\,\ci{b},\quad
\ci{0}=0\text.
\end{equation}
Then, for all~\(a\in\gR\), $\Ann_\gR(a)=\gen{1-\ci{a}}$ and $\ci{a}$
is idempotent, so that~\(\gR\) is a pp-ring.

\begin{lemma}[pp-ring splitting lemma]\label{thScindageQi}
Consider $n$ elements~\(x_1,\dots,x_n\) in a pp-ring~$\gR$.
There exists a fundamental system of orthogonal idempotents~\((e_j)\) of cardinal~$2^n$ such that in each of the
components~$\gR[1/e_j]$, each~$x_i$ is zero or regular.
\end{lemma}

The fact that a pp-ring can be systematically split into two components
leads to the following general method.
The essential difference with the previous splitting lemma
is that we do not know a priori the finite family of elements that
will cause the splitting.

\rdb
\mni {\bf Elementary local-global machinery No.~1.}\label{MethodQI}
{\sl Most algorithms that work with nontrivial integral rings can be modified to work with pp-rings by splitting the ring
into two components whenever the algorithm written for integral rings
uses the test
“{is this element zero or regular?}”. In the first component the element in question
is zero, in the second one it is regular}.

\medskip We say that an ideal is \textsl{prime} if it produces a quotient ring without zerodivisor. This does not exclude the ideal~\(\gen{1}\).
These conventions \citep[adopted in][]{CACM} do not use negation and avoid some constructively offensive case-by-case reasonings.

\subsection{Valuation rings}\label{subsubsecdomval}

 A \textsl{valuation ring}~\(\gV\) is a subring of a discrete field~\(\gK\)
satisfying the axiom

\Regles{\labu $\,\, xy=1 \Vd x\in \gV \MA{\tsbf{ or }} y\in \gV \quad (x,y\in\gK)$}

We then say that $\gV$ is a \textsl{valuation ring of the discrete field $\gK$} and that $(\gK,\gV)$ is a \textsl{valued field}.

A valuation ring is the same as a local Bézout domain, or as an integral ring whose divisibility group is totally ordered.

In a valued field $(\gK,\gV)$, we say that $x$ divides $y$ and we write $x\di y$ if there exists a $z\in\gV$ such that $xz=y$.
We denote by $\Gamma(\gV)$ (or by~$\Gamma$ if the context is clear) the group $\gK\eti\!/\gV\eti$ (noted additively), with the order relation~$\leq$ induced by the relation $\di$ on $\gK\eti$. We let $\Gamma_\infty=\Gamma\cup\so\infty$ (where $\infty$ is introduced as a maximal element). Under these conditions, the natural application $v\colon\gK\to\Gamma_\infty$ is called the \textsl{valuation} of the valued field. We have
\[
v(xy)=v(x)+v(y), \;\text{and }v(x+y)\geq \min(v(x),v(y)) \text{ with equality if } v(x)\neq v(y).
\]
We also have $\gV=\sotq{x\in\gK}{v(x)\geq 0}$ and 
$\gV\eti=\sotq{x\in\gK}{v(x)= 0}$.

\subsection{Dimension of a distributive lattice}\label{subsecdimtrdi}

In this paragraph, we explain the constructive definition for the dimension of a distributive lattice. A distributive lattice $\gT$ can be seen as the set of compact open subsets of a spectral space, which is called the dual space of $\gT$. The definition of this dimension agrees within classical mathematics with the dimension of the dual spectral space, which is also the maximal length of chains of prime ideals in $\gT$: see Theorem~\ref{th-dico-trdi-spec-dim1}.

\smallskip An \textsl{ideal}  $\fb $ of a distributive lattice $(\gT,\vi,\vu,0,1)$ is a
subset that satisfies the conditions
\begin{equation}\label{eqIdeal}
\left.
\begin{array}{rcl}
  & & 0 \in \fb  \\
x,y\in \fb & \implies  & x\vu y \in \fb  \\
x\in \fb ,\; z\in \gT& \implies  & x\vi z \in \fb\text.  \\
\end{array}
\right\}
\end{equation}
Let us denote by $\gT/(\fb=0)$ the quotient lattice obtained by forcing the elements of $\fb$ to be zero.
We can also define the ideals as the kernels of morphisms.

A \textsl{principal ideal} is an ideal generated by a single element $a$: it is denoted by $\dar a$ and
we have $\dar a=\sotq{x\in\gT}{x\leq a}$.
This ideal, endowed with the laws $\vi$ and $\vu$ from~$\gT$, is a distributive lattice
in which the maximal element is $a$. The canonical injection $\dar
a\rightarrow \gT$ is not a morphism of distributive lattices because
the image of $a$ is not equal to $1$. On the other hand,
the application $\gT\rightarrow \dar a,\;x\mapsto x\vi a$
is a surjective morphism which endows $\dar a$ with the
quotient structure of $\gT/(a=1)$.

The notion of \textsl{filter} is the opposite notion (i.e.\ obtained by reversing the order relation) to that of ideal.

The following rule, called \textsl{cut}, is particularly important
for distributive lattices:
\begin{equation}\label{coupure1}
  (x\vi a \leq  b) \vii  (a \leq x\vu  b)
 \implies  (a \leq  b).
\end{equation}

If $A\in\Pfe(\gT)$ (the set of finitely enumerated subsets of~$\gT$), we let
\[
  \ndsp \Vu A:=\Vu_{x\in A}x\qquad \text{and}\qquad \Vi A:=\Vi_{x\in A}x.
\]

We denote by $A \vdash_\gT B$ the relation defined as follows on the set $\Pfe(\gT)$:
\[
  A \vdash_\gT B \; \; \equidef\; \; \Vi  A\;\leq \;
\Vu  B.
\]

This relation satisfies the following axioms, in which we
write $x$ for $\{x\}$ and $A, B$ for~$A\cup B$:
\vspace{-.5em}
\[
  \arraycolsep3pt\begin{array}{rcrclll}
& & x &\vda& x &\; &(R) \\[1mm]
 \text{if } A \vda B & \text{then} & A,A' &\vda& B,B' &\; &(M) \\[1mm]
\text{if } (A,x \vda B) \text{ and } (A \vda B,x)
& \text{then} & A &\vda& B\text. &\;
&(T)
\end{array}
\]
The relation is said to be \textsl{reflexive}, \label{remotr} \textsl{monotone}, and
\textsl{transitive}.
The third axiom (transitivity) can be seen as a
generalisation of Rule~(\ref{coupure1}) and is also called cut.

\begin{definition}
\label{defEntrel}
For an arbitrary set $S$, a relation on $\Pfe(S)$ that is
reflexive, monotone, and transitive is
called an {\sl entailment relation}.
\end{definition}

The following theorem is fundamental. It states that the
three axioms of entailment relations are exactly what is needed for
the distributive-lattice interpretation to work.

\begin{theorem}[fundamental theorem of entailment relations, \citealt{Lor1951,CC00}] \label{thEntRel1}~ \
Let $S$ be a set with an entailment relation
$\vdash$ on $\Pfe(S)$. Consider the distributive lattice $\gT$ defined by
generators and relations as follows: the generators are the
elements of $S$ and the relations are
\[
   A\; \vdash_\gT \;  B
\]
each time $A\; \vdash \; B$.  Then, for all $A$,
 $B$ in $\Pfe(S)$, we have
\[
    {A\; \vdash_\gT \;  B}
    \implies  {A\; \vdash \;  B}\text.
\]
\end{theorem}

In classical mathematics, a \textsl{prime ideal} $\fp$ of a distributive lattice $\gT\neq \Un$ is an ideal whose
complement $\fv$ is a filter (which is then a \textsl{prime filter}).  We then have $\gT/({\fp=0},\allowbreak{\fv=1})\simeq\Deux$.  It
is the same to give a prime ideal of $\gT$ or a morphism of
distributive lattices $\gT\rightarrow \Deux$.

\begin{theorem}[dimension of a distributive lattice, see {\citealt{CL2003,Lom02,Lom2020}}, {\citealt[chapter XIII]{CACM}}] \label{th-dico-trdi-spec-dim1}
In classical mathematics, the following properties are 
equivalent for a nontrivial distributive lattice and for~$n\geq 0$.
\begin{enumerate}
\item The lattice is of dimension $\leq n$, i.e.\ by definition the length of any chain of prime ideals is $\leq n$.
\item For any $x\in\gT$, the quotient lattice $\gT/(x=0,I_x=0)$ is of dimension $\leq n-1$, where $I_x=\sotq{y}{\allowbreak x\vi y=0}$.\footnote{A lattice is said to be of dimension $-1$ if it is trivial, i.e.\ reduced to a point; this initialises the induction in Item~\textsl{2}. It is easy to check that a lattice is zero-dimensional, i.e.\ of dimension $\leq 0$, if, and only if, it is a Boolean algebra.}
\item
For any sequence $(x_0,\dots,x_n)$ in $\gT$ there exists a sequence $(y_0,\dots,y_n)$ that is \textup{complementary} in the following sense:
\begin{equation}\label{eqC2G}
\left.\arraycolsep3pt
\begin{array}{rcl}
1& \vda  &   y_n, x_n\\
 y_n,  x_n & \vda  &  y_{n -1}, x_{n -1}  \\
\vdots~~~~& \vdots  &~~~~  \vdots \\
  y_1, x_1& \vda  &  y_0, x_0  \\
y_0, x_0& \vda  & 0\text.
\end{array}
\right\}
\end{equation}
\end{enumerate}
\end{theorem}

For example, for $n=2$, the inequalities in~\eqref{eqC2G} correspond to the following diagram in~$\gT$.
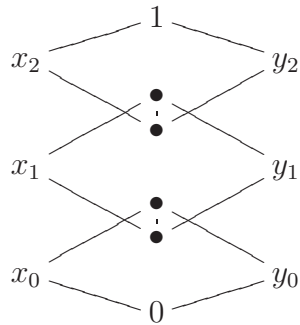
\[
  \SCo{1.5pt}{1.2cm}{x_0}{x_1}{x_2}{y_0}{y_1}{y_2}
\]

\noindent \textsl{Important definition and remark.} Items~\textsl{2} and~\textsl{3} are equivalent in constructive mathematics and are used for defining the dimension of $\gT$, denoted by $\Kdim(\gT)$ and called \textsl{Krull dimension} of the distributive lattice.

\noindent Note however that from a constructive point of view, only the assertion~``$\Kdim(\gT)\leq n$'' has been clearly defined. 
In order to settle ``$\Kdim(\gT)= n$'' it would be necessary to prove furthermore $\lnot (\Kdim(\gT)\leq n+1)$. Fortunately, most classical theorems have an assumption of the form $\Kdim(\gT)\leq n$.
\eoe

We also have the following constructive result which simplifies the use of Item~\textsl{3}.

\begin{lemma} \label{lem-trdi-spec-dim}
Let $S$ be a subset of a distributive lattice $\gT$ which generates $\gT$ as a distributive lattice.
In constructive mathematics, for $n\geq 0$, the following properties are 
equivalent.
\begin{enumerate}
\item
For any sequence $(x_0,\dots,x_n)$ in $\gT$ there exists a complementary sequence $(y_0,\allowbreak\dots,y_n)$ in $\gT$.
\item
For any sequence $(x_0,\dots,x_n)$ in $S$ there exists a complementary sequence $(y_0,\allowbreak\dots,y_n)$ in $\gT$.
\end{enumerate}

\end{lemma}

In the two final paragraphs of this section, we explain how  spectra of distributive lattices give rise to spectral spaces in abstract algebra. The duality between distributive lattices and spectral spaces has been established by \citet*{Sto37}. \citet*{Hoc1969} introduces the terminology of spectral spaces and has popularised their use, but neither cites \citealt*{Sto37} nor indicates the link between spectral spaces and distributive lattices.

\subsection{Krull dimension of a commutative ring}\label{subsecdimcomring}

We now recall the main idea of the constructive approach of  \cite{Joy71,Joy75} to the spectrum of a commutative ring.

If $\fa$ is an ideal of $\gR$, we denote by $\DR(\fa)$ (or $\rD(\fa)$ if the context is clear) the
nilradical of the ideal $\fA$:
\begin{equation} \label{eqZar}
\begin{array}{rclcl}
\DR(\fa)&  = & \sqrt[\gR]{\fa} &=&\sotq{x\in\gR}{\Ex m\in\NN\;\; 
x^m\in\fa}\text.
\end{array}
\end{equation}
When $\fa=\gen{x_1,\ldots ,x_n}$, we denote $\DR(\fa)$ by
$\DR(x_1,\ldots ,x_n)$.
If the context is clear, we also write $\wi{x}$ for $\DR(x)$.

By definition, the {\sl Zariski lattice} of $\gR$, denoted by $\ZarR$, is the set of  $\DR(x_1,\ldots ,x_n)$’s with the order relation of inclusion. The greatest lower bound and least upper bound are given by 
\[
\DR(\fa_1)\vi\DR(\fa_2)=\DR(\fa_1\fa_2)\quad \mathrm{and} \quad
\DR(\fa_1)\vu\DR(\fa_2)=\DR(\fa_1+\fa_2).
\]
The Zariski lattice of $\gR$
is a distributive lattice and $\DR(x_1,\ldots ,x_n)=
\wi{x_1}\vu\cdots \vu\wi{x_n}$.
The elements  $\wi{x}$ form a system of generators (stable under $\vi$) of $\ZarR$.

If $M$  is the  multiplicative monoid in $\gR$ generated by $U=\so{u_1,\dots,u_m}$ and $\fa=\gen{\an}$ is a finitely generated ideal,
 we have the equivalences
\begin{equation} \label {eqZarA}
\Vi_{i\in \so{1,\dots,m}} \wi {u_i} \;\leq_{\ZarR} \Vu_{j\in \so{1,\dots,n}} \wi{a_j}
\quad\Longleftrightarrow \;
\prod_{u\in  \so{1,\dots,m}} u_i  \in \sqrt[\gR]{\fa}
\quad\Longleftrightarrow \quad
M\,\cap\, \fa\neq \emptyset\text.
\end{equation}

This describes completely the distributive lattice $\ZarR$. 
In fact, the first formula in~(\ref{eqZarA}) defines an entailment relation on $\gR$ which gives another description of $\ZarR$.

We also get the following characterisation \citep[see][]{CC00,CL2003}.
\begin{proposition}[Joyal's definition of the spectrum of a commutative ring]
\label{propZar}~\\
 The lattice $\ZarR$ is
(up to unique isomorphism) the lattice generated by the symbols $\DR(a)$ for $a\in\gR$
subject to the following relations:
\[
  \DR(0_\gR) =0\text{, }
  \DR(1_\gR)= 1\text{, }
  \DR(x+y) \leq \DR(x)\vu\DR(y)\text{, }
  \DR(xy) = \DR(x)\vi \DR(y)\text.
\]
\end{proposition}

The construction $\gR\mapsto\ZarR$ yields a functor from the category of commutative rings
to  the category of distributive lattices.
Via this functor the projection $\gR\to\gR/\DR(0_\gR)$ gives an 
isomorphism
$\ZarR\to \Zar(\gR/\DR(0_\gR))$. We have $\ZarR=\Un$ if and only if  $1_\gR=0_\gR$.

In classical mathematics there is an easy proof that the dimension of the distributive lattice $\ZarR$ is the usual Krull dimension of the ring $\gR$.

\smallskip The Zariski lattice of a ring is the paradigmatic example of a distributive  lattice generated by a dynamical algebraic structure. We explain the general method in the following paragraph.

\subsection{Dynamical algebraic structures}

References: \citealt*{CLR01,Lom98,Lom06,Lom2020,BC2005,Coq2005}. The (finitary) dynamical algebraic structures are explicitly named in~\citealt{Lom98,Lom06}. In \citealt*{CLR01} they are implicit, but described in the form of their presentations.
They are also implicit in \citealt{Lom02}, and, last but not least, in \citealt*{D5} (D5), which has been an essential inspiration: one can compute safely in the algebraic closure of a discrete field even when it is not possible to construct this algebraic closure.
It is therefore appropriate to consider the algebraic closure as a dynamical algebraic structure à la D5 rather than as a usual algebraic structure: \textsl{lazy evaluation à la D5 provides a constructive semantics for the algebraic closure of a discrete field}.

A more detailed study can be found in \citealt{Lom-tgac} (in progress); 
see also \citealt{LM2022}.

A (finitary) dynamical theory $\sa{T}=(\cL,\cA)$ is a purely computational version, without logic, of a coherent theory. The language $\cL$ is given by a signature; the axioms (elements of $\cA$) are dynamical rules.

A \textsl{dynamical algebraic structure $\gD$ for a dynamical theory \sa{T}} is given by generators and relations:
$\gD=\big((G, R), \sa{T}\big)$. If $(G,R)$ is the positive diagram of a usual algebraic structure $\gR$, we denote this by $\gD=\sa{T}(\gR)$.

Let us consider a dynamical algebraic structure $\gD=\big((G,R),\sa{T}\big)$ for a dynamical theory $\sa{T}=(\cL,\cA)$.
Let $S$ be a set of closed atomic formulas of $\gD$. We define the entailment relation~\(\vdash\) on $S$ associated with $\gD$ as follows for $A_i$ and $B_j\in S$:

\vspace{-.5em}
\begin{equation} \label {eq2}
 A_1,\dots,A_n  \,\vdash B_1,\dots,B_m
   \equidef
    A_1\tsbf{,}\, \dots\tsbf{,}\,  A_n \Vdi{\gD} B_1\MA{\tsbf{ or }}\dots\MA{\tsbf{ or }} B_m\text.
\end{equation}
We denote by $\Zar(\gD,S)$ the distributive lattice generated by this entailment relation.
Intuitively, this lattice is the lattice of truth values of the formulas of $S$ in the dynamical algebraic structure $\gD$.

The  \textsl{(complete) Zariski lattice of a dynamical algebraic structure $\gD$} is defined by taking for~$S$ the set $\operatorname{Clat}(\gD)$ of all closed atomic formulas of~$\gD$. It is denoted by $\Zar(\gD,\sa T)$, or by $\Zar(\gD)$, or by a particular name corresponding to the theory~\sa{T}.

The dual spectral space is called the Zariski \textsl{spectrum of the dynamical algebraic structure~$\gD$}, or it can also be given a special name.

Finally, the \textsl{Krull dimension of $\gD$} is by definition equal to
$\Kdim(\Zar(\gD))$. The definition of $\Vdim(\gR)$ in the second constructive approach to valuative dimension (page~\pageref{subsecSpecVal}) is given according to this scheme.

\section{Three constructive definitions of the valuative dimension}\label{secdival2}

\subsection{Minimal pp-closure of a reduced ring}\label{subsecCqimiared}

We recall here some essential results given in \citealt{CACM}.
If the context is clear, we denote by $a\epr$ the annihilator of the element $a$ in $\gR$. We also use the notation $\fa\epr$ for the annihilator of an ideal~$\fa$.

\begin{lemma}\label{lem20MorRc}
Let $\gR$ be a reduced ring
and $a\in\gR$.
We define
\[
  \gR_{\so{a}}\eqdefi\gR\sur{a\epr}\times \gR\sur{({a\epr})\epr}
\]
and we denote by $\psi_a\colon\gR\to\gR_{\so{a}}$ the canonical homomorphism.
\begin{enumerate}
\item $\psi_a(a)\epr$ is generated by the idempotent $(\ov 0,\wi 1)$;
therefore $\psi_a(a)\epr=(\ov 1,\wi 0)\epr$.
\item $\psi_a$ is injective (we hence identify $\gR$ with a subring of $\gR_{\so{a}}$).
\item Let $\fb$ be an ideal in $\gR_{\so{a}}$; then the ideal $\psi_a^{-1}(\fb\epr)=\fb\epr\cap\gR$ is an annihilator in $\gR$.
\item The ring $\gR_{\so{a}}$ is reduced.
\end{enumerate}
\end{lemma}

\begin{lemma}\label{lem3MorRc}
Let $\gR$ be reduced
and $a,b\in\gR$. Then, with the notation of Lemma~\ref{lem20MorRc},
 the two rings $(\gR_{\so{a}})_{\so{b}}$ and $(\gR_{\so{b}})_{\so{a}}$ are canonically isomorphic.
\end{lemma}

\rdb
\noi{\it Note. } The case where $ab=0$ is typical: when we meet it, we should like to
split the ring into components where things are “{clear}”.
The previous construction then gives the three components
\[
\gR\sur{(ab\epr)\epr}, \; \gR\sur{(a\epr b)\epr}, \;  \gR\sur{(a\epr b\epr)\epr}.
\]
In the first one, $a$ is regular and $b=0$; in the second one,
$b$ is regular and $a=0$; in the third one, $a=b=0$.
\eoe

\begin{lemma}\label{lem2qi}
If $\gR\subseteq\gC$ with $\gC$ a pp-ring, the smallest pp-subring of~$\gC$ containing~$\gR$ is equal to $\gR[(e_a)_{a\in\gR}]$,
where $e_a$ is the idempotent of $\gC$ such that $\Ann_\gC(a)=\gen{1-e_a}_\gC$. More generally, if $\gR\subseteq\gB$
with $\gB$ reduced and if any element $a$ of $\gR$ has an annihilator in $\gB$
generated by an idempotent $1-e_a$,
then $\gR[(e_a)_{a\in\gR}]$ is a pp-subring of~$\gB$.
\end{lemma}

 Let us denote by $\Rred=\gR/\!\sqrt[\gR]{\gen{0}}$ the reduced ring generated by $\gR$.

\begin{thdef}[minimal pp-closure]\label{thAmin} ~ 
\\
Let $\gR$ be a reduced ring.
We can define a ring $\Rmin$ as a filtered colimit
 by iterating the basic construction of
 replacing~$\gE$ (the ring “{in progress}”, which contains $\gR$) by
 \[\gE_{\so{a}}\eqdefi\gE\sur{a\epr}\times \gE\sur{({a\epr})\epr}=\gE\sur{\Ann_\gE (a) }\times \gE\sur{\Ann_\gE(\Ann_\gE (a))}\]
 when $a$ runs through $\gR$.
\begin{enumerate}
\item This ring $\Rmin$ is a pp-ring, it contains $\gR$ and it is integral over $\gR$.
\item For any $x\in\Rmin$,
$x\epr\cap\gR$ is an annihilating ideal in $\gR$.
\end{enumerate}
This ring $\Rmin$ is called the \textup{minimal pp-closure of $\gR$}.
\\
When $\gR$ is not necessarily reduced,
  one takes $\gR\qim\!\eqdefi (\Rred)\qim$.
\end{thdef}

We denote by $\cP_n$ the set of finite subsets of $\so{1,\dots,n}$. Here is a description of each ring
obtained at a finite level of the $\Rmin$ construction.

\begin{lemma}\label{lem4MorRc}
Let $\gR$ be a reduced ring and $(\ua) = (\an)$ a sequence of~$n$ elements
of~$\gR$.  For $I\in\cP_n$, we denote by $\fa_I$ the ideal
\[\fa_I = \big(\Prod_{i\in I} \gen{a_i}\epr \Prod_{j\notin I} a_j\big)\epr
= \big(\gen{a_i\mathrel; i \in I}\epr \Prod_{j\notin I} a_j\big)\epr
.\]
Then $\Rmin$ contains the following ring: the product 
\[\gR_{\so\ua} = \Prod_{I\in\cP_n} \gR\sur{\fa_I}\]
of $2^n$ quotient rings of~$\gR$ (some possibly zero).
 \end{lemma}

We denote by $\BB(\gR)$ the Boolean algebra of idempotents of the ring $\gR$.

\begin{lemma}\label{fact2Amin}\begin{enumerate}\item[]
\item Let $\gR$ be a pp-ring.
\begin{enumerate}
\item $\Rmin=\gR$.
\item $\RX$ is a pp-ring, and $\BB(\gR)=\BB(\RX)$.
\end{enumerate}

\item For any ring $\gR$ we have a canonical isomorphism
\[\Rmin[\Xn]\simeq(\RXn)\qim.\]
\end{enumerate}
\end{lemma}

Our suggestion is that the proper generalisation of the notion of the field of fractions
of a ring $\gR$
is not the ring $\Frac\gR$ but the reduced zero-dimensional ring $\Frac\Rmin$.

\subsection{First constructive approach to the valuative dimension: \texorpdfstring{$\vdim$}{vdim}}\label{subsecdimvalcons}

In the book \citealt{CACM}, the authors use the notation $\Vdim(\gR)$ instead of $\vdim(\gR)$. In this article, we prefer to reserve this notation for the notion defined by Coquand. So we use $\vdim$ for the definition given in the book.

\smallskip The following definition is constructive because we have a constructive definition of Krull dimension. The subsequent Theorem~\ref{thValDim} is furthermore proved constructively. So, for the case of an integral domain, it provides a constructive proof for the equivalence of Items \textsl{2}, \textsl{3} and \textsl{4} in the classical Theorem~\ref{thclassdimval}.

\begin{definition}[according to {\citealt{CACM}}]\phantomsection\label{defiValdim}
\begin{enumerate}
\item[]
\item If $\gR$ is a pp-ring, the \textsl{valuative dimension}
is defined as follows. If $n\in\NN$ and $\gK=\Frac\gR$, we say that \textsl{the valuative
dimension of $\gR$ is less than or equal to $n$} and we write
$\vdim\gR\leq n$ if for any sequence $(\xm)$ in~$\gK$ we have
$\Kdim \gR[\xm] \leq n$. By convention $\vdim\gR=-1$ if $\gR$ is trivial.
\item In the general case we define “{$\vdim\gR\leq n$}” by “{$\vdim\Rmin\leq n$}”.\footnote{
Consider in this footnote $\Vdim\gR$ as defined in \citealt{Cah90} and recalled in the introduction. \citet{CACM} show that within classical mathematics $\vdim\gR\leq d\iff\Vdim \gR\leq d$. Hence this proves that $\Vdim\gR=\Vdim\Rmin$ within classical mathematics. This is a way to reduce the general case to the integral case without using prime ideals.}
\end{enumerate}
\end{definition}

For $n= 0,-1$, we have $\vdim\gR=n\iff\Kdim\gR =n$.

We note that, as with the Krull dimension, we have not really defined $\vdim\gR$ as an element of $\NN\cup\so\infty$: only the property “{$\vdim\gR\leq n$}” is well defined, constructively, for any integer $n\geq -1$.

\begin{theorem}[{\citealt[Theorem XIII-8.19]{CACM}}]\label{thValDim} 
The following equivalences are valid.
\begin{enumerate}
\item If $n\geq 1$ and $k\geq -1$, then
\begin{equation} \label {eq1thValDim}
\vdim\gR\leq k\iff\vdim\RXn\leq n+k.
\end{equation}
If $\gR$ is nontrivial and $\vdim\gR<\infty$, this means intuitively that we have an equality \[\vdim\RXn=n+\vdim\gR\text.\]
\item\label{thValDim2} If $n\geq 0$, then
\begin{equation} \label {eq2thValDim}
\vdim\gR\leq n\iff \Kdim \RXn \leq 2n.
\end{equation}
%
\item In the case where $\gR$ is a pp-ring, we also have the equivalence (for $n\geq 0$)
\[
  \vdim\gR\leq n\iff \Kdim \Rxn \leq n\text{ for all $x_1,\dots,x_n\in\Frac\gR$.}
\]
\end{enumerate}
\end{theorem}
Item~\textsl{\ref{thValDim2}} is correct within classical mathematics for the valuative dimension $\Vdim$ as defined in \citealt{Cah90}.

\subsection{Second approach: valuative lattice and valuative spectrum, \texorpdfstring{$\Vdim$}{Vdim}}\label{subsecSpecVal}

First of all let us recall that concerning the dimension of a distributive lattice or of the dual spectral space, there are three constructively equivalent approaches. The historically first one is the one defined by \citet{Joy75}, taken up in \citealt{CP2001} and \citealt{Coq2009}. The second one comes from the notion of potential chain of prime ideals of a commutative ring introduced by \citet{Lom06}. The third one is based on the notion of ``boundary'' ideal (or ``boundary'' filter) and gives rise to a definition by induction. The equivalence of these notions is essentially proven in \citealt{CL2003} and treated in great detail in \citealt{CL2001-2018}.

In the case of a subring $\gR$ of a discrete field $\gK$, \citet{Coq2009} defines the valuative lattice $\Val(\gK,\gR)$ as the distributive lattice which translates the valid rules for the predicate $\Vr$ (introduced below) in the dynamical theory $\sa{Val}(\gK,\gR)$ of valuation rings of the field~$\gK$ with subring~$\gR$.\footnote{Without however using as such the language of dynamical theories.}
Finally $\Val(\gR)$ is an abbreviated notation for the lattice $\Val(\Frac\gR,\gR)$.

The dynamical algebraic structure $\sa{Val}(\gK,\gR)$ can be described as the dynamical theory built on the signature
\[(\cdot=0,
\Vr(\cdot);0,1,\cdot+\cdot,-\cdot,\cdot\times\cdot,(a)_{a\in\gK} )
\]
in which the elements of $\gK$ are constants of the theory.\footnote{For the bare dynamical theory, we delete everything concerning $\gR$ and $\gK$.}

First there are the axioms of nontrivial discrete fields over the language of commutative rings.

\DeuxRegles{
\labu $\vd 0=0$
\labu $\,\,x=0\tsbf{,}\, y=0\vd x+y=0$
\labu $\,\,1=0 \vd \Bot $
}
{
\labu $\,\,x=0\vd xy=0$
\labu $\vd x=0\MA{\tsbf{ or }} \Exists y \,xy=1$
\item[\vspace{\fill}]
}

Then we add the diagram of $\gK$:

\DeuxRegles{
\labu $\vd 0_\gK= 0$
\labu $\vd 1_\gK= 1$
}
{
\labu $\vd a+b=c\quad (\text{if } a+b=_\gK c)$
\labu $\vd ab= c\quad (\text{if } ab=_\gK c)$
}

Finally there are the axioms describing the properties of the predicate $\Vr(x)$ which means that $x$ belongs to the potential valuation ring of the field $\gK$.
\DeuxRegles{
\labu $\vd \Vr(a)\quad (\text{if } a\in\gR)$

\labu $\,\,\Vr(x)\tsbf{,}\,\Vr(y)\vd \Vr(x+y)$
}
{
\labu $\,\,\Vr(x)\tsbf{,}\,\Vr(y)\vd \Vr(xy)$
\labu $\,\,xy=1\vd \Vr(x)\MA{\tsbf{ or }}\Vr(y)$
}

The lattice $\Val(\gK,\gR)$ is then the distributive lattice generated by the entailment relation
$\vdash_{\Val,\gK,\gR}$ on $\gK\eti$ defined by the equivalence

\vspace{-1em}
\begin{equation} \label {eq01}
\begin{aligned}
 x_1,\dots,x_n  &\,\,\vdash_{\Val,\gK,\gR}    y_1,\dots,y_m\\
   \equidef\Vr(x_1)\tsbf{,}\, \dots\tsbf{,}\, \Vr(x_n)&\vdi{\sA{Val}(\gK,\gR)} \Vr(y_1)\MA{\tsbf{ or }} \dots\MA{\tsbf{ or }} \Vr(y_m)\text.
 \end{aligned}
\end{equation}

Finally we denote $\Val(\Frac\gR,\gR)$ by $\Val(\gR)$.

In this framework, the valuative dimension of $\gR$, which we shall denote by $\Vdim\gR$, is defined as $\Kdim(\Val(\gR))$.

\citet[Theorem 8]{Coq2009} gives a Valuativstellensatz
in the form of the following equivalence (for $y_i$ and $x_j\in\gK\eti$):
\begin{multline*}
\Vr(x_1)\tsbf{,}\, \dots\tsbf{,}\, \Vr(x_n)\vdi{\sA{Val}(\gK,\gR)} \Vr(y_1)\MA{\tsbf{ or }} \dots\MA{\tsbf{ or }} \Vr(y_m)\\
\iff 1\in \gen{y_1^{-1},\dots,y_m^{-1}}\gR[x_1,\dots,x_n,y_1^{-1},\dots,y_m^{-1}]\text.
\end{multline*}
Chasing the denominators gives the following equivalent formulation:
\begin{equation} \label {eqVr}
y_1^{p_1}\cdots y_m^{p_m} =
\begin{aligned}[t]
  &Q(x_1,\dots,x_n,y_1,\dots,y_m)\text{ with }Q\in\gR[\Xn,\Ym]\\
  &\text{a polynomial whose monomials have a degree}\\
  &\text{in $(\Ym)$ strictly less than $(p_1,\dots,p_m)$.}
\end{aligned}
\end{equation}

What happens if we accept that some $x_j$'s or $y_i$'s are possibly zero? Since the rule $\vd \Vr(0)$ is valid, the rule

\Regles{\labu $\,\,\Vr(x_1)\tsbf{,}\, \dots\tsbf{,}\, \Vr(x_n)\vd \Vr(y_1)\MA{\tsbf{ or }} \dots\MA{\tsbf{ or }} \Vr(y_m)$}

\noindent is always satisfied if one of the $y_i$'s is zero; and if one of the~$x_j$'s is zero, it is equivalent to the same rule where we have deleted the corresponding~$\Vr(x_j)$ to the left of $\vd$.
The same facts can be observed for the Valuativstellensatz expressed in the form \pref{eqVr}: if one of the~$y_i$'s is zero, we take $Q=0$; if one of the $x_j$'s is zero, then it plays no part in \pref{eqVr}.  Thus the Valuativstellensatz in the form \pref{eqVr} is always valid, which avoids reasoning case by case.

The lattice $\Val\gR$ can thus be characterised as the distributive lattice generated by the entailment relation
$\vdash_{\Val\gR}$ on $\gK$ defined by the equivalence
\begin{equation} \label {eqVr2}
 x_1,\dots,x_n  \vdash_{\Val\gR}    y_1,\dots,y_m\equidef
 \begin{aligned}[t]
   &\exists p_1,\dots,p_m\geq0\ \exists Q\in \gR[\uX,\uY]\enskip y_1^{p_1}\cdots y_m^{p_m} =\\
   &Q(x_1,\dots,x_n,y_1,\dots,y_m)\text{, the monomials of $Q$}\\
   &\text{of degree in $\uY$ strictly less than $(p_1,\dots,p_m)$.}
\end{aligned}
\end{equation}

To make the calculations to come more readable, we introduce the predicate \[\Vp(x)\equidef \Exists u (ux=1\tsbf{,}\, \Vr(u)).\]
In other words, we add a predicate $\Vp(x)$ to the signature with the two axioms

\DeuxRegles{
\labu $\,\,ux=1\tsbf{,}\, \Vr(u)\vd \Vp(x)$
}
{
\labu $\,\,\Vp(x)\vd \Exists u \;(ux=1\tsbf{,}\, \Vr(u)) $
}

The new theory is a conservative extension of the former one. Moreover we have the following valid rules which allow to compute $\Vr$ from $\Vp$:

\DeuxRegles{
  \labu $\vd \Vr(0)$
  \labu \makebox[0pt][l]{$\,\,\Vr(x)\vd x=0\MA{\tsbf{ or }} \Exists u \;(ux=1\tsbf{,}\, \Vp(u))$}
}
{
\labu $\,\,ux=1\tsbf{,}\, \Vp(u)\vd \Vr(x)$}

We can read $\Vp(x)$ as $\Vr(1/x)$, where $\Vr(1/0)\vd \Bot$ (collapse of the theory). The predicate $\Vp(x)$ means that the element $x$ of $\gK$ is not a residually zero element of~$\gV$.
In particular, this predicate satisfies the following axioms in the dynamical theory considered:

\DeuxRegles{
\labu $\vd \Vp(a)\quad (\text{if } a\in\Rti)$
\labu $\,\,\Vp(x+y)\vd \Vp(x)\MA{\tsbf{ or }}\Vp(y)$
\labu $\,\,\Vp(x)\tsbf{,}\,\Vp(y)\vd \Vp(xy)$
}
{
\labu $\,\,\Vp(0)\vd \Bot$
\labu $\,\,xy=1\vd \Vp(x)\MA{\tsbf{ or }}\Vp(y)$
\item[\vspace{\fill}]
}

The Valuativstellensatz becomes, this time without restriction on $x_j$ and $y_i\in\gK$,
\begin{multline*}
  \Vp(y_1)\tsbf{,}\, \dots\tsbf{,}\, \Vp(y_m)\vd \Vp(x_1)\MA{\tsbf{ or }} \dots\MA{\tsbf{ or }} \Vp(x_n)\\
  \iff 1\in \gen{\xn}\gR[\xn,y_1^{-1},\dots,y_m^{-1}]\text,
\end{multline*}
where the right-hand side must be rewritten in the following form to avoid $0^{-1}$:
\begin{multline} \label {eqV'}
  \exists p_1,\dots,p_m\enskip\exists P_1,\dots,P_n\in\gR[\uX,\uY]\enskip
  y_1^{p_1}\dots y_m^{p_m}= x_1P_1(\ux,\uy)+\dots+x_nP_n(\ux,\uy)\text,\\
  \text{where the multiexponents of $\uY$ in the $P_j$'s are all $\leq(p_1,\dots,p_m)$.}
\end{multline}
If one of the $y_i$'s is zero, the equality is automatically satisfied by taking the $P_j$'s identically zero.  

This predicate $\Vp$ is the predicate denoted by~$\Nrn$ in the article \citealt{Lom2000}
which gives a very general Valuativstellensatz.

\begin{remark} \label{remV'}
 A lattice $\Valp\gR$ associated to the predicate $\Vp$ can be introduced. We can then show that $\Valp\gR$ is isomorphic to the lattice opposite to $\Val\gR$. This is because $x\mt x^{-1}$ is a bijection of $\gK\eti$ onto itself.
In this article, we only use the fact that the characterisations \pref{eqVr}
and \pref{eqV'} are equivalent.  \eoe
\end{remark}

\smallskip We shall introduce the valuative lattice of an arbitrary ring in our very last subsection on page~\pageref{subsec-vdim=Vdim2}.

\subsection{Third approach: graded monomial order, \texorpdfstring{$\dimv$}{dimv}}\label{subsecOmonRat}

We denote by $<_{\mathrm{lex}}$ the monomial order on $\ZZ^n$ corresponding to the lexicographic order.

A rational monomial graded order $<_M$ on $\ZZ^n$ or, equivalently, on the monomials of $\gR[X_1^{\pm1}, \allowbreak\dots, X_1^{\pm1}]$, is defined by means of a matrix $M\in \Mat_n(\NN)$ invertible in $\Mat_n(\QQ)$ with the coefficients of the first row all $>0$, as follows:
\[
(e_1,\dots,e_n)<_M (f_1,\dots,f_n) \equidef
M\cdot \cmatrix{e_1\\ \vdots \\ e_n} <_{\mathrm{lex}} \cmatrix{f_1\\ \vdots\\f_n}\text.
\]
When the coefficients of the first row of $M$ are equal to $1$, the monomial order $<_M$ is an order subordinate to total degree.
The monomial \textsl{graded lexicographic} order $<_{\mathrm{grlex}}$ is
the one defined by the matrix
\[\cmatrix{
1&1&&\dots&1\\
1&0&0&\dots&0\\
0&1&0& &0\\
\vdots & && &\vdots\\
0&0&\dots& 1&0
}\text.\]

\smallskip \citet{Lom06} characterises constructively the Krull dimension of an arbitrary ring~$\gR$  by the equivalence between $\Kdim\gR\leq n$ and the fact that for all $x_0,\dots,x_n\in\gR$ we have a polynomial $P\in\gR[X_0,\dots,X_n]$ which vanishes at $(x_0,\dots,x_n)$ and whose trailing coefficient for the lexicographic order is equal to $1$. For example, $\Kdim\gR\leq 1$ if and only if, for all $x_0,x_1\in\gR$ we can find an equality $0=x_0^{e_0}(x_1^{e_1}(1+c_1x_1)+c_0x_0)$: here the $c_i$'s are elements of $\gR$ or just as well elements of $\gR[x_0,x_1]$.

\smallskip \citet{KV2014} show, within classical mathematics, that for a noetherian ring one can characterise the Krull dimension in the same way using an arbitrary monomial order.

\smallskip \citet{KY2020} show, within classical mathematics, that for an arbitrary ring one can characterise the valuative dimension in the same way, provided one uses a graded rational monomial order instead of the lexicographic order.

In other words, we can paraphrase them by the following definition.
\begin{definition}\label{defKY1}
We say that $\dimv\gR\leq n$ if,  considering a graded rational monomial order
$<_M$, we have for all $x_0,\dots,x_n\in\gR$ a polynomial $P\in\gR[X_0,\dots,X_n]$ which vanishes at $(x_0,\dots,x_n)$ and whose trailing coefficient for the order $<_M$ is equal to $1$.
\end{definition}

The result in~\citealt{KY2020} is then the following.
\begin{theorem}\phantomsection\label{thKY1}
  \begin{enumerate}
  \item[]
\item This definition of $\dimv\gR\leq n$ does not depend on the matrix $M$ considered.
\item It is equivalent within classical mathematics to the fact that the valuative dimension of~$\gR$
is $\leq n$.
\end{enumerate}
\end{theorem}

By convention $\dimv\gR=-1$ means that the ring is trivial.

In fact, the proof of Theorem~\ref{thKY1} in \citealt{KY2020} is clearly constructive for the case of an integral ring~$\gR$: for all $n\geq 0$ we have constructively the equivalence
\begin{equation} \label {eqvdim-dimv}
\vdim\gR\leq n \quad \text{(Definition~\ref{defiValdim}.1)} \iff \dimv\gR\leq n\quad \text{(Definition~\ref{defKY1}).}
\end{equation}

\section{Constructive equivalence of the three constructive definitions}\label{secdival3}

\subsection{\texorpdfstring{$\vdim=\dimv$}{vdim=dimv}}\label{subsecvdim=dimv}

Let us state a first lemma which extends \pref{eqvdim-dimv} to the pp-case.

\begin{lemma} \label{lemvdim-dimvQi}
For a pp-ring $\gR$ we have the equivalence
\begin{equation} \label {eqvdim-dimvQi}
\vdim\gR\leq d \iff \dimv\gR\leq d\text.
\end{equation}
\end{lemma}
%
\begin{proof}
We take the constructive proof of \pref{eqvdim-dimv} given in the integral case and use the elementary local-global machinery No.~1.
\end{proof}

To extend the constructive equivalence \pref{eqvdim-dimv} from the case of an integral ring to the case of an arbitrary ring, given that in the general case we have defined $\vdim\gR\leq d$
as meaning $\vdim\Rmin\leq d$, and given Lemma~\ref{lemvdim-dimvQi}, we need only prove constructively the equivalence
\begin{equation} \label {eqdimv-dimv}
\dimv\gR\leq d \iff \dimv\Rmin\leq d\text.
\end{equation}
In particular, this implies constructively the analogue of Equivalence \pref{eq2thValDim}
for $\dimv$.

Equivalence \pref{eqdimv-dimv} reduces to the following two lemmas.
\begin{lemma} \label{lem-dimv-dimvred} We always have
\begin{equation} \label {eq-dimv-dimvred}
\dimv\gR\leq d \iff \dimv\Rred\leq d\text.
\end{equation}
\end{lemma}
\begin{proof}
The proof is left to the reader.
\end{proof}

\begin{lemma} \label{lem-dimv-dimvred2} Let $\gR$ be a reduced ring and $a\in\gR$. Then we have
\begin{equation} \label {eq-dimv-dimvred2}
\dimv\gR\leq d \iff \dimv\gR_{\so a}\leq d\text.
\end{equation}
\end{lemma}
%
\begin{proof}
  The implication $\implies$ is quite simple. On the one hand, the implication \[\dimv\gR\leq d \implies \dimv(\gR/\fa)\leq d\] is clear for any quotient $\gR/\fa$. And on the other hand we see that
  \[
    \dimv\gB\leq d\text{ and }\dimv\gC \leq d \implies \dimv(\gB\times \gC)\leq d.
  \]
Let us consider the converse implication. Consider $x_0,\dots,x_d\in\gR$. We have first of all a polynomial $P_1(X_0,\dots,X_d)\in\gR[\uX]$ with trailing coefficient equal to $c_1=1+y_1$ and
such that $P_1(x_0,\dots,x_d)=z_1$ with $y_1,z_1\in a\epr$. We have moreover a polynomial $P_2(\uX)$ with trailing coefficient $c_2=1+y_2$ and such that $P_2(\ux)=z_2$, with $y_2,z_2\in (a\epr)\epr$.
We then have \[y_1y_2=y_1z_2=z_1y_2=z_1z_2=0.\]
If we multiply $P_1$ and~$P_2$ by suitable monomials, we can assume that their trailing monomials coincide. Then $Q_2=P_2-y_2P_1$ has trailing coefficient $(1+y_2)-y_2\*{(1+y_1)}=1$ and satisfies $Q_2(\ux)=z_2$.
Similarly $Q_1=P_1-y_1P_2$ has trailing coefficient $(1+y_1)-y_1(1+y_2)=1$ and satisfies $Q_1(\ux)=z_1$. So the polynomial $Q_1Q_2$ suits.
\end{proof}
%

\subsection{\texorpdfstring{$\vdim=\Vdim$}{vdim=Vdim} in the integral case}\label{subsec-vdimv=Vdim1}

\begin{lemma} \label{lem-Vdim<=vdim}
For an integral ring $\gR$, and for an integer $n\geq -1$, the following applies:
\[\Vdim\gR\leq n\implies\vdim\gR\leq n\text.\]
\end{lemma}
%
\begin{proof}
The article \citealt{Coq2009} shows on the one hand that
\[\Vdim\gR\leq n\implies\Kdim\gR\leq n\]
\vspace{-1.5em}

\noindent and on the other hand that
\[\Vdim\gR\leq n\implies\Vdim\RX\leq n+1.\]
Therefore $\Vdim\gR\leq n\implies\Kdim\RXn\leq 2n$. In \citealt{CACM} it is shown that
$\Kdim\RXn\leq 2n\implies\vdim\gR\leq n$.
\end{proof}

\subsubsection{Proof of the converse inequality}

We shall prove $\vdim\leq n\implies\Vdim\leq n$ by constructing complementary sequences, relying on Lemma~\ref{lem-trdi-spec-dim}. In order to understand the proof, we shall treat first the cases \(n=2,3,4\). When \(n=2\), the complementary sequence consists of elements of the form~\(\Vp(y)\). In cases  \(n=3\) and \(n=4\) we find new ideas in order construct a complementary sequence in the distributive lattice generated by these elements.

\medskip \noindent \fbox{$\vdim\leq 2\implies\Vdim\leq 2$}

\smallskip \noindent Suppose that we have $x_0,x_1,x_2$ nonzero in the field of fractions $\gK$ of $\gR$. We are looking for $y_0,y_1,y_2 \in \gK$ (we know that
$y_2$ is zero) such that:
\begin{align}
\Vp(x_0) \vee \Vp(y_0) &=1\text,\label{eq:1}\\
\Vp(x_0) \wedge \Vp(y_0) &\leq \Vp(x_1) \vee \Vp(y_1)\text,\label{eq:2}\\
\Vp(x_1) \wedge \Vp(y_1) &\leq \Vp(x_2) \vee \Vp(y_2) =\Vp(x_2)\text{ (since $y_2=0$),}\label{eq:3}\\
\Vp(x_2) \wedge \Vp(y_2) &=0\text{ (this is guaranteed by $y_2=0$).}\notag
\end{align}
\eqref{eq:1} amounts to saying that $1=\gen{x_0,y_0}$ in $\gR[x_0,y_0]$.

\noindent \eqref{eq:2} amounts to saying that $1=\gen{x_1,y_1}$ in $\gR[x_0^{-1}, y_0^{-1}, x_1,y_1]$.

\noindent \eqref{eq:3} amounts to saying that $1=\gen{x_2}$ in $\gR[x_1^{-1}, y_1^{-1}, x_2]$.

\smallskip \noindent We use the fact that $\Kdim (\gR[x_0,x_1,x_2]) \leq 2$. We have a polynomial $P \in \gR[X_0,X_1,X_2]$ with
trailing coefficient $1$ (for the lexicographic monomial order with $X_2>X_1 > X_0$) which vanishes at $(x_0,x_1,x_2)$. Let $X_2^{n}X_1^mX_0^{\ell}$ be the trailing monomial of $P$.  Dividing $P(x_0,x_1,x_2)$ by $x_2^{n}x_1^mx_0^{\ell}$, we obtain an equality
\[
  1+ x_0f_0(x_0) + x_1 f_1(x_1,x_0^{\pm1}) + x_2 f_2(x_2,x_1^{\pm1},x_0^{\pm1})=0,
\]

\noindent where $f_0\in \gR[X_0]$, $f_1 \in \gR[X_1,X_0^{\pm1}]$ and $f_2 \in \gR[X_2,X_1^{\pm1},X_0^{\pm1}]$ (the $x_0f_0(x_0)$ comes from monomials of $P$ other than $M$ whose degree in $X_2$ is
equal to $n$ and whose degree in $X_1$ is
equal to $m$, $x_1 f_1(x_1,x_0^{\pm1})$ comes from the other monomials of $P$ whose degree in $X_2$ is
equal to $n$, while $x_2 f_2(x_2,x_1^{\pm1},x_0^{\pm1})$ comes from monomials of $P$ whose degree in $X_2$ is
$>n$). For some $r_0,r_1 \in \mathbb{N}$, we have:
\[
  1+ x_0f_0(x_0) + x_0^{r_0} \bigl( x_1 g_1(x_1,x_0^{-1}) + \frac{x_2}{x_1^{r_1}}g_2(x_2,x_1,x_0^{-1})\bigr)=0 ,
\]
where $g_1 \in \gR[X_1,X_0^{-1}]$ and $g_2 \in \gR[X_2,X_1,X_0^{-1}]$. We see that $y_0= \dfrac{1+ x_0f(x_0)}{x_0^{r_0}}$ and $y_1=\dfrac{x_2}{x_1^{r_1}}$ suit (note that $x_2=x_1^{r_1}y_1 \in \gR[ x_1,y_1] \subseteq \gR[x_0^{-1}, y_0^{-1}, x_1,y_1]$).

\smallskip \noindent Let us consider an example of collapse:
\[
  x_0 x_1^2x_2^2 + 2 x_0^2 x_1^2x_2^2 +3 x_1^4x_2^2 + x_0^2 x_1^5x_2^2 + 3x_2^3 +2 x_1x_2^3 + x_0^2 x_1^3x_2^4 = 0.
\]
Dividing by $x_0 x_1^2x_2^2$, we obtain an equality
\[
   1+ 2x_0^2 + x_0\Bigl(x_1 \Bigl(\frac{3x_1}{x_0^2} +x_1^2\Bigr) + \frac{x_2}{x_1^2}\Bigl(\frac{3}{x_0^2}+ \frac{2x_1}{x_0^2}+ x_2x_1^3\Bigr)\Bigr)=0.
\]
We see that $y_0= \dfrac{1+ 2x_0^2}{x_0}$ and $y_1=\dfrac{x_2}{x_1^2}$ suit.\medskip

 \medskip \noindent \fbox{$\vdim\leq 3\implies\Vdim\leq 3$}

 \smallskip \noindent Let $x_0,x_1,x_2,x_3$ be nonzero in the field of fractions $\gK$ of $\gR$. We look for $\mathfrak u_0,\mathfrak u_1,\mathfrak u_2,\mathfrak u_3$ in the distributive lattice generated by the $V'(x)$ which form a complementary sequence of $x_0,x_1,x_2,x_3$ (see the inequalities in~$(\ref{eqC2G})$). In other words: 
 {\small\[
 1=V'(x_0) \vee \mathfrak u_0,\, V'(x_0) \wedge \mathfrak u_0 \leq V'(x_1) \vee \mathfrak u_1,\ldots,\, V'(x_2) \wedge \mathfrak u_2 \leq V'(x_3) \vee \mathfrak u_3,\, V'(x_3) \wedge \mathfrak u_3=0.
 \]}

 \noindent We propose to find the $\mathfrak u_i$ in the form 
 \[
 \mathfrak u_0=\Vp(y_0),\mathfrak u_1=\Vp(y_1) \wedge \Vp(x_0),\,\mathfrak u_2= \Vp(y_2) ,\,\mathfrak u_3= 0
 \] with
 $y_0,y_1,y_2 \in \gK$ such that
 \begin{align}
 \Vp(x_0) \vee \Vp(y_0) &=1,\label{eq:31}\\
   \Vp(x_0) \wedge \Vp(y_0) &\leq \Vp(x_1) \vee \mathfrak u_1= \big( \Vp(x_1) \vee \Vp(y_1) \big) \wedge \big(\Vp(x_1) \vee \Vp(x_0)\big)\text{ or also}\notag\\
   \Vp(x_0) \wedge  \Vp(y_0) &\leq  \Vp(x_1) \vee  \Vp(y_1)\text{ and}\label{eq:32}\\
   \Vp(x_0) \wedge  \Vp(y_0) &\leq \Vp(x_1) \vee \Vp(x_0)\text,\label{eq:32bis}\\
 \Vp(x_1) \wedge \mathfrak u_1 &\leq \Vp(x_2) \vee \Vp(y_2)\text,\label{eq:33}\\
 \Vp(x_2) \wedge \Vp(y_2) &\leq \Vp(x_3) \vee 
 \mathfrak u_3=\Vp(x_3),\label{eq:34}\\
 \Vp(x_3) \wedge \mathfrak u_3 &=0\text{ (this is guaranteed by $\mathfrak u_3=0$).}\notag
 \end{align}
 \eqref{eq:31} amounts to saying that $1\in\gen{x_0,y_0}$ in $\gR[x_0,y_0]$.

 \noindent \eqref{eq:32} amounts to saying that $1\in\gen{x_1,y_1}$ in $\gR[x_0^{-1}, y_0^{-1}, x_1,y_1]$.

 \noindent 
 \eqref{eq:32bis} is always satisfied since \(\Vp(x_0) \wedge \Vp(y_0) \leq \Vp(x_0) \leq \Vp(x_1) \vee \Vp(x_0)\).\\
 \eqref{eq:33} amounts to saying that $1\in\gen{x_2,y_2}$ in $\gR[x_1^{-1}, y_1^{-1},x_0^{-1}, x_2,y_2]$.\\
 \eqref{eq:34} amounts to saying that $1\in\gen{x_3}$ in $\gR[x_2^{-1}, y_2^{-1}, x_3]$.

 \smallskip \noindent We use the fact that $\Kdim (\gR[x_0,x_1,x_2,x_3]) \leq 3$. \\
 We have a polynomial $P \in \gR[X_0,X_1,X_2,X_3]$ with
 trailing coefficient $1$ (for the lexicographic monomial order with $X_3>X_2>X_1 > X_0$) which vanishes at $(x_0,x_1,x_2,x_3)$. Let $X_3^{n}X_2^{m}X_1^pX_0^{q}$ be the trailing monomial of $P$.  Dividing $P(x_0,x_1,x_2,x_3)$ by $x_3^{m}x_2^nx_1^{p}x_0^{q}$, we obtain an equality
 \[  
 1+ x_0f_0(x_0) + x_1 f_1(x_1,x_0^{\pm1}) + x_2 f_2(x_2,x_1^{\pm1},x_0^{\pm1}) + x_3 f_3(x_3,x_2^{\pm1},x_1^{\pm1},x_0^{\pm1}) =0,
 \]
 where $f_0\in \gR[X_0]$, $f_1 \in \gR[X_1,X_0^{\pm1}]$, $f_2 \in \gR[X_2,X_1^{\pm1},X_0^{\pm1}]$ and $f_3 \in \gR[X_3,X_2^{\pm1},\allowbreak X_1^{\pm1},\allowbreak X_0^{\pm1}]$:
 \begin{itemize}
 \item $x_0f_0(x_0)$ comes from monomials of $P$ other than $M$ whose degree in $X_3$ is
 equal to $m$, the degree in $X_2$ is
 equal to $n$ and the degree in $X_1$ is
 equal to $p$;
 \item $x_1 f_1(x_1,x_0^{\pm1})$ comes from the other monomials of $P$ whose degree in $X_3$ is equal to $m$ and whose degree in $X_2$ is equal to $n$;
 \item $x_2 f_2(x_2,x_1^{\pm1},x_0^{\pm1})$ comes from the other monomials of $P$ whose degree in $X_3$ is equal to
 $m$;
 \item $x_3 f_3(x_3,x_2^{\pm1},x_1^{\pm1},x_0^{\pm1})$ comes from the monomials of $P$ whose degree in $X_3$ is $>m$.
 \end{itemize}
 For some $r_0,r_1 \in \mathbb{N}$, we have
 \begin{equation} \label {eq17}
 1+ x_0f_0(x_0) + x_0^{r_0} \bigl( x_1 g_1(x_1,x_0^{-1}) + {x_1^{r_1}}\bigl(x_2g_2(x_2,x_1^{-1},x_0^{-1})+\frac{x_3}{x_2^{r_2}}\,g_3(x_3,x_2,x_1^{-1},x_0^{-1})\bigr)\bigr)=0 ,
 \end{equation}
 where $g_1 \in \gR[X_1,X_0^{-1}]$, $g_2 \in \gR[X_2,X_1^{-1},X_0^{-1}]$, and $g_3\in \gR[X_3,X_2,X_1^{-1},X_0^{-1}]$.
 \\
 Let $y_0= \dfrac{1+x_0f_0(x_0)}{x_0^{r_0}}$,
 $y_1=\dfrac{y_0+x_1 g_1(x_1,x_0^{-1})}{x_1^{r_1}}$, $y_2=\dfrac{x_3}{x_2^{r_2}}$.
 \\
 Condition \eqref{eq:31}: $1= x_0^{r_0}y_0-x_0f_0(x_0)\in\gen{x_0,y_0}$ in $\gR[x_0,y_0]$ (even if $r_0=0$).
 \\
 Condition \eqref{eq:32}: $y_0=y_1x_1^{r_1}- x_1 g_1(x_1,x_0^{-1})$;
  if $y_0\neq 0$ we divide by $y_0$
 and we obtain $1\in\gen{x_1,y_1}$ in $\gR[x_1,y_1,x_0^{-1},y_0^{-1}]$.
 \\
 Condition \eqref{eq:34}: $ y_2x_2^{r_2}=x_3$ (and we divide by $y_2x_2^{r_2}$ if $y_2\neq 0$).
 \\
 Condition \eqref{eq:33}: Equality (\ref{eq17}) may be rewritten as follows:
 \[
 x_2g_2(x_2,x_1^{-1},x_0^{-1})+y_2g_3(y_2x_2^{r_2},x_2,x_1^{-1},x_0^{-1}) = -y_1\text.
 \]
 If $y_1\neq 0$ we divide by $y_1$.
 \\
 Moreover, we may also conclude if $y_0=0$ or $y_1=0$ (see the comment after Equality~(\ref{eqV'})).

 \medskip \noindent \fbox{$\vdim\leq 4\implies\Vdim\leq 4$}

 \smallskip \noindent Let $x_0,x_1,x_2,x_3,x_4$ be nonzero in the field of fractions $\gK$ of $\gR$. We look for $\mathfrak u_0,\mathfrak u_1,\mathfrak u_2,\allowbreak\mathfrak u_3,\allowbreak\mathfrak u_4$ in the distributive lattice generated by the $V'(x)$ which form a complementary sequence of $x_0,x_1,x_2,x_3,x_4$ (see the inequalities in $(\ref{eqC2G})$). In other words: 
 \[
 1=V'(x_0) \vee \mathfrak u_0, V'(x_0) \wedge \mathfrak u_0 \leq V'(x_1) \vee \mathfrak u_1,\dots, V'(x_3) \wedge \mathfrak u_3 \leq V'(x_4) \vee \mathfrak u_4, V'(x_4) \wedge \mathfrak u_4=0.
 \]
 We propose to find the $\mathfrak u_i$ in the form
 {\small
 \[\mathfrak u_0=\Vp(y_0),\mathfrak u_1=\Vp(y_1) \wedge \Vp(x_0),\mathfrak u_2= V'(y_2) \wedge V'(x_1) \wedge V'(x_0),\mathfrak u_3= \Vp(y_3),\mathfrak u_4= 0 
 \]}

 \noindent with
 $y_0,y_1,y_2,y_3 \in \gK$ such that
 \begin{align}
 V'(x_0) \vee V'(y_0) &=1\text,\label{eq:41}\\
 V'(x_0) \wedge V'(y_0) &\leq V'(x_1) \vee \mathfrak u_1\text{, or also}\notag\\
 V'(x_0) \wedge V'(y_0) &\leq V'(x_1) \vee V'(y_1)\text{,}\label{eq:42}\\
 V'(x_1) \wedge \mathfrak u_1 &\leq V'(x_2) \vee \mathfrak u_2\text{, or also}\notag\\
 V'(x_1) \wedge V'(y_1) \wedge V'(x_0) &\leq V'(x_2) \vee V'(y_2)\text{ and}\label{eq:43}\\
 V'(x_2) \wedge \mathfrak u_2 &\leq V'(x_3) \vee V'(y_3)\text,\label{eq:44}\\
 V'(x_3) \wedge V'(y_3) &\leq V'(x_4) \vee \mathfrak u_4= V'(x_4)\text,\label{eq:45}\\
 V'(x_4) \wedge \mathfrak u_4 &=0\text{ (this is guaranteed by $\mathfrak u_4=0$).}\notag
 \end{align}
 \eqref{eq:41} amounts to saying that $1\in\gen{x_0,y_0}$ in $\gR[x_0,y_0]$.\\
 \eqref{eq:42} amounts to saying that $1\in\gen{x_1,y_1}$ in $\gR[y_0^{-1},x_0^{-1} , x_1,y_1]$.\\
 \eqref{eq:43} amounts to saying that $1\in\gen{x_2,y_2}$ in $\gR[y_1^{-1},x_1^{-1},x_0^{-1}, x_2,y_2]$.\\
 \eqref{eq:44} amounts to saying that $1\in\gen{x_3,y_3}$ in $\gR[y_2^{-1},x_2^{-1},x_1^{-1},x_0^{-1}, x_3,y_3]$.\\
 \eqref{eq:45} amounts to saying that $1\in\gen{x_4}$ in $\gR[x_3^{-1}, y_3^{-1}, x_4]$.

 \medskip

 \noindent We use the fact that $\Kdim\gR[x_0,x_1,x_2,x_3,x_4] \leq 4$. We have a polynomial $P \in \gR[X_0,\allowbreak X_1,\allowbreak X_2,\allowbreak X_3,\allowbreak X_4]$ with
 trailing coefficient $1$ (for the lexicographic monomial order with $X_4>X_3>X_2>X_1 > X_0$) which vanishes at $(x_0,x_1,x_2,x_3,x_4)$. \\
 Let $X_4^{\ell}X_3^{n}X_2^{m}X_1^pX_0^{q}$ be the trailing monomial of $P$.  Dividing $P(x_0,x_1,x_2,x_3,x_4)$ by $x_4^{\ell}x_3^{m}x_2^nx_1^{p}x_0^{q}$, we obtain an equality
 \begin{multline*}
   1+ x_0f_0(x_0) + x_1 f_1(x_1,x_0^{\pm1}) + x_2 f_2(x_2,x_1^{\pm1},x_0^{\pm1}) +x_3 f_3(x_3,x_2^{\pm1},x_1^{\pm1},x_0^{\pm1})\\
   + x_4 f_4(x_4,x_3^{\pm1},x_2^{\pm1},x_1^{\pm1},x_0^{\pm1})=0\text,
 \end{multline*}
 where $f_0 \in \gR[X_0]$, $f_1 \in \gR[X_1,X_0^{\pm1}]$, $f_2 \in \gR[X_2,X_1^{\pm1},X_0^{\pm1}]$, $f_3 \in \gR[X_3,X_2^{\pm1},\allowbreak X_1^{\pm1},\allowbreak X_0^{\pm1}]$, and $f_4 \in \gR[X_4,X_3^{\pm1},X_2^{\pm1},X_1^{\pm1},X_0^{\pm1}]$:
 \begin{itemize}
 \item $x_0f_0(x_0)$ comes from monomials of $P$ other than $M$ whose degree in $X_4$ is
 equal to~$\ell$, the degree in $X_3$ is
 equal to $m$, the degree in $X_2$ is
 equal to $n$ and the degree in $X_1$ is
 equal to $p$,
 \item $x_1 f_1(x_1,x_0^{\pm1})$ comes from the other monomials of $P$ whose degree in $X_4$ is
 equal to~$\ell$, the degree in $X_3$ is
 equal to $m$, the degree in $X_2$ is
 equal to $n$,
 \item $x_2 f_2(x_2,x_1^{\pm1},x_0^{\pm1})$ comes from the other monomials of $P$ whose degree in $X_4$ is
 equal to~$\ell$ and the degree in $X_3$ is
 equal to $m$,

 \item $x_3 f_3(x_3,x_2^{\pm1},x_1^{\pm1},x_0^{\pm1})$ comes from the monomials of $P$ whose degree in $X_4$ is
 equal to~$\ell$,
 \item $x_4 f_4(x_4,x_3^{\pm1},x_2^{\pm1},x_1^{\pm1},x_0^{\pm1})$ comes from the monomials of $P$ whose degree in $X_4$ is
 $>\ell$.
 \end{itemize}
 For some $r_0,r_1,r_2,r_3 \in \mathbb{N}$, we have
 \begin{multline}
   1+ x_0f_0(x_0) + x_0^{r_0}\Bigl( x_1 g_1(x_1,x_0^{-1}) + {x_1^{r_1}}\Bigl(x_2g_2(x_2,x_1^{-1},x_0^{-1})\\
   +x_2^{r_2}\bigl(x_3g_3(x_3,x_2^{-1},x_1^{-1},x_0^{-1})+\frac{x_4}{x_3^{r_3}}\,g_4(x_4,x_3,x_2^{-1},x_1^{-1},x_0^{-1})\bigr)\Bigr)\Bigr)=0,\label{eq:18}
 \end{multline}
 where $g_1 \in \gR[X_1,X_0^{-1}]$, $g_2 \in \gR[X_2,X_1^{-1},X_0^{-1}]$, $g_3\in \gR[X_3,X_2^{-1},X_1^{-1},X_0^{-1}]$, and $g_4\in\gR[X_4,X_3,\allowbreak X_2^{-1},X_1^{-1},X_0^{-1}]$.
 \\
 Let $y_0= \dfrac{1+x_0f_0(x_0)}{x_0^{r_0}}$,
 $y_1=\dfrac{y_0+x_1 g_1(x_1,x_0^{-1})}{x_1^{r_1}}$, $y_2=\dfrac{y_1+x_2 g_2(x_2,x_1^{-1},x_0^{-1})}{x_2^{r_2}}$, $y_3=\dfrac{x_4}{x_3^{r_3}}$.
 \\
 Condition \eqref{eq:41}: $1= x_0^{r_0}y_0-x_0f_0(x_0)\in\gen{x_0,y_0}$ in $\gR[x_0,y_0]$ (even if $r_0=0$).
 \\
 Condition \eqref{eq:42}: $y_0=y_1x_1^{r_1}- x_1 g_1(x_1,x_0^{-1})$;
  if $y_0\neq 0$ we divide by $y_0$
 and we obtain $1\in\gen{x_1,y_1}$ in $\gR[x_1,y_1,x_0^{-1},y_0^{-1}]$.
 \\
 Condition \eqref{eq:43}: $y_1=y_2x_2^{r_2}- x_2 g_2(x_2,x_1^{-1},x_0^{-1})$;
  if $y_1\neq 0$ we divide by $y_1$
 and we obtain $1\in\gen{x_2,y_2}$ in $\gR[y_1^{-1},x_1^{-1},x_0^{-1}, x_2,y_2]$.
 \\
 Condition \eqref{eq:44}: Equality~\eqref{eq:18} reads
 \[
 x_3g_3(x_3,x_2^{-1},x_1^{-1},x_0^{-1})+y_3g_4(y_3x_3^{r_3},x_3,x_2^{-1},x_1^{-1},x_0^{-1})  = -y_2
 \]
 If $y_2\neq 0$ we divide by $y_2$ and get that $1\in\gen{x_3,y_3}$ in $\gR[y_2^{-1},x_2^{-1},x_1^{-1},x_0^{-1}, x_3,y_3]$.
 \\
 Condition \eqref{eq:45}: $ y_3x_3^{r_3}=x_4$ (and we divide by $y_3x_3^{r_3}$ if $y_3\neq 0$).

\medskip \noindent \fbox{$\vdim\leq n\implies\Vdim\leq n$}

\smallskip \noindent Let $x_0,\dots,x_n$ be nonzero in the field of fractions $\gK$ of $\gR$. 
We look for $\mathfrak u_0,\mathfrak u_1,\dots,\mathfrak u_n$ in the distributive lattice generated by the $V'(x)$ which form a complementary sequence of $x_0,x_1,\dots,x_n$. In other words: 
{\small\[
1=V'(x_0) \vee \mathfrak u_0,\, V'(x_0) \wedge \mathfrak u_0 \leq V'(x_1) \vee \mathfrak u_1,\ldots,\,V'(x_{n-1}) \wedge \mathfrak u_{n-1} \leq V'(x_n) \vee \mathfrak u_n ,\, V'(x_n) \wedge \mathfrak u_n=0.
\]
}

\noindent We use the fact that $\Kdim\gR[x_0,\ldots,x_n] \leq n$. We have a polynomial $P \in \gR[X_0,\ldots,X_n]$ with
trailing coefficient $1$ (for the lexicographic monomial order with $X_n>X_{n-1}> \cdots > X_0$) which vanishes at $(x_0,\ldots,x_n)$. Let $X_n^{q_n} \cdots X_0^{q_0}$ be the trailing monomial of~$P$.  By dividing $P(x_0,\ldots,x_n)$ by $x_n^{q_n} \cdots x_0^{q_0}$, we obtain an equality
\[1+ x_0f_0(x_0) + x_1 f_1(x_1,x_0^{\pm1}) + \ldots +
 x_n f_n(x_n,x_{n-1}^{\pm1},\ldots,x_0^{\pm1})=0,
\]
where $f_0\in \gR[X_0], \,f_1 \in \gR[X_1,X_0^{\pm1}],\ldots, \, f_n \in \gR[X_n,X_{n-1}^{\pm1},\ldots,X_0^{\pm1}]$.

\medskip

\noindent For some $r_0,\ldots,r_n \in \mathbb{N}$, we have
\begin{multline*}
1+ x_0f_0(x_0) + x_0^{r_0} \bigg( x_1 g_1(x_1,x_0^{-1}) +\cdots + x_{n-3}^{r_{n-3}} \Big( x_{n-2}g_{n-2}(x_{n-2},x_{n-3}^{-1},\ldots,x_0^{-1})\\
+ x_{n-2}^{r_{n-2}}\big(x_{n-1}g_{n-1}(x_{n-1},x_{n-2}^{-1},\ldots,x_0^{-1})+
\frac{x_n}{x_{n-1}^{r_{n-1}}}\,g_n(x_n,x_{n-1},x_{n-2}^{-1},\ldots x_0^{-1}) \cdots \big) \Big) \bigg) =0\text, 
\end{multline*}
where \(g_1 \in \gR[X_1,X_0^{-1}]\), \dots, \(g_{n-1}\in\gR[X_{n-1},X_{n-2}^{-1},\dots,X_0^r{-1}]\),
\(g_n\in\gR[X_n,\allowbreak X_{n-1}, \allowbreak X_{n-2}^{-1},\allowbreak \dots,X_0^{-1}]\).

\begin{sloppypar}
\noindent Let \(y_0= \dfrac{1+x_0f_0(x_0)}{x_0^{r_0}}\),
\(y_1=\dfrac{y_0+x_1 g_1(x_1,x_0^{-1})}{x_1^{r_1}}\), \dots, \(y_{n-2}=\dfrac{y_{n-3}+x_{n-2} g_{n-2}(x_{n-2},x_{n-3}^{-1},\dots,x_0^{-1})}{x_{n-2}^{r_{n-2}}}\), \(y_{n-1}=\dfrac{x_n}{x_{n-1}^{r_{n-1}}}\).

\smallskip \noindent It is then sufficient to take \(\mathfrak u_0=V'(y_0)\), \(\mathfrak u_1=V'(y_1)\wedge V'(x_0)\), \(\mathfrak u_2=V'(y_2)\wedge\allowbreak V'(x_1) \wedge\allowbreak  V'(x_0)\), \dots,
\(\mathfrak u_{n-2}=V'(y_{n-2})\wedge V'(x_{n-3}) \wedge \cdots \wedge V'(x_0)\), \(\mathfrak u_{n-1}=V'(y_{n-1})\), and $\mathfrak u_n=0$.

\end{sloppypar}

\subsection{\texorpdfstring{$\vdim=\Vdim$}{vdim=Vdim} in the general case}\label{subsec-vdim=Vdim2}

We note that when $\gR$ is a pp-ring, the ring $\Frac(\gR)$
is a reduced zero-dimensional ring. Moreover a discrete field is a reduced zero-dimensional ring in which any idempotent is equal to $0$ or $1$.

We then start with the following remark which follows from the elementary local-global machinery No.~1.
\begin{lemma} \label{Vdimqi}
  \begin{sloppypar}
Let $\gR$ be a pp-ring. Let us define $\Val(\gR):=\Val(\Frac(\gR),\gR)$ and $\Vdim(\gR):=\Kdim(\Val(\gR))$ as in \pref{eq01} by replacing $\gK$ by $\Frac(\gR)$.
Then we obtain the equality $\vdim(\gR)=\Vdim(\gR)$ as in the integral case.
\end{sloppypar}
\end{lemma}

In the remainder of this paragraph, we do not give proofs:
we refer to the general study \citealt{LM2022}.

\smallskip
We define a dynamical theory \sa{val} as follows. We consider the signature
\[
  (\cdot\di\cdot\mkern1mu; \cdot+\cdot, \cdot\times\cdot,-\,\cdot,0,1)\text.
\]
The axioms are as follows.

\DeuxRegles
{
\lab{Col$_{\val}$} $\,\, 0\di 1 \vd \Bot$ \quad (collapse)
\Lab{av1} $\vd 1 \di -1$
\Lab{av2} $\,\,a \di  b \Vd ac \di  bc$
\Lab{Av1} $\,\,a \di  b \tsbf{,}\, b \di  c \Vd a \di  c$
}
{
\Lab{Av2} $\,\,a \di  b \tsbf{,}\, a \di  c \Vd a \di  b + c$
\Lab{AV1} $\vd a \di  b \MA{\tsbf{ or }} b\di a$
\Lab{AV2} $\,\,ax \di  bx  \Vd a \di  b \MA{\tsbf{ or }} 0 \di x$
\item[\vspace{\fill}]}

The equality $x=0$ is defined as an abbreviation for $x\di 0$.

\begin{definition} \label{defivalkK}~
\begin{enumerate}
\item If $\gR$ is a commutative ring, the dynamical algebraic structure $\sa{val}(\gR)$ is obtained by taking as presentation the positive diagram of the ring $\gR$.
\item If $\gk\subseteq \gR$ are two rings,\footnote{We use $\gk$ as notation for the small ring in order to invoke the intuition provided by the frequent situation where $\gk$ is a discrete field.
} or more generally if $\varphi\colon\gk\to\gR$ is an algebra, we denote by $\sa{val}(\gR,\gk)$ the dynamical algebraic structure whose presentation is given by
\begin{itemize}
\item the positive diagram of~$\gR$ as a commutative ring;
\item the axioms
$\vd 1 \di \varphi(x)$ for the elements $x$ of $\gk$.
%
\end{itemize}

%
%
\end{enumerate}
\end{definition}

The two dynamical algebraic structures $\sa{val}(\gR)$ and $\sa{val}(\gR,\gZ)$, where $\gZ$ is the smallest subring of $\gR$, are canonically isomorphic.

\begin{definition} \label{defivalkK0}
Let $\gk$ be a subring of a ring $\gR$.
We define the distributive lattice $\val(\gR,\gk)$ through the entailment relation 
 $\vdash_{\gR,\gk,\mathrm{val}}$ on the set $\gR\times\gR$ given by the following equivalence. 
\vspace{-.8em}
\begin{equation} \label {eqvalkK}
\begin{aligned} 
 (a_1,b_1),\dots,(a_n,b_n) &\,\vdash_{\gR,\gk,\mathrm{val}}\, (c_1,d_1),\dots,(c_m,d_m)\\ 
\equidef\quad  a_1\di b_1\tsbf,\dots\tsbf,\, a_n\di b_n &\Vdi{\sA{val}(\gR,\gk)} c_1\di d_1\tsbf{ or } \dots\tsbf{ or } c_m\di d_m\text.
 \end{aligned}
\end{equation}
\end{definition}

\smallskip
The following result can be proved.
\begin{lemma} \label{lemvalVal}
Let $\gR$ be an integral ring with field of fractions~$\gK$. We have natural morphisms $\Val(\gK,\gR)\to\val(\gK,\gR)$ and $\val(\gR,\gR)\to\val(\gK,\gR)$. These are isomorphisms.
\end{lemma}

The following definition is therefore reasonable. We shall see that it coincides with the one given in Lemma~\ref{Vdimqi} in the case of pp-rings.

\begin{definition} \label{defiVdimgnle}
Let $\gR$ be an arbitrary commutative ring. We define $\Vdim(\gR)\leq n$ by
$\Kdim(\val(\gR,\allowbreak\gR))\leq n$.
\end{definition}

This dimension coincides with the one already defined when $\gR$ is integral. But it is not in general equal to the dimension of the lattice $\val(\Frac(\gR),\gR)$.

From our point of view, this means that $\Frac(\Rmin)$ is a much better substitute than $\Frac(\gR)$ for the field of fractions when $\gR$
is not an integral ring. In fact $\Rmin$ coincides with $\Frac(\gR)$
only for pp-rings.

Finally, we can prove the following theorem.

\begin{theorem} \label{thvdimAmin}
The distributive lattices \[\val(\gR,\gR) \text{ and }
 \val(\Rmin,\Rmin)\simeq\val(\Frac(\Rmin),\Rmin)\]
have the same Krull dimension.
\end{theorem}

With Lemma~\ref{Vdimqi} and Definitions~\ref{defivalkK} and \ref{defiVdimgnle} this completes the work.

\smallskip \noindent \textsl{Note}. We could have defined $\Vdim(\gR)=\Vdim(\Rmin)$ directly without using the 
theory \sa{val}, but this would have been an ad hoc definition, because $\Vdim(\gR)$ has no direct natural definition for an arbitrary ring if we use
only the theory \sa{Val}.

\normalsize
\endgroup
\stopcontents[english]

\clearpage
\newpage
\thispagestyle{empty}

~

\clearpage
\newpage

\renewcommand\thepage{F\arabic{page}}\renewcommand\theHsection{F\arabic{section}}
\begingroup

\clearpage
\setcounter{page}{1} 

\selectlanguage{french}
\def\frenchproofname{\textsl{Démonstration}}

\FrenchFootnotes



\theoremstyle{plain}
\newtheorem{ftheorem}{Théorème}[section]
\newtheorem{fthdef}[ftheorem]{Théorème et définition}
\newtheorem{fpstf}[ftheorem]{Positivstellensatz formel}
\newtheorem{fpst}[ftheorem]{Positivstellensatz}
\newtheorem{flemma}[ftheorem]{Lemme}
\newtheorem{fcorollary}[ftheorem]{Corolaire}
\newtheorem{fconjecture}[ftheorem]{Conjecture}
\newtheorem{fproposition}[ftheorem]{Proposition}
\newtheorem{fpbu}[ftheorem]{Problème universel}
\newtheorem{fprpta}[ftheorem]{Propriétés attendues}
\newtheorem{fpropdef}[ftheorem]{Proposition et définition}
\newtheorem{ffact}[ftheorem]{Fait}
\newtheorem{fplcc}[ftheorem]{Principe local-global concret}

\newtheorem{ftheoremc}[ftheorem]{Th\'{e}or\`{e}me\etoz}
\newtheorem{flemmac}[ftheorem]{Lemme\etoz}
\newtheorem{fcorollaryc}[ftheorem]{Corolaire\etoz}
\newtheorem{fproprietec}[ftheorem]{Propri\'{e}t\'{e}\etoz}
\newtheorem{fpropositionc}[ftheorem]{Proposition\etoz}
\newtheorem{ffactc}[ftheorem]{Fait\etoz}
\newtheorem{fvalsatz}[ftheorem]{\vst}

\theoremstyle{definition}
\newtheorem{fconvention}[ftheorem]{Convention}
\newtheorem{fdefinition}[ftheorem]{Définition}
\newtheorem{fdfni}[ftheorem]{Définition informelle}
\newtheorem{fdefinitions}[ftheorem]{Définitions}
\newtheorem{fnotation}[ftheorem]{Notation}
\newtheorem{fproblem}[ftheorem]{Problème}
\newtheorem{fquestion}[ftheorem]{Question}
\newtheorem{fquestions}[ftheorem]{Questions}
\newtheorem{fcontext}[ftheorem]{Contexte}
\newtheorem{fdefinitionc}[ftheorem]{Définition\etoz}
\newtheorem{fdefinota}[ftheorem]{Définition et notation}

\theoremstyle{remark}
\newtheorem{fexample}[ftheorem]{Exemple}
\newtheorem{fexamples}[ftheorem]{Exemples}
\newtheorem{fnotes}[ftheorem]{Notes}
\newtheorem{fremark}[ftheorem]{Remarque}
\newtheorem{fremarks}[ftheorem]{Remarques}
\newtheorem{fcomment}[ftheorem]{Commentaire}



\newcommand{\vou}{\MA{\tsbf{ ou }}}
\newcommand{\Vou}{\MA{\tsbf{OU}}}
\newcommand \EXists[1] {\tsbf{Introduire }{#1}\tsbf{ tel que }\,}
\newcommand \Atcl {\mathrm{Atcl}}
\newcommand \Tcl {\mathrm{Tcl}}
\newcommand \Atclv {\mathrm{Atclv}}
\newcommand \vet {{\tsbf{,}}\,}
\newcommand \tcl {\Tcl}

\newcommand \num {{n$^{\mathrm{ o}}$}}

\renewcommand\equidef{\buildrel{\textrm{déf}}\over{\;\Longleftrightarrow\;}}
\renewcommand\eqdef{\buildrel{\textrm{déf}}\over {\;=\;}}
\renewcommand\eqdefi{\eqdef}



\newcounter{MF}
\newcommand\stMF{\stepcounter{MF}}

\newcommand{\lec}{\stMF\ifodd\value{MF}lecteur\xspace\else 
lectrice\xspace\fi}

\newcommand{\lecs}{\stMF\ifodd\value{MF}lecteurs\xspace\else 
lectrices\xspace\fi}

\newcommand{\alec}{\stMF\ifodd\value{MF}au lecteur\xspace\else%
à la lectrice\xspace\fi}

\newcommand{\dlec}{\stMF\ifodd\value{MF}du lecteur\xspace\else%
de la lectrice\xspace\fi}

\newcommand{\llec}{\stMF\ifodd\value{MF}le lecteur\xspace\else la lectrice\xspace\fi}

\newcommand{\Llec}{\stMF\ifodd\value{MF}Le lecteur\xspace\else La lectrice\xspace\fi}

\newcommand{\lui}{\ifodd\value{MF}lui\xspace\else
elle\xspace\fi}

\newcommand{\celui}{\ifodd\value{MF}celui\xspace\else
celle\xspace\fi}

\newcommand{\ceux}{\ifodd\value{MF}ceux\xspace\else
celles\xspace\fi}

\newcommand{\er}{\ifodd\value{MF}er\xspace\else
ère\xspace\fi}

\newcommand{\eux}{\ifodd\value{MF}eux\xspace\else
elles\xspace\fi}

\newcommand{\eUx}{\ifodd\value{MF}eux\xspace\else
euse\xspace\fi}

\newcommand{\eUX}{\ifodd\value{MF}eux\xspace\else
euses\xspace\fi}

\newcommand{\leux}{\ifodd\value{MF}leux\xspace\else
leuse\xspace\fi}

\newcommand{\il}{\ifodd\value{MF}il\xspace\else
elle\xspace\fi}

\newcommand{\ien}{\ifodd\value{MF}ien\xspace\else
ienne\xspace\fi}

\newcommand{\iens}{\ifodd\value{MF}iens\xspace\else
iennes\xspace\fi}

\newcommand{\e}{\ifodd\value{MF}\xspace \else e\xspace\fi}

\newcommand{\n}{\ifodd\value{MF}n\xspace\else nne\xspace\fi}

\makeatletter
\newcommand{\la}{\@ifstar{\ifodd\value{MF}le\else
la\fi}{\stMF\ifodd\value{MF}le\else la\fi}}
\makeatother

\newcommand \rem{\rdb
\noi{\sl Remarque. }}

\newcommand \REM[1]{\rdb
\noi{\sl Remarque#1. }}

\newcommand \rems{\rdb
\noi{\sl Remarques. }}

\newcommand \exl{\rdb
\noi{\bf Exemple. }}

\newcommand \EXL[1]{\rdb
\noi{\bf Exemple: #1. }}

\newcommand \exls{\rdb
\noi{\bf Exemples. }}

\newcommand \thref[1] {théorème~\ref{#1}}

\newcommand\oge{\leavevmode\raise.3ex\hbox{$\scriptscriptstyle\langle\!\langle\,$}}
\newcommand\feg{\leavevmode\raise.3ex\hbox{$\scriptscriptstyle\,\rangle\!\rangle$}}

\newcommand\gui[1]{\oge{#1}\feg}

\newcommand \facile{\begin{proof}
La démonstration est laissée \alec.
\end{proof}
}

\newcommand\comm{\rdb
\noi{\sl Commentaire. }}

\newcommand\COM[1]{\rdb
\noi{\sl Commentaire #1. }}

\newcommand\comms{\rdb
\noi{\sl Commentaires. }}

\newcommand\Pb{\rdb
\noi{\bf Problème. }}

\newcommand\eoq{\hbox{}\nobreak
\vrule width 1.4mm height 1.4mm depth 0mm}

\newcommand \Cad {C'est-à-dire\xspace}
\newcommand \recu {récur\-rence\xspace}
\newcommand \hdr {hypo\-thèse de \recu}
\newcommand \cad {c'est-à-dire\xspace}
\newcommand \cade {c'est-à-dire en\-co\-re\xspace}
\newcommand \ssi {si, et seu\-lement si,\xspace}
\newcommand \cnes {con\-di\-tion néces\-sai\-re et suf\-fi\-san\-te\xspace}
\newcommand \spdg {sans per\-te de géné\-ra\-lité\xspace}
\newcommand \Spdg {Sans per\-te de géné\-ra\-lité\xspace}

\newcommand \Propeq {Les pro\-pri\-é\-tés sui\-van\-tes sont 
équi\-va\-len\-tes.}
\newcommand \propeq {les pro\-pri\-é\-tés sui\-van\-tes sont 
équi\-va\-len\-tes.}



\newcommand \Ali {appli\-ca\-tion $\gA$-\lin}
\newcommand \Alis {appli\-ca\-tions $\gA$-\lins}

\newcommand \Kli {appli\-ca\-tion $\gK$-\lin}
\newcommand \Klis {appli\-ca\-tions $\gK$-\lins}

\newcommand \Bli {appli\-ca\-tion $\gB$-\lin}
\newcommand \Blis {appli\-ca\-tions $\gB$-\lins}

\newcommand \Cli {appli\-ca\-tion $\gC$-\lin}
\newcommand \Clis {appli\-ca\-tions $\gC$-\lins}

\newcommand \ac{algé\-bri\-quement clos\xspace}  

\newcommand \acl {an\-neau \icl}
\newcommand \acls {an\-neaux \icl}

\newcommand \adp {an\-neau de Pr\"u\-fer\xspace}
\newcommand \adps {an\-neaux de Pr\"u\-fer\xspace}

\newcommand \adpc {\adp \coh}
\newcommand \adpcs {\adps \cohs}

\newcommand \adu {\alg de décom\-po\-sition univer\-selle\xspace}
\newcommand \adus {\algs de décom\-po\-sition univer\-selle\xspace}

\newcommand \adv {anneau de valuation\xspace}
\newcommand \advs {anneaux de valuation\xspace}

\newcommand \advl {anneau \dvla} 
\newcommand \advls {anneaux \dvlas} 

\newcommand \Afr {Anneau \frl}
\newcommand \Afrs {Anneaux \frls}
\newcommand \afr {anneau \frl}
\newcommand \afrs {anneaux \frls}

\newcommand \aFr {\hyperref[theorieAfr]{anneau \frl}}
\newcommand \aFrs {\hyperref[theorieAfr]{anneau \frls}}

\newcommand \afrr {\afr réduit\xspace}
\newcommand \afrrs {\afrs réduits\xspace}
\newcommand \Afrrs {\Afrs réduits\xspace}

\newcommand \afrvr {\afr avec \ravs}
\newcommand \aFrvr {\hyperref[theorieAfrrv]{\afrvr}}
\newcommand \afrvrs {\afrs avec \ravs}

\newcommand \aftr {anneau réticulé \ftm réel\xspace}
\newcommand \aftrs {anneaux réticulés \ftm réels\xspace}

\newcommand \aG {\alg galoisienne\xspace}
\newcommand \aGs {\algs galoisiennes\xspace}

\newcommand \agB {\alg de Boole\xspace}
\newcommand \agBs {\algs de Boole\xspace}

\newcommand \agH {\alg de Heyting\xspace}
\newcommand \agHs {\algs de Heyting\xspace}

\newcommand \agq{algé\-bri\-que\xspace}
\newcommand \agqs{algé\-bri\-ques\xspace}

\newcommand \agqt{algé\-bri\-que\-ment\xspace}

\newcommand \aKr {anneau de Krull\xspace}
\newcommand \aKrs {anneaux de Krull\xspace}

\newcommand \alg {algè\-bre\xspace}
\newcommand \algs {algè\-bres\xspace}

\newcommand \algo{algo\-rithme\xspace}
\newcommand \algos{algo\-rithmes\xspace}

\newcommand \algq{al\-go\-rith\-mi\-que\xspace}
\newcommand \algqs{al\-go\-rith\-mi\-ques\xspace}

\newcommand \ali {appli\-ca\-tion \lin}
\newcommand \alis {appli\-ca\-tions \lins}

\newcommand \alo {an\-neau lo\-cal\xspace}
\newcommand \alos {an\-neaux lo\-caux\xspace}

\newcommand \algb {an\-neau \lgb}
\newcommand \algbs {an\-neaux \lgbs}

\newcommand \alrd {\alo \dcd}
\newcommand \alrds {\alos \dcds}

\newcommand \anar {anneau \ari}
\newcommand \anars {anneaux \aris}

\newcommand \anor {an\-neau nor\-mal\xspace}
\newcommand \anors {an\-neaux nor\-maux\xspace}

\newcommand \apf {\alg \pf}
\newcommand \apfs {\algs \pf}

\newcommand \apG {\alg pré\-galoisienne\xspace}
\newcommand \apGs {\algs pré\-galoisiennes\xspace}

\newcommand \arc {anneau réel clos\xspace}
\newcommand \aRc {\hyperref[theorieArc]{\arc}}
\newcommand \arcs {anneaux réels clos\xspace}

\newcommand \ari{arith\-mé\-tique\xspace}  
\newcommand \aris{arith\-mé\-tiques\xspace}  

\newcommand \Asr {Anneau \str}
\newcommand \Asrs {Anneaux \strs}
%
%
%
%
%
%
%
%
%
%
\newcommand \bif {borne infé\-rieure\xspace} %
%
%
%
%
%
\newcommand \cara{carac\-té\-ris\-tique\xspace}  
\newcommand \caras{carac\-té\-ris\-tiques\xspace}  
\newcommand \carn{carac\-té\-ri\-sation\xspace}  
\newcommand \carns{carac\-té\-ri\-sations\xspace}  
\newcommand \carar{carac\-té\-riser\xspace}
\newcommand \cdac{\cdi \ac}  
\newcommand \cdacs{\cdis \ac}  

\newcommand \cdi{corps discret\xspace}
\newcommand \cdis{corps discrets\xspace}
\newcommand \cdf{corps de fractions\xspace}
\newcommand \coe {coef\-fi\-cient\xspace}
\newcommand \coes {coef\-fi\-cients\xspace}
\newcommand \coh {co\-hé\-rent\xspace}
\newcommand \cohe {co\-hé\-rente\xspace}
\newcommand \cohs {co\-hé\-rents\xspace}
\newcommand \com {co\-ma\-xi\-maux\xspace}
\newcommand \come {co\-ma\-xi\-males\xspace}
\newcommand \cop {complé\-men\-taire\xspace}
\newcommand \cops {complé\-men\-taires\xspace}
\newcommand \crcd {corps réel clos discret\xspace}
\newcommand \crcds {corps réels clos discrets\xspace}
\newcommand \cqi {clô\-ture \qi} 
\newcommand \cqis {clô\-tures \qis} 
\newcommand \cvd{corps valué discret\xspace}
\newcommand \cvds{corps valués discrets\xspace}
%
%
%
%
\newcommand \dcd {rési\-duel\-lement dis\-cret\xspace}
\newcommand \dcds {rési\-duel\-lement dis\-crets\xspace}
\newcommand \ddk {dimension de~Krull\xspace}
\newcommand \ddi {de dimension infé\-rieure ou égale à~}
\newcommand \demo{démon\-stra\-tion\xspace}     
\newcommand \demos{démon\-stra\-tions\xspace}     
\newcommand \deno{déno\-mi\-nateur\xspace}     
\newcommand \denos{déno\-mi\-nateurs\xspace}   
\newcommand \dfn{défi\-nition\xspace}  
\newcommand \dfns{défi\-nitions\xspace}  
\newcommand \dij{disjonc\-tive\xspace}  
\newcommand \dijs{disjonc\-tives\xspace}  
\newcommand \dok {domaine de Dedekind\xspace}
\newcommand \doks {domaines de Dedekind\xspace}
\newcommand \dvz {di\-viseur de zéro\xspace}
\newcommand \dvzs {di\-viseurs de zéro\xspace}
\newcommand \dve {divisi\-bi\-lité\xspace}
%
%
%
%
%
\newcommand \eco {\elts \com}
\newcommand \egmt {éga\-lement\xspace}
\newcommand \egt {éga\-lité\xspace}
\newcommand \egts {éga\-lités\xspace}
\newcommand \elr{élé\-men\-taire\xspace}  
\newcommand \elrs{élé\-men\-taires\xspace}  
\newcommand \elt{élé\-ment\xspace}  
\newcommand \elts{élé\-ments\xspace}  
\newcommand \entrel {rela\-tion impli\-ca\-tive\xspace}
\newcommand \entrels {rela\-tions impli\-ca\-tives\xspace}
\newcommand \eqn {équation\xspace}  
\newcommand \eqns {équations\xspace}  
\newcommand \eqv {équi\-valent\xspace}  
\newcommand \eqve {équi\-va\-lente\xspace}  
\newcommand \eqvs {équi\-valents\xspace}  
\newcommand \eqves {équi\-val\-entes\xspace}  
\newcommand \eqvc {équi\-va\-lence\xspace}  
\newcommand \eqvcs {équi\-va\-lences\xspace}  
\newcommand \esid {essen\-tiel\-lement iden\-tique\xspace}  
\newcommand \esids {essen\-tiel\-lement iden\-tiques\xspace}  
\newcommand \eseq {essen\-tiel\-lement \eqve}  
\newcommand \eseqs {essen\-tiel\-lement \eqves}  
\newcommand\evc{es\-pa\-ce vec\-to\-riel\xspace} 
\newcommand\evcs{es\-pa\-ces vec\-to\-riels\xspace} 
%
%
%
%
%
%
%
%
%
%
%
%
%
%
%
%
\newcommand\gmq{géomé\-trique\xspace}  
\newcommand\gmqs{géomé\-triques\xspace}  
\newcommand\gnl{géné\-ral\xspace}  
\newcommand\gnle{géné\-rale\xspace}  
\newcommand\gnls{géné\-raux\xspace}  
\newcommand\gnles{géné\-rales\xspace}  
\newcommand\gnlt{géné\-ra\-lement\xspace}  
\newcommand\gnn{géné\-ra\-li\-sa\-tion\xspace}  
\newcommand\gnns{géné\-ra\-li\-sa\-tions\xspace}  
\newcommand\gnq {géné\-rique\xspace}  
\newcommand\gnqs {géné\-riques\xspace}  
\newcommand \gnt{géné\-ra\-lité\xspace}
\newcommand \gnts{géné\-ra\-lités\xspace}
\newcommand\gtr{géné\-ra\-teur\xspace}  
\newcommand\gtrs{géné\-ra\-teurs\xspace}  
%
%
\newcommand \homo {homo\-mor\-phisme\xspace}
\newcommand \homos {homo\-mor\-phismes\xspace}
\newcommand \icl {inté\-gra\-lement clos\xspace}
\newcommand \icle {inté\-gra\-lement close\xspace}
\newcommand \icles {inté\-gra\-lement closes\xspace}
\newcommand \id {idéal\xspace}
\newcommand \ids {idéaux\xspace}
\newcommand \idep {idéal pre\-mier\xspace}
\newcommand \ideps {idéaux pre\-miers\xspace}
\newcommand \idm {idem\-po\-tent\xspace}
\newcommand \idms {idem\-po\-tents\xspace}
\newcommand \idme {idem\-po\-tente\xspace}
\newcommand \idmes {idem\-po\-tentes\xspace}
\newcommand \idp {idéal prin\-ci\-pal\xspace}
\newcommand \idps {idé\-aux prin\-ci\-paux\xspace}
\newcommand \inteq {intui\-ti\-vement \eqve}
\newcommand \inteqs {intui\-ti\-vement \eqves}
\newcommand \iso {iso\-mor\-phisme\xspace}
\newcommand \isos {iso\-mor\-phismes\xspace}
\newcommand \itf {idéal \tf}
\newcommand \itfs {idé\-aux \tf}
\newcommand \iv {inversible\xspace}
\newcommand \ivs {inversibles\xspace}
%
%
%
%
%
\newcommand \lgb {local-global\xspace}
\newcommand \lgbe {locale-globale\xspace}
\newcommand \lgbs {local-globals\xspace}
\newcommand \lot {loca\-lement\xspace}
\newcommand \lon {loca\-li\-sation\xspace}
\newcommand \lons {loca\-li\-sations\xspace}
\newcommand \lsdz {\lot \sdz}
\newcommand \mo {mo\-no\"{\i}de\xspace}
\newcommand \mos {mo\-no\"{\i}des\xspace}
\newcommand \moco {\mos \com}
\newcommand \mom {mo\-nô\-me\xspace}
\newcommand \moms {mo\-nô\-mes\xspace}
%
%
%
%
%
%
%
%
%
%
\newcommand \ncr{néces\-saire\xspace}       
\newcommand \ncrs{néces\-saires\xspace}       
\newcommand \ncrt{néces\-sai\-rement\xspace}       
\newcommand \ndz {régu\-lier\xspace}
\newcommand \ndzs {régu\-liers\xspace}
\newcommand \nocos {\noes \cohs}
\newcommand \noe {noethé\-rien\xspace}
\newcommand \noes {noethé\-riens\xspace}
\newcommand \noee {noethé\-rienne\xspace}
\newcommand \noees {noethé\-riennes\xspace}
%
%
%
%
\newcommand \oqc {ouvert \qc}
\newcommand \oqcs {ouverts \qcs}
\newcommand \ort{or\-tho\-go\-nal\xspace}  
\newcommand \orte{or\-tho\-go\-na\-le\xspace}  
\newcommand \orts{or\-tho\-go\-naux\xspace}  
\newcommand \ortes{or\-tho\-go\-na\-les\xspace}  
%
%
%
%
%
%
%
\newcommand \peq {purement équa\-tion\-nelle\xspace}
\newcommand \peqs {purement équa\-tion\-nelles\xspace}
\newcommand \plg {principe \lgb}
\newcommand \plgs {principes \lgbs}
\newcommand \plga {\plg abstrait\xspace}
\newcommand \plgas {\plgs abstraits\xspace}
\newcommand \plgc {\plg concret\xspace}
\newcommand \plgcs {\plgs concrets\xspace}
\newcommand \pn {présen\-ta\-tion\xspace}
\newcommand \pns {présen\-ta\-tions\xspace}
\newcommand \Pol {Poly\-nôme\xspace}
\newcommand \Pols {Poly\-nômes\xspace}
\newcommand \pol {poly\-nôme\xspace}
\newcommand \pols {poly\-nômes\xspace}
\newcommand \prmt {préci\-sé\-ment\xspace}
\newcommand \Prmt {Préci\-sé\-ment\xspace}
\newcommand \prt {pro\-pri\-été\xspace}
\newcommand \prts {pro\-pri\-étés\xspace}
\newcommand \ptf {\pro \tf}
\newcommand \ptfs {\pros \tf}
\newcommand \qc {quasi-compact\xspace}
\newcommand \qcs {quasi-compacts\xspace}
\newcommand \qi {quasi intègre\xspace}
\newcommand \qis {quasi intègres\xspace}
\newcommand \qiv {quasi inverse\xspace}
\newcommand \qivs {quasi inverses\xspace}
%
%
\newcommand \ralg {règle \agq}
\newcommand \ralgs {règles \agqs}
\newcommand \rdi {\rde inté\-grale\xspace}
\newcommand \rdis {\rdes inté\-grales\xspace}
\newcommand \rdij {règle \dij}
\newcommand \rdijs {règles \dijs}
\newcommand \rdt {rési\-duel\-lement\xspace}
\newcommand \rdv {rela\-tion de \dve valuative\xspace}
\newcommand \rdvs {rela\-tions de \dve valuative\xspace}
\newcommand \rdy {règle dyna\-mique\xspace}
\newcommand \rdys {règles dyna\-miques\xspace}
\newcommand \red {règle directe\xspace}
\newcommand \reds {règles directes\xspace}
%
\newcommand \rex {règle exis\-ten\-tielle simple\xspace}
\newcommand \rexs {règles exis\-ten\-tielles simples\xspace}
\newcommand \rexri {règle exis\-ten\-tielle rigide\xspace}
\newcommand \rexris {règles exis\-ten\-tielles rigides\xspace}
\newcommand \rsim {règle de simplification\xspace}
\newcommand \rsims {règles de simplification\xspace}
\newcommand \rmq {\rcm de quotients\xspace} 
\newcommand \rvq {\rcv par quotients\xspace} 
\newcommand \rmqs {\rcms de quotients\xspace} %
\newcommand \rvqs {\rcvs par quotients\xspace} %
%
%
%
\newcommand \sad {\salg dynamique\xspace}
\newcommand \sads {\salgs dynamiques\xspace}
\newcommand \salg {structure \agq}
\newcommand \salgs {structures \agqs}
\newcommand \sdz {sans \dvz}
\newcommand \sfio {sys\-tème fondamental d'\idms ortho\-gonaux\xspace}
\newcommand \sfios {sys\-tèmes fondamentaux d'\idms ortho\-gonaux\xspace}
\newcommand \sgr {\sys \gtr}
\newcommand \sgrs {\syss \gtrs}

\newcommand \sps {espace spectral\xspace}
\newcommand \spss {espaces spectraux\xspace}
\newcommand \sys {sys\-tème\xspace}
\newcommand \syss {sys\-tèmes\xspace}
%
\newcommand \talg {théorie \agq}
\newcommand \talgs {théories \agqs}
\newcommand \tco {théorie cohé\-rente\xspace}
\newcommand \tcos {théories cohé\-rentes\xspace}
\newcommand \tdij {théorie \dij}
\newcommand \tdijs {théories \dijs}
\newcommand \tdy {théorie dyna\-mique\xspace}
\newcommand \tdys {théories dyna\-miques\xspace}
\newcommand \tel {théorie exis\-ten\-tielle\xspace}
\newcommand \tels {théories exis\-ten\-tielles\xspace}
\newcommand \telri {théorie exis\-ten\-tiel\-lement rigide\xspace}
\newcommand \telris {théories exis\-ten\-tiel\-lement rigides\xspace}
\newcommand \tf {de type fini\xspace}
\newcommand \tfo {théorie formelle\xspace}
\newcommand \tfos {théorie formelles\xspace}
\newcommand \tgm {théorie \gmq}
\newcommand \tgms {théories \gmqs}
\newcommand \Tho {Théo\-rème\xspace}
\newcommand \tho {théo\-rème\xspace}
\newcommand \thos {théo\-rèmes\xspace}
\newcommand \tpe {théorie \peq}
\newcommand \tpes {théories \peqs}
\newcommand \trdi {treil\-lis dis\-tri\-bu\-tif\xspace}
\newcommand \trdis {treil\-lis dis\-tri\-bu\-tifs\xspace}
%
%
%
%
\newcommand \uvle {uni\-ver\-selle\xspace}
\newcommand \uvles {uni\-ver\-selles\xspace}
%
%
%
%
\newcommand \valn {valuation\xspace}
\newcommand \valns {valuations\xspace}
\newcommand \vst {Valuativ\-stel\-lensatz\xspace}
\newcommand \vsts {Valuativ\-stel\-lensätze\xspace}
\newcommand \zed {zéro-di\-men\-sionnel\xspace}
\newcommand \zede {zéro-di\-men\-sion\-nelle\xspace}
\newcommand \zedr {\zed réduit\xspace}
\newcommand \zedrs {\zeds réduits\xspace}
%
%


\newcommand \cof {cons\-truc\-tif\xspace}
\newcommand \cofs {cons\-truc\-tifs\xspace}

\newcommand \cov {cons\-truc\-tive\xspace}
\newcommand \covs {cons\-truc\-tives\xspace}

\newcommand \coma {\maths\covs}
\newcommand \clama {\maths clas\-siques\xspace}

\renewcommand \cot {cons\-truc\-ti\-vement\xspace}

\newcommand \matn {mathé\-ma\-ticien\xspace}
\newcommand \matne {mathé\-ma\-ti\-cienne\xspace}
\newcommand \matns {mathé\-ma\-ticiens\xspace}
\newcommand \matnes {mathé\-ma\-ti\-ciennes\xspace}

\newcommand \maths {mathé\-ma\-tiques\xspace}
\newcommand \mathe {mathé\-ma\-tique\xspace}

\newcommand \prco {démons\-tration \cov}
\newcommand \prcos {démons\-trations \covs}

\thickmuskip = 7mu plus 2mu

\pagestyle{headings}
\patchcmd{\sectionmark}{\MakeUppercase}{}{}{}

\stMF
\startcontents[french]

\rdb
\label{beginfrench}

\title{Dimension valuative, points de vue constructifs}

\sibil{\author{Henri Lombardi, Stefan Neuwirth, Ihsen Yengui}
}
\sinotbil{
\author{Henri Lombardi
\thanks{Université de Franche-Comté, Laboratoire de mathématiques de Besançon, UMR, CNRS 6623,  25000
Besançon, France. {\tt henri.lombardi@univ-fcomte.fr}}
\and
Stefan Neuwirth
\thanks{Université de Franche-Comté, Laboratoire de mathématiques de Besançon, UMR CNRS 6623, 16~route de Gray, 25000
Besançon, France. {\tt stefan.neuwirth@univ-fcomte.fr}}
 \and  Ihsen Yengui
\thanks{Département de mathématiques,  Faculté des sciences de Sfax, Université de Sfax, 3000 Sfax, Tunisia. {\tt ihsen.yengui@fss.rnu.tn}}
}}

\maketitle

\begin{abstract}
Il existe plusieurs caractérisations de la dimension valuative d'un anneau commutatif. Des versions constructives de cette dimension ont été proposées, et démontrées équivalentes à la notion classique dans le cadre des mathématiques classiques. Contrairement aux version classiques, les versions constructives ont un contenu calculatoire clair et elles peuvent être utilisées dans les cas usuels d'anneaux commutatifs.
Cet article étudie les relations en termes de calculs explicites entre trois définitions constructives qui ont été proposées. De cette manière ces trois versions sont prouvées équivalentes directement en mathématiques constructives.

\end{abstract}

\medskip \noindent {\bf Keywords:} Mathématiques constructives, dimension valuative d'un anneau commutatif, \algos, programme de Hilbert pour l'\alg abstraite.

\smallskip \noindent {\bf MSC:} 
13B40, 13J15, 03F65

\newpage 
 
\setcounter{tocdepth}{4}
\markboth{Table des matières}{Table des matières}

\printcontents[french]{}{1}{}

\setcounter{section}{0}
\setcounter{subsection}{0}
\setcounter{equation}{0}

\section{Introduction}
Cet article est écrit dans le style des \coma à la Bishop (\cite{fBi67,fBB85,fBR1987,fACMC,fMRR,fYen2015}).

On utilisera quand c'est \ncr le vocabulaire et les notations des \sads. Voir \cite{fCLR01,fCL05,fLom06,fLom2020}.

Nous comparons dans cet article les différentes versions \covs de la dimension valuative que l'on trouve dans \cite{fCoq2009,fKY2020,fCACM,fACMC}, ainsi qu'une version \cov qui étend celle de \cite[Coquand]{fCoq2009} au cas d'un anneau non \ncrt intègre.

\smallskip Le lecteur qui ne connait pas les \coma dans le style de Bishop peut consulter les chapitres 1 et 2 de  \citealt{fBi67}, l'article  \citealt{fCL05} et les recensions \citealt{fmyhill72} and \citealt{fstolzenberg70}. 

\smallskip Quand une \dfn ou un \tho classiques utilisent des notions abstraites sans contenu calculatoire, nous essayons en \coma de trouver ce que nous appelons une \textsl{version \cov} de cette \dfn ou de ce \tho. Cette version doit être \eqve en \clama à la version classique. De plus il est \ncr que les exemples de base usuels puissent être traités avec la version \cov. Par exemple la version \cov d'un \alo est simplement un anneau dans lequel, lorsqu'une somme finie d'\elts est \iv, l'un des termes de la somme doit être \iv.

\subsection{Définition et \prts \caras de la dimension valuative en \clama}\label{fsubsecdimval}

En \clama, la dimension valuative d'un anneau intègre \(\gR\) est la longueur maximum \(n\) d'une chaine d'\advs \(\gV_0\subsetneq \dots \subsetneq \gV_n=\gK\) du corps de fractions \(\gK=\Frac\gR\) qui contiennent \(\gR\).

La dimension valuative d'un anneau arbitraire est définie comme la borne supérieure des dimensions valuatives de ses quotients intègres~(\cite{fCah90}). 

%

En \clama, les \eqvcs suivantes sont bien connues (on note $\Kdim(\gR)$ la dimension de Krull de l'anneau $\gR$, i.e.\ la longueur maximale d'une chaine d'\ideps dans~$\gR$).

Rappelons qu'un \cdi est de dimension de Krull $0$, que la dimension de l'anneau trivial, qui n'a pas de quotient intègre, est par convention égale à  $-1$.

\begin{ftheorem} \label{fthclassdimval} Soit $\gR$ un anneau commutatif intègre  non trivial avec $\gK=\Frac(\gR)$. \Propeq
\begin{enumerate}
\item\label{fthclassdimval1} $\gR$ est de dimension valuative $\leq n$.
\item\label{fthclassdimval2} Pour tout entier $k$ et tous $x_1,\dots,x_k\in \gK$, $\Kdim(\gR[\xk])\leq n+k$.
\item\label{fthclassdimval3} Pour tout entier $k$, $\Kdim(\gR[\Xk])\leq n+k$.
\item\label{fthclassdimval4}  $\Kdim(\gR[\Xn])\leq 2n$.
%
%
\end{enumerate}
En outre, sans supposer $\gR$ intègre, mais en le supposant non trivial, les points 1, 3 et~4 sont toujours \eqvs. 
\end{ftheorem}

 Enfin, dans l'article \cite{fKY2020}, qui fait suite à \cite[Kemper \& Trung]{fKV2014}, une nouvelle \carn \cov de la dimension valuative d'un anneau commutatif est proposée, inspirée de la \carn \cov de la \ddk donnée dans \cite{fLom06}.

\smallskip Nous sommes donc en possession d'au moins trois approches possibles de la dimension valuative d'un anneau commutatif en \clama. Celle correspondant au point \emph{\ref{fthclassdimval1}} ci-dessus, celle correspondant aux points \emph{\ref{fthclassdimval3}} et \emph{\ref{fthclassdimval4}}
et celle proposée par Kemper et Yengui. 

\smallskip
Nous proposons dans la  section \ref{fsecdival2} de rappeler des \dfns \covs précises concernant ces trois approches. 

Nous notons $\Vdim(\gR)$ une \dfn constructive correspondant à la \carn donnée dans le point \textsl{1} du~\thref{fthclassdimval}.

Nous notons $\vdim(\gR)$ une \dfn constructive correspondant à la \carn donnée dans le point \textsl{3} du~\thref{fthclassdimval}.

Nous notons $\dimv(\gR)$ la \dfn constructive donnée par Kemper et Yengui.

Ces \dfns \covs ont déjà été démontrées \eqves à la \dfn classique en \clama, au moins dans le cas d'un anneau intègre.  Dans la section \ref{fsecdival3} nous démontrons \underline{\cot} l'\eqvc de ces trois \dfns en toute généralité.

\subsection{Terminologie \cov de base}\label{fsubsubsectermcov}

Une partie $P$ d'un ensemble $E$ est dite \textsl{détachable} lorsque la \prt~\hbox{$x\in P$} est décidable pour les $x\in E$. Autrement dit, la règle suivante est satisfaite:

\Regles{\labu \(\vd x\in P \vou x\notin P .\)}

\noindent Pour décrire cette situation, il faut donc introduire à la fois le prédicat d'appartenance et le prédicat opposé.

Nous disons qu'un anneau est \textsl{intègre} (ou est un domaine d'intégrité), lorsque tout \elt est nul ou régulier, et qu'un anneau est un \textsl{corps discret} lorsque tout \elt est nul ou inversible. Cela n'exclut pas l'anneau trivial.

Un anneau est dit \textsl{\sdz} lorsque  la règle

\Regles{\labu \(\,\, xy=0\Vd x=0 \vou y=0\)}

\noindent est satisfaite. Un anneau intègre est \sdz. La réciproque, valable en \clama, n'est pas assurée \cot\footnote{{En \coma, \gui{ou} a son sens intuitif, \emph{i.e.} l'une des deux \prts est valide de manière explicite. Sachant d'un anneau est \sdz avec un \gui{ou} explicite, il n'y a pas de \demo
 \cov que tout \elt est nul ou régulier avec un \gui{ou} explicite.
}}.

Certaines versions locales des notions d'anneau intègre et d'anneau \sdz sont discernables y compris en \clama.

Un anneau est dit \textsl{\lsdz\footnote{En anglais, un \gui{pf-ring}: les \idps sont plats.}} lorsque, la règle suivante est satisfaite

\Regles{\labu \(\,\, ab=0\Vd \,\Exists s,t\;(sa=0\vet tb=0\vet s+t=1).\)}

\noindent
Alors, dans \(\gR[1/s]\) l'\elt~\(a\) est nul, et dans \(\gR[1/t]\) l'\elt \(b\) est nul\footnote{En \clama, les \elts d'un anneau peuvent être vus comme des \gui{fonctions} définies sur le spectre de Zariski. Ici nous avons deux ouverts de base $\rD(s)$ et $\rD(t)$ qui recouvrent le spectre de Zariski; sur le premier $a=0$, sur le second $b=0$.}.

Un anneau est dit \textsl{\qi\footnote{En anglais, un \gui{pp-ring}: les \idps sont projectifs.}} si l'annulateur de tout \elt \(a\) est engendré par un \idm (\ncrt unique), noté \(1-e_a\). On a
$\gR\simeq \gR[1/e_a]\times \aqo{\gR}{e_a}$.
Dans l'anneau $\gR[1/e_a]$, l'\elt $a$ est \ndz, et dans $\aqo{\gR}{e_a}$, $a$ est nul\footnote{En \clama nous avons une partition du spectre de Zariski en deux ouverts de base $\rD(1-e_a)$ et $\rD(e_a)$; sur le premier $a=0$, sur le second $a$ est régulier.}.
Un anneau \qi est \lsdz mais la réciproque n'est pas valable.
Notons que l'on~a $e_{ab}=e_a e_b$, $e_aa=a$ et $e_0=0$.
Les anneaux \qis ont une \dfn purement équationnelle. Supposons en effet qu'un anneau commutatif soit muni d'une loi  unaire \hbox{$a\mapsto
\ci{a}$} qui vérifie les trois axiomes suivants:
\begin{equation}\label{feqaqis}
\ci{a}\,a=a,\quad
\ci{(ab)}=\ci{a}\,\ci{b},\quad
\ci{0}=0
\end{equation}
Alors,   pour tout $a\in\gR$,  $\Ann(a)=\gen{1-\ci{a}}$ et $\ci{a}$
est \idm, de sorte que l'anneau  est \qi.

\begin{flemma}[lemme de scindage \qi] \label{fthScindageQi}
Soient $n$ \elts $x_1$, \dots, $x_n$ dans un anneau \qi~$\gR$.
Il existe un \sfio $(e_j)$ de cardinal~$2^n$ tel que dans chacune des
composantes~\hbox{$\gR[1/e_j]$}, chaque~$x_i$ est nul ou \ndz.
\end{flemma}

Le fait de pouvoir scinder systématiquement en deux composantes un anneau \qi
conduit à la méthode \gnle suivante.
La différence essentielle avec le lemme de scindage précédent
est que l'on  ne connait pas à priori la famille finie d'\elts qui
va provoquer le scindage.

\rdb
\mni {\bf Machinerie locale-globale \elr\ \num1.}\label{fMethodeQI}
{\sl La plupart des \algos qui fonctionnent avec les anneaux intègres non
triviaux peuvent être modifiés de manière à fonctionner avec les anneaux
\qis, en scindant l'anneau
en deux composantes chaque fois que l'\algo écrit pour les anneaux intègres
utilise le test
\gui{cet \elt est-il nul ou \ndz?}. Dans la première composante l'\elt en question
est nul, dans la seconde il est \ndz.}

\medskip Nous disons qu'un idéal est \textsl{premier} s'il produit au quotient un anneau \sdz. Cela n'exclut pas l'\id $\gen{1}$.
Ces conventions (adoptées dans \cite{fCACM}) n'utilisent pas la négation et permettent d'éviter certains raisonnements cas par cas litigieux d'un point de vue \cof.

\subsection{Domaines de valuation}\label{fsubsubsecdomval}

 Un \textsl{domaine de valuation} $\gV$ est un sous-anneau d'un corps discret $\gK$
satisfaisant l'axiome

\Regles{\labu \(\,\, xy=1 \Vd x\in \gV \vou y\in \gV \quad (x,y\in\gK)\)}


On dit alors que $\gV$ est un \textsl{anneau de valuation du corps discret $\gK$} et que $(\gK,\gV)$ est un \textsl{corps valué}.

Un domaine de valuation est la même chose qu'un domaine de Bézout local, ou encore un anneau intègre dont le groupe de divisibilité est totalement ordonné.

Dans un corps valué $(\gK,\gV)$ on dit que $x$ divise $y$ et l'on écrit $x\di y$ s'il existe \hbox{un $z\in\gV$} tel que $xz=y$.
On note $\Gamma(\gV)$ (ou $\Gamma$ si le contexte est clair) le groupe $\gK\eti\!/\gV\eti$ (noté additivement), avec la relation d'ordre~$\leq$ induite par la relation $\di$ dans $\gK\eti$. On note $\Gamma_\infty=\Gamma\cup\so\infty$ (où $\infty$ est un \elt maximum). Dans ces conditions, l'application naturelle $v:\gK\to\Gamma_\infty$ est appelée la \textsl{valuation} du corps valué. On a
\[
v(xy)=v(x)+v(y), \;\hbox{et }v(x+y)\geq \min(v(x),v(y)) \hbox{ avec égalité si } v(x)\neq v(y).
\]
On a aussi $\gV=\sotq{x\in\gK}{v(x)\geq 0}$ et 
$\gV\eti=\sotq{x\in\gK}{v(x)= 0}$.

\subsection{Dimension des treillis distributifs}\label{fsubsecdimtrdi}

Dans ce paragraphe nous expliquons la \dfn \cov de la dimension d'un \trdi. Un \trdi $\gT$ peut être vu comme l'ensemble des \oqcs d'un \sps, que l'on appelle \textsl{l'espace dual} de $\gT$. La \dfn \cov de cette dimension correspond à la définition de la dimension de l'\sps dual en \clama, qui est \egmt la longueur maximum d'une chaine croissante d'\ideps de $\gT$: voir le \thref{fth-dico-trdi-spec-dim1}.

\smallskip Un \textsl{idéal} $\fb$ d'un \trdi $(\gT,\vi,\vu,0,1)$ est une
partie qui satisfait les con\-traintes:%
\begin{equation}\label{feqIdeal}
\left.
\begin{array}{rcl}
  & & 0 \in \fb  \\
x,y\in \fb & \Longrightarrow  & x\vu y \in \fb  \\
x\in \fb ,\; z\in \gT& \Longrightarrow  & x\vi z \in \fb  \\
\end{array}
\right\}
\end{equation}
On note $\gT/(\fb=0)$ le treillis quotient obtenu en forçant les \elts de $\fb$ à être nuls. On peut aussi définir les \ids comme les noyaux de morphismes entre treillis.

Un \textsl{\idp} est un \id engendré par un seul \elt $a$, il est noté $\dar a$.
On~a $\dar a=\sotq{x\in\gT}{x\leq a}$.
L'\id $\dar a$, muni des lois $\vi$ et $\vu$ de $\gT$ est un \trdi
dans lequel l'\elt maximum est $a$. L'injection canonique $\dar
a\rightarrow \gT$ n'est pas un morphisme de \trdis parce que
l'image de $a$ n'est pas égale à $1_\gT$. Par contre
l'application $\gT\rightarrow \dar a,\;x\mapsto x\vi a$
est un morphisme surjectif, qui définit donc $\dar a$ comme la
structure quotient $\gT/(a=1)$.

La notion de \textsl{filtre} est la notion opposée (\cad obtenue en renversant la relation d'ordre) à celle d'idéal.

Une règle particulièrement importante
pour les \trdis, appelée \textsl{coupure}, est la
suivante
\begin{equation}\label{fcoupure1}
  (x\vi a \leq  b) \vii  (a \leq x\vu  b)
 \,\Rightarrow\,  (a \leq  b).
\end{equation}

Si $A\in\Pfe(\gT)$ (ensemble des parties finiment énumérées de~$\gT$)  on notera
$$\ndsp \Vu  A:=\Vu _{x\in A}x\qquad {\rm et}\qquad \Vi  A:=\Vi _{\!x\in A}x.
$$

On  note $A \vda B$ ou $A \vdash_\gT B$ la relation définie comme suit  sur l'ensemble $\Pfe(\gT)$:

\snic{A \vda B \; \; \equidef\; \; \Vi  A\;\leq \;
\Vu  B.}

Cette relation vérifie les axiomes suivants, dans lesquels on
écrit $x$ pour $\{x\}$ et $A, B$  \hbox{pour $ A\cup B$}.

\vspace{-.5em}
$$\arraycolsep3pt\begin{array}{rcrclll}
&    & x  &\vda& x    &\; &(R)     \\[1mm]
 \hbox{si } A \vda B &   \hbox{alors}  & A,A' &\vda& B,B'   &\; &(M)     \\[1mm]
\hbox{si } (A,x \vda B) \hbox{ et }  (A \vda B,x)
& \hbox{alors} & A &\vda& B &\;
&(T).
\end{array}$$
On dit que la relation est \textsl{réflexive}, \label{fremotr} \textsl{monotone} et
\textsl{transitive}.
La troisième règle (transitivité) peut être vue comme une
\gnn de la règle (\ref{fcoupure1}) et s'appelle \egmt la
règle de \textsl{coupure}.

\begin{fdefinition}
\label{fdefEntrel}
Pour un ensemble $S$ arbitraire, une relation sur  $\Pfe(S)$  qui est
réflexive, monotone et transitive est
appelée une {\sl \entrel} (en anglais, {\sl entailment relation}).
\end{fdefinition}

Le théorème suivant est fondamental. Il dit que les
trois propriétés des \entrels sont exactement ce qu'il faut pour que
l'interprétation en forme de \trdi soit adéquate.

\begin{ftheorem}[\tho fondamental des \entrels] \label{fthEntRel1} \cite{fCC00,fLor1951} \\
Soit un ensemble $S$  avec une \entrel
$\vdash_S$ sur $\Pfe(S)$. On considère le \trdi $\gT$ défini par
\gtrs et relations comme suit: les \gtrs sont les
\elts de $S$ et les relations sont les
$$ A\; \vdash_\gT \;  B
$$
chaque fois que $A\; \vdash_S \; B$.  Alors, pour tous $A$,
 $B$ dans $\Pfe(S)$,  on~a
$$  A\; \vdash_\gT \;  B
\; \Longrightarrow \; A\; \vdash_S \;  B.
$$
\end{ftheorem}

En \clama un {\em \id premier} $\fp$ d'un \trdi $\gT\neq \Un$ est un \id dont
le complémentaire $\fv$ est un filtre (qui est alors un {\em
filtre premier}).  On a alors $\gT/(\fp=0,\fv=1)\simeq\Deux$.  Il
revient au même de se donner un \idep de $\gT$ ou un morphisme de
\trdis $\gT\rightarrow \Deux$.

\begin{ftheorem}[dimension d'un \trdi] \emph{Voir \cite{fCL2003,fLom02}, \cite[chapitre XIII]{fACMC}.} \label{fth-dico-trdi-spec-dim1}
En \clama, pour un \trdi non trivial et pour $n\geq 0$, \propeq
\begin{enumerate}
\item Le treillis est de dimension $\leq n$, \cad par \dfn la longueur de toute chaine d'\ideps est $\leq n$.
\item Pour tout $x\in\gT$ le treillis quotient $\gT/(x=0,I_x=0)$, où $I_x=\sotq{y}{x\vi y=0}$, est de dimension $\leq n-1$\footnote{Un treillis est dit de dimension \(-1\) s'il est trivial, i.e. réduit à un point, cela initialise la récurrence dans le point \textsl{2}. On vérifie facilement qu'un treillis est \zed, \cad de dimension \(\leq 0\) \ssi c'est une algèbre de Boole.}.
\item
Pour toute suite $(x_0,\dots,x_n)$ dans $\gT$ il existe une suite \emph{complémentaire}  $(y_0,\dots,y_n)$ au sens suivant
\begin{equation}\label{feqC2G}
\left.\arraycolsep3pt
\begin{array}{rcl}
1& \vda  &   y_n, x_n\\
 y_n,  x_n & \vda  &  y_{n -1}, x_{n -1}  \\
\vdots~~~~& \vdots  &~~~~  \vdots \\
  y_1, x_1& \vda  &  y_0, x_0  \\
y_0, x_0& \vda  & 0
\end{array}
\right\}
\end{equation}
\end{enumerate}
\end{ftheorem}

Par exemple, pour $n=2$ les inégalités dans le point \textsl{3} correspondent au dessin suivant dans~$\gT$.
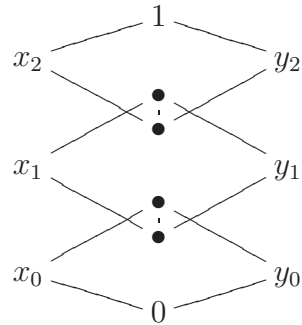
$$\SCo{1.5pt}{1.2cm}{x_0}{x_1}{x_2}{y_0}{y_1}{y_2}$$

\noindent \textsl{Définition et remarque importantes.} Les points \textsl{2} et \textsl{3} sont \eqvs en \coma et servent de \dfn pour la dimension de \(\gT\), notée \(\Kdim(\gT)\) et appelée \textsl{dimension de Krull} du \trdi.

Notez cependant que d'un point de vue \cof on a seulement défini clairement
la phrase \(\Kdim(\gT)\leq n\). Pour établir \gui{$\Kdim(\gT)= n$} il serait \ncr de démontrer aussi $\lnot (\Kdim(\gT)\leq n+1)$. 
Heureusement, la plupart des \thos en \clama fonctionnent avec une hypothèse du type  $\Kdim(\gT)\leq n$.

On a en outre le résultat \cof suivant, qui simplifie l'utilisation du point \textsl{3}.

\begin{flemma} \label{flem-trdi-spec-dim}
Soit $S$ une partie d'un \trdi $\gT$ qui engendre $\gT$ en tant que \trdi.
En \coma,  pour $n\geq 0$, \propeq
\begin{enumerate}
\item
Pour toute suite $(x_0,\dots,x_n)$ dans $\gT$ il existe une suite \emph{complémentaire}  $(y_0,\dots,y_n)$ dans \(\gT\).
\item
Pour toute suite $(x_0,\dots,x_n)$ dans $S$ il existe une suite \emph{complémentaire}  $(y_0,\dots,y_n)$ dans \(\gT\).
\end{enumerate}

\end{flemma}

Dans les deux paragraphes qui terminent cette section, nous expliquons comment les spectres des \trdis conduisent à des \spss en algèbre abstraite. La dualité entre \trdis et \spss a été établie par  \citet*{fSto37}. \citet*{fHoc1969} a introduit la terminologie d'\sps et a popularisé leur utilisation. Mais il ne cite pas  \citealt*{fSto37}
et il n'indique pas le lien entre les \spss et les \trdis.

\subsection{Dimension de Krull d'un anneau commutatif}\label{fsubsecdimcomring}

Nous rappelons tout d'abord l'idée principale de l'approche \cov du spectre d'un anneau commutatif développée dans 
\cite{fJoy71,fJoy76}.

Si $\fa$ est un \id de $\gR$, nous notons $\DR(\fa)$ (ou $\rD(\fa)$ si le contexte est clair) le
nilradical de l'\id $\fa$:
\begin{equation} \label{feqZar}
\begin{array}{rclcl}
\DR(\fa)&  = & \sqrt[\gR]{\fa} &=&\sotq{x\in\gR}{\Ex m\in\NN\;\; 
x^m\in\fa}\text.
\end{array}
\end{equation}
Quand $\fa=\gen{x_1,\ldots ,x_n}$, nous notons $\DR(\fa)$ sous la forme
$\DR(x_1,\ldots ,x_n)$.
Si le contexte est clair nous notons aussi $\wi{x}$ pour $\DR(x)$.

Par \dfn, le {\sl treillis de Zariski} de $\gR$, noté $\ZarR$, est l'ensemble des $\DR(x_1,\ldots ,x_n)$ avec l'inclusion pour relation d'ordre. Les bornes inférieure et supérieure sont données par  
\[
\DR(\fa_1)\vi\DR(\fa_2)=\DR(\fa_1\fa_2)\quad \mathrm{et} \quad
\DR(\fa_1)\vu\DR(\fa_2)=\DR(\fa_1+\fa_2).
\]
Le treillis de Zariski de $\gR$
est un \trdi et $\DR(x_1,\ldots ,x_n)=
\wi{x_1}\vu\cdots \vu\wi{x_n}$.
les \elts  $\wi{x}$ forment \sgr (stable par $\vi$) de $\ZarR$.

Si $M$  est un \mo  multiplicatif dans $\gR$ engendré par  $U=\so{u_1,\dots,u_m}$ et si $\fa=\gen{\an}$ est un \itf nous avons les \eqvcs 
\begin{equation} \label {feqZarA}
\Vi_{i\in \so{1,\dots,m}} \wi {u_i} \;\leq_{\ZarR} \Vu_{j\in \so{1,\dots,n}} \wi{a_j}
\quad\Longleftrightarrow \;
\prod_{u\in  \so{1,\dots,m}} u_i  \in \sqrt[\gR]{\fa}
\quad\Longleftrightarrow \quad
M\,\cap\, \fa\neq \emptyset \; 
\end{equation}

Cela décrit complètement le \trdi $\ZarR$. En effet la première formule
dans~(\ref{feqZarA}) définit une \entrel sur $\gR$ qui donne une autre description $\ZarR$.

Nous avons aussi la   \carn suivante \citep[voir][]{fCC00,fCL2003}.
\begin{fproposition}[\dfn à la Joyal du spectre d'un anneau commutatif]
\label{fpropZar}~\\
Le \trdi $\ZarR$ est
(à \iso unique près) le treillis engendré par les symboles $\DR(a)$ pour $a\in\gR$
soumis aux relations suivantes:
\[\begin{array}{cccc}
\DR(0_\gR) =0   ,\; \DR(1_\gR)= 1 ,\;
   \DR(x+y) \leq \DR(x)\vu\DR(y) ,\; \DR(xy) = \DR(x)\vi \DR(y).
\end{array}\]
\end{fproposition}

La construction $\gR\mapsto\ZarR$ donne un foncteur de la catégorie des anneaux commutatifs vers celle des \trdis.
Via ce foncteur la projection $\gR\to\gR/\DR(0_\gR)$ donne un \iso
$\ZarR\to \Zar(\gR/\DR(0_\gR))$. Nous avons $\ZarR=\Un$ \ssi  $1_\gR=0_\gR$.

En \clama on a une \demo simple du fait que la dimension du \trdi $\ZarR$ est la \ddk de l'anneau $\gR$.

\smallskip Le treillis de Zariski d'un anneau est l'exemple paradigmatique d'un \trdi engendré par une \sad. Nous expliquons la méthode \gnle dans le paragraphe qui suit.

\subsection{Structures \agqs dynamiques}

Références: \cite{fCLR01,fLom98,fLom06,fLom2020,fBC2005,fCoq2005}. Les \sads (du premier ordre) sont explicitement nommées dans~\cite{fLom98,fLom06}. Dans  \cite{fCLR01} elles sont implicites, mais décrites sous la forme de leurs présentations.
Elles sont \egmt implicites dans \cite{fLom02}, et, last but not least, dans \cite[{D5}, 1985]{fD5}, qui a été une source d'inspiration essentielle: on peut calculer de manière sûre dans la clôture algébrique d'un corps discret, même quand il n'est pas possible de construire cette clôture algébrique.
Il suffit donc de considérer la clôture \agq comme une \sad \gui{à la D5} plutôt que comme une structure \agq usuelle: \textsl{l'évaluation paresseuse à la D5 fournit une sémantique \cov pour la clôture \agq d'un \cdi}.

Une étude plus détaillée en cours de rédaction se trouve dans \cite[Théories géométriques pour l'algèbre constructive]{fLom-tgac}. Voir aussi \citealt{fLM2022}.

Une \tdy (finitaire) $\sa{T}=(\cL,\cA)$ est une version purement calculatoire, sans logique, d'une théorie cohérente. Le langage $\cL$ est donné par une signature, les axiomes (\elts de $\cA$) sont des \rdys.

Une \textsl{\sad $\gD$ pour une \tdy \sa{T}} est donnée par \gtrs et relations:
$\gD=\big((G, R), \sa{T}\big)$. Si $(G,R)$ est le diagramme positif d'une structure \agq usuelle $\gR$ on note $\gD=\sa{T}(\gR)$.

Considérons une  \sad \(\gD=\big((G,R),\sa{T}\big)\) pour une \tdy $\sa{T}=(\cL,\cA)$.
Soit $S$  un ensemble de formules atomiques closes de \(\gD\). On définit la \entrel $\vdash_{\gD,S}  $ sur $S$ associée à \(\gD\)  comme suit pour des \(A_i\) et \(B_j\in S\):

\vspace{-.5em}
\begin{equation} \label{feq2}
\begin{aligned}
 A_1,\dots,A_n  &\,\vdash_{\gD,S} B_1,\dots,B_m
   \qquad\quad \equidef     \\[.2em]
    A_1\vet \dots\vet  A_n &\Vdi{\gD} B_1\vou\dots\vou B_m
 \end{aligned}
\end{equation}
On pourra noter $\Zar(\gD,S)$ le \trdi engendré par cette \entrel.
Intuitivement ce treillis est le treillis des valeurs de vérité des formules de \(S\) dans la \sad \(\gD\).

Le \textsl{treillis de Zariski (complet) d'une \sad \(\gD\)} est défini en prenant pour~$S$ l'ensemble $\Atcl(\gD)$ de toutes les formules atomiques closes de~\(\gD\). On le note $\Zar(\gD,\sa T)$ ou $\Zar(\gD)$ ou avec un nom particulier correspondant à la théorie~\sa{T}.

L'espace spectral dual est appelé le \textsl{spectre de Zariski de la \sad~\(\gD\)}, ou peut aussi lui attribuer un nom particulier.

Enfin  la \textsl{dimension de Krull de \(\gD\)} est par \dfn égale à
\(\Kdim(\Zar(\gD))\). La \dfn de \(\Vdim(\gR)\) dans le paragraphe suivant est donnée selon ce schéma.

\section{Trois \dfns \covs de la dimension valuative}\label{fsecdival2}

\subsection{Clôture \qi minimale d'un anneau réduit}\label{fsubsecCqimiared}

Nous rappelons ici quelques résultats essentiels donné dans \cite{fCACM}.
Si le contexte \(a\in\gR\) est clair, nous notons $a\epr$ l'\id annulateur de l'\elt $a$ dans \(\gR\). Nous utilisons aussi la notation $\fa\epr$ pour l'annulateur d'un \id~$\fa$.

\begin{flemma}\label{flem20MorRc}
Soit   \(\gR\) un  anneau réduit
et $a\in\gR$.
On définit
$$\gR_{\so{a}}\eqdefi\gR\sur{a\epr}\times \gR\sur{({a\epr})\epr}$$
et l'on note $\psi_a:\gR\to\gR_{\so{a}}$  l'\homo canonique.
\begin{enumerate}
\item $\psi_a(a)\epr$ est engendré par l'\idm \((\ov 0,\wi 1)\),
donc \(\psi_a(a)\epr=(\ov 1,\wi 0)\epr\).
\item $\psi_a$ est injectif (on peut identifier \(\gR\) à un sous-anneau de \(\gR_{\so{a}}\)).
\item Soit $\fb$  un \id  dans  \(\gR_{\so{a}}\), alors l'\id \(\psi_a^{-1}(\fb\epr)=\fb\epr\cap\gR\) est un \id annulateur dans \(\gR\).
\item L'anneau \(\gR_{\so{a}}\) est réduit.
\end{enumerate}
\end{flemma}

\begin{flemma}\label{flem3MorRc}
Soit   \(\gR\) réduit
et \(a,b\in\gR\). Alors avec les notations du lemme \ref{flem20MorRc}
 les deux anneaux \((\gR_{\so{a}})_{\so{b}}\) et \((\gR_{\so{b}})_{\so{a}}\) sont canoniquement isomorphes.
\end{flemma}

\rem Le cas où $ab=0$ est typique: quand on le rencontre, on voudrait
bien scinder l'anneau en composantes où les choses sont \gui{claires}.
La construction précédente donne alors les trois composantes
\[
\gR\sur{(ab\epr)\epr}, \; \gR\sur{(a\epr b)\epr}\, \hbox{ et } \,\gR\sur{(a\epr b\epr)\epr}.
\]
Dans la première $a$ est \ndz et $b=0$, dans la seconde
$b$ est \ndz et $a=0$, et dans la troisième $a=b=0$.
\eoe

\begin{flemma}\label{flem2qi}
Si \(\gR\subseteq\gC\) avec \(\gC\) \qi, le plus petit sous-anneau \qi de~\(\gC\) contenant~\(\gR\) est égal à \(\gR[(e_a)_{a\in\gR}]\),
où $e_a$ est l'\idm de $\gC$ tel que $\Ann_\gC(a)=\gen{1-e_a}_\gC$. Plus \gnlt, si $\gR\subseteq\gB$
avec $\gB$ réduit et si tout \elt $a$ de \(\gR\) admet un annulateur dans $\gB$
engendré par un \idm $1-e_a$,
alors le sous-anneau $\gR[(e_a)_{a\in\gR}]$ de $\gB$ est \qi.
\end{flemma}

 On note  \(\Rred=\gR/\!\sqrt[\gR]{\gen{0}}\) l'anneau réduit engendré par \(\gR\).

\begin{fthdef}[clôture \qi minimale]\label{fthAmin} ~ 
\\
Soit \(\gR\) un anneau réduit.
On peut définir un anneau $\Rmin$ comme limite inductive filtrante
 en itérant la construction de base qui consiste à
 remplacer~$\gE$ (l'anneau \gui{en cours}, qui contient \(\gR\)) par
\[\gE_{\so{a}}\eqdefi\gE\sur{a\epr}\times \gE\sur{({a\epr})\epr}=\gE\sur{\Ann_\gE (a) }\times \gE\sur{\Ann_\gE(\Ann_\gE (a) )},\]
 lorsque $a$ parcourt \(\gR\).
\begin{enumerate}
\item  Cet anneau $\Rmin$ est \qi, il contient \(\gR\) et il est entier sur \(\gR\).
\item Pour tout $x\in\Rmin$,
$x\epr\cap\gR$ est un \id annulateur dans \(\gR\).
\end{enumerate}
Cet anneau $\Rmin$ est appelé la \emph{clôture \qi minimale de \(\gR\)}.
\\
Lorsque \(\gR\) n'est pas \ncrt réduit,
  on prendra $\gR\qim\!\eqdefi (\Rred)\qim$.
\end{fthdef}

On note \(\cP_n\) l'ensemble des parties finies de \(\so{1,\dots,n}\). Voici une description de chaque anneau
obtenu à un étage fini de la construction de $\Rmin$.

\begin{flemma}\label{flem4MorRc}
Soit \(\gR\) un anneau réduit et $(\ua) = (\an)$ une suite de~$n$ \elts
de~\(\gR\).  Pour $I\in\cP_n$, on note $\fa_I$ l'\id
\[\fa_I = \big(\Prod_{i\in I} \gen{a_i}\epr \Prod_{j\notin I} a_j\big)\epr
= \big(\gen{a_i, i \in I}\epr \Prod_{j\notin I} a_j\big)\epr
.\]
Alors $\Rmin$ contient l'anneau suivant, produit de $2^n$
anneaux quotients de~\(\gR\) (certains éventuellement nuls):
\[\gR_{\so\ua} = \Prod_{I\in\cP_n} \gR\sur{\fa_I}.\]
 \end{flemma}

On note \(\BB(\gR)\) l'\agB des \idms de l'anneau \(\gR\).

\begin{flemma}\label{ffact2Amin}~
\begin{enumerate}
\item Soit \(\gR\) un anneau \qi.
\begin{enumerate}
\item $\Rmin=\gR$.
\item $\RX$ est \qi, et $\BB(\gR)=\BB(\RX)$.
\end{enumerate}

\item Pour tout anneau \(\gR\) on a un \iso canonique
\[\Rmin[\Xn]\simeq(\RXn)\qim.\]
\end{enumerate}
\end{flemma}

Notre suggestion est que la bonne \gnn de la notion du corps de fractions
d'un anneau $\gR$
n'est pas l'anneau  $\Frac\gR$ mais l'anneau \zedr $\Frac\Rmin$.

\subsection{Première approche \cov de la dimension valuative: \(\vdim\)}\label{fsubsecdimvalcons}

\begin{sloppypar}
  Dans l'ouvrage \cite{fCACM}, les auteurs notent $\Vdim(\gR)$ au lieu de $\vdim(\gR)$. Nous préférons dans cet article, réserver $\Vdim(\gR)$ pour la notion définie par T. Coquand.
\end{sloppypar}

\smallskip Dans \cite{fCACM} la \dfn suivante est \cov parce qu'il y a une \dfn \cov de la \ddk. Et le \tho suivant est démontré \cot.
Donc, pour le cas d'un anneau intègre, le~\thref{fthValDim} 
donne une version \cov de l'\eqvc entre les points \textsl{2}, \textsl{3} et \textsl{4} dans le \tho classique~\ref{fthclassdimval}.

\begin{fdefinition}[selon \cite{fCACM}]\phantomsection\label{fdefiValdim} 
\begin{enumerate}
\item []
\item Si $\gR$ est un anneau \qi, la \textsl{dimension valuative}
est définie comme suit. Soit $n\in\NN$ et $\gK=\Frac\gR$, on dit que \textsl{la dimension
valuative de $\gR$ est inférieure ou égale à $n$} et l'on écrit
$\vdim\gR\leq n$ si pour toute suite  $(\xm)$ dans~$\gK$ on a
$\Kdim \gR[\xm] \leq n$. Par convention $\vdim\gR=-1$ si $\gR$ est trivial.
\item Dans le cas \gnl on définit \gui{$\vdim\gR\leq n$} par \gui{$\vdim\Rmin\leq n$}\footnote{
\citet{fCACM} démontrent qu'en \clama on a l'\eqvc \hbox{$\vdim\gR\leq d\iff\Vdim \gR\leq d$} où $\Vdim\gR$ est définie comme dans \citealt{fCah90}. Par suite, cela démontre que $\Vdim\gR=\Vdim\Rmin$ en \clama. C'est une manière de ramener le cas \gnl au cas d'un anneau intègre sans utiliser d'\ideps.}.
\end{enumerate}
\end{fdefinition}

Pour $n= 0,-1$, on a $\vdim(\gR)=n\Leftrightarrow\Kdim(\gR)=n$.

On note que, comme pour la dimension de Krull, on n'a pas vraiment défini \(\vdim\gR\)  comme un \elt de \(\NN\cup\so\infty\): seule la propriété \gui{\(\vdim\gR\leq n\)} est bien définie, \cot, pour tout entier \(n\geq -1\).

\begin{ftheorem}[{\citealt[Theorem XIII-8.19]{fCACM}}]\label{fthValDim}
Les \eqvcs suivantes sont valides.
\begin{enumerate}
\item Si $n\geq 1$ et $k\geq -1$, alors
\begin{equation} \label{feq1thValDim}
\vdim\gR\leq k\iff\vdim\RXn\leq n+k.
\end{equation}
Si \(\gR\) est non trivial et \(\vdim\gR<\infty\), cela signifie intuitivement que l'on a une \egt $\vdim\RXn=n+\vdim\gR$.
\item Si $n\geq 0$, alors
\begin{equation} \label{feq2thValDim}
\vdim\gR\leq n\iff \Kdim \RXn \leq 2n.
\end{equation}
%
\item Dans le cas où $\gR$ est \qi, ceci est aussi \eqv à:\\
pour tous $x_1$, \dots, $x_n$ dans $\Frac\gR$, on a $\Kdim\Rxn\leq n$.
\end{enumerate}
\end{ftheorem}
%

\subsection{Deuxième approche: treillis valuatif et spectre valuatif, \(\Vdim\)}\label{fsubsecSpecVal}

Tout d'abord rappelons que concernant la dimension d'un \trdi ou de l'espace spectral dual, il y a trois approches \covs \cot \eqves. La première historiquement est celle définie par \citet{fJoy76}, reprise dans \citealt{fCP2001} et \cite{fCoq2009}. La deuxième provient de la notion de chaine potentielle d'\ideps, à la manière de ce qui est fait pour les anneaux commutatifs dans \cite{fLom06}. La troisième est basée sur la notion d'\id \gui{bord} (ou de filtre \gui{bord}) et donne lieu à une \dfn par récurrence. L'équivalence de ces notions est pour l'essentiel démontrée dans \cite{fCL2003} et traitée très en détail dans \cite{fCL2001-2018}.

Dans le cas d'un sous-anneau (intègre) \(\gR\) d'un corps discret \(\gK\), l'article \cite{fCoq2009} définit le treillis valuatif \(\Val (\gK,\gR)\) comme le \trdi qui traduit les règles valides pour le prédicat \(\Vr\)\footnote{\(\Vr\) rappelle \gui{Valuation ring}.} dans la \tdy  \(\sa{Val}(\gK,\gR)\) des \advs du corps~\(\gK\) qui contiennent \(\gR\) \footnote{Sans toutefois utiliser en tant que tel le langage des \tdys.}.
Enfin le treillis \(\Val (\gR)\) est une notation abrégée de \(\Val (\Frac\gR,\gR)\).

La \sad \(\sa{Val}(\gK,\gR)\) peut être décrite comme la \tdy construite sur la signature
\[(\cdot=0,
\Vr(\cdot)\mathrel{;}0,1,\cdot+\cdot,-\cdot,\cdot\times\cdot,(a)_{a\in\gK} )
\]
dans laquelle les \elts de \(\gK\) sont des constantes de la théorie\footnote{Pour la \tdy \sa{Val} toute nue, on supprime tout ce qui concerne $\gR$ et $\gK$.}.

Tout d'abord il y a les axiomes des \cdis non triviaux sur le langage des anneaux commutatifs.

\DeuxRegles{
\labu \(\vd 0=0\)
\labu \(\,\,x=0\vet y=0\vd x+y=0\)
\labu \(\,\,1=0 \vd \Bot \)
}
{
\labu \(\,\,x=0\vd xy=0\)
\labu \(\vd x=0\vou \Exists y \,xy=1\)
}

Ensuite on ajoute le diagramme de \(\gK\):

\DeuxRegles{
\labu \(\vd 0_\gK= 0\)
\labu \(\vd 1_\gK= 1\)
}
{
\labu \(\vd a+b=c\quad (\hbox{si } a+b=_\gK c)\)
\labu \(\vd ab= c\quad (\hbox{si } ab=_\gK c)\)
}

Enfin il y a les axiomes décrivant les \prts du prédicat  \(\Vr(x)\) qui signifie que \(x\) appartient à l'\adv supposé du corps \(\gK\).

\DeuxRegles{
\labu \(\vd \Vr(a)\quad (\hbox{si } a\in\gR)\)

\labu \(\,\,\Vr(x)\vet\Vr(y)\vd \Vr(x+y)\)
}
{
\labu \(\,\,\Vr(x)\vet\Vr(y)\vd \Vr(xy)\)
\labu \(\,\,xy=1\vd \Vr(x)\vou\Vr(y)\)
}

Le treillis \(\Val (\gK,\gR)\) est alors le \trdi  engendré par la \entrel
\(\vdash_{\Val,\gK,\gR}\)  sur \(\gK\eti\) définie par l'\eqvc

\vspace{-1em}
\begin{equation} \label{feq01}
\begin{aligned}
 x_1,\dots,x_n  &\,\,\vdash_{\Val,\gK,\gR}    y_1,\dots,y_m
  \qquad\quad     \equidef     \\[.1em]
   \Vr(x_1)\vet \dots\vet \Vr(x_n)&\vdi{\sA{Val}(\gK,\gR)} \Vr(y_1)\vou \dots\vou \Vr(y_m)
 \end{aligned}
\end{equation}

Enfin on note \(\Val (\gR)\) pour \(\Val (\Frac\gR,\gR)\).

Dans ce cadre, la dimension valuative de \(\gR\), que nous noterons \(\Vdim\gR\), est définie comme égale à \(\Kdim(\Val (\gR))\).

Le \tho 8 de l'article  \cite{fCoq2009} donne un \vst
sous la forme de l'\eqvc suivante (pour des \(y_i\) et \(x_j\in\gK\eti\))
\begin{multline*}
\Vr(x_1)\tsbf{,}\, \dots\tsbf{,}\, \Vr(x_n)\vdi{\sA{Val}(\gK,\gR)} \Vr(y_1)\MA{\tsbf{ ou }} \dots\MA{\tsbf{ ou }} \Vr(y_m)\\
\iff 1\in \gen{y_1^{-1},\dots,y_m^{-1}}\gR[x_1,\dots,x_n,y_1^{-1},\dots,y_m^{-1}]\text.
\end{multline*}
En chassant les dénominateurs on obtient la formulation \eqve suivante
\begin{equation} \label{feqVr}
y_1^{p_1}\cdots y_m^{p_m} =
\begin{aligned}[t]
&Q(x_1,\dots,x_n,y_1,\dots,y_m)\in\gR[\Xn,\Ym]\\
&\text{un \pol dont tous les \moms ont un degré}\\
&\text{en }  \uY  \text{ strictement inférieur à }  (p_1,\dots,p_m) .
\end{aligned}
\end{equation}

Que se passe-t-il si nous acceptons que certains \(x_i\) ou \(y_j\) soient éventuellement nuls? Comme la règle \(\vd \Vr(0)\) est valide,  la règle

\Regles{\labu \(\,\,\Vr(x_1)\vet \dots\vet \Vr(x_n)\vd \Vr(y_1)\vou \dots\vou \Vr(y_m)\)}

\noindent est toujours satisfaite si l'un des \(y_j\) est nul; et si l'un des
\(x_i\) est nul, elle est \eqve à la même règle où l'on a supprimé le \(\Vr(x_i)\) correspondant à la gauche du \(\vd\).
On constate les mêmes faits pour le \vst exprimé sous la forme \pref{feqVr}: si l'un des \(y_j\) est nul on prend \(Q=0\); si l'un des \(x_i\) est nul, il n'intervient pas dans \pref{feqVr}.  Ainsi le \vst sous la forme \pref{feqVr} est toujours valable, ce qui évite d'avoir à raisonner cas par cas.

Le treillis \(\Val \gR\) peut donc être caractérisé comme le \trdi  engendré par la \entrel
\(\vdash_{\Val\gR}\)  sur \(\gK\) définie par l'\eqvc
%
%
\begin{equation} \label {feqVr2}
 x_1,\dots,x_n  \vdash_{\Val\gR}    y_1,\dots,y_m\equidef
 \begin{aligned}[t]
   &\exists p_1,\dots,p_m\geq0\ \exists Q\in \gR[\uX,\uY]\enskip y_1^{p_1}\cdots y_m^{p_m} =\\
   &Q(x_1,\dots,x_n,y_1,\dots,y_m)\text{, les \moms de $Q$}\\
   &\text{dont le degré en $\uY$ est  $<(p_1,\dots,p_m)$.}
\end{aligned}
\end{equation}

\medskip Pour rendre les calculs à venir plus lisibles, nous introduisons le prédicat \[\Vp(x)\equidef \Exists u (ux=1\vet \Vr(u)).\]
En d'autres termes, nous ajoutons un prédicat \(\Vp(x)\) dans la signature avec les deux axiomes

\DeuxRegles{
\labu \(\,\,ux=1\vet \Vr(u)\vd \Vp(x)\)
}
{
\labu \(\,\,\Vp(x)\vd \Exists u \;(ux=1\vet \Vr(u)) \)
}

La nouvelle théorie est une extension conservative de la précédente. En outre on a les règles valides suivantes qui permettent de retrouver \(\Vr\) à partir de \(\Vp\)

\DeuxRegles{
  \labu $\vd \Vr(0)$
  \labu \makebox[0pt][l]{$\,\,\Vr(x)\vd x=0\MA{\tsbf{ ou }} \Exists u \;(ux=1\tsbf{,}\, \Vp(u))$}
}
{
\labu $\,\,ux=1\tsbf{,}\, \Vp(u)\vd \Vr(x)$}

On peut lire \(\Vp(x)\)  comme \(\Vr(1/x)\) en considérant que \(\Vr(1/0)\vd \Bot\) (effondrement de la théorie). Le prédicat \(\Vp(x)\) signifie que l'\elt \(x\) de \(\gK\) n'est pas un \elt résiduellement nul de \(\gV\).
Ce prédicat satisfait notamment les axiomes suivants dans la \tdy considérée:

\DeuxRegles{
\labu \(\vd \Vp(a)\quad (\hbox{si } a\in\Rti)\)
\labu \(\,\,\Vp(x+y)\vd \Vp(x)\vou\Vp(y)\)
\labu \(\,\,\Vp(x)\vet\Vp(y)\vd \Vp(xy)\)
}
{
\labu \(\,\,\Vp(0)\vd \Bot\)
\labu \(\,\,xy=1\vd \Vp(x)\vou\Vp(y)\)
}

Le \vst devient, cette fois-ci pour des \(x_i\) et \(y_j\in\gK\) sans restriction

\Regles{\labu \(\,\,\Vp(y_1)\vet \dots\vet \Vp(y_n)\vd \Vp(x_1)\vou \dots\vou \Vp(x_m)\)}
\vspace{-.3em}
\[\iff \qquad 1\in \gen{\xm}\gR[\xm,y_1^{-1},\dots,y_n^{-1}]\]
qu'il faut réécrire sous la forme suivante pour éviter les \(0^{-1}\): il existe des exposants $p_1,\dots,p_m$ et des \pols \(P_j\in\gR[\uX,\uY]\) tels que
\begin{multline} \label {feqV'}
  \exists p_1,\dots,p_m\enskip\exists P_1,\dots,P_n\in\gR[\uX,\uY]\enskip
  y_1^{p_1}\dots y_m^{p_m}= x_1P_1(\ux,\uy)+\dots+x_nP_n(\ux,\uy)\text,\\
  \text{où les multiexposants  $\uY$ dans les $P_j$'s sont tous $\leq(p_1,\dots,p_m)$.}
\end{multline}
 
Si l'un des $y_i$ est nul, l'\egt est automatiquement satisfaite car on peut prendre les~$P_j$ nuls.  

Ce prédicat \(\Vp\) est le prédicat \(\Nrn\) dans l'article \cite{fLom2000}
qui donne un \vst très \gnl.

\begin{fremark}
 \label{fremV'}
 On peut introduire le treillis
\(\Valp\gR\) associé au prédicat \(\Vp\). On peut démontrer alors que \(\Valp\gR\) est isomorphe au treillis opposé à \(\Val\gR\). Cela tient à ce que \(x\mt x^{-1}\) est une bijection de \(\gK\eti\) sur lui-même.
Dans cet article, nous utilisons seulement le fait que les \carns \pref{feqVr}
et \pref{feqV'} sont équivalentes.  \eoe
\end{fremark}


\medskip Nous introduirons le treillis valuatif d'un anneau arbitraire
dans la toute dernière section page~\pageref{fsubsec-vdim=Vdim2}.

\subsection{Troisième approche: ordre monomiaux rationnels gradués, \(\dimv\)}\label{fsubsecOmonRat}

On note \(<_{\mathrm{lex}}\) l'ordre monomial sur  \(\ZZ^n\) correspondant à l'ordre lexicographique.

Un \textsl{ordre monomial rationnel gradué} \(<_M\) sur \(\ZZ^n\) ou, de manière \eqve, sur les \moms de \(\gR[X_1^{\pm1}, \dots, X_1^{\pm1}]\) est défini au moyen d'une matrice \(M\in \Mat_n(\NN)\) inversible dans \(\Mat_n(\QQ)\) avec les \coes de la première ligne tous \(>0\), de la manière suivante:
\[
(e_1,\dots,e_n)<_M (f_1,\dots,f_n) \equidef
M\cdot \cmatrix{e_1\\ \vdots \\ e_n} <_{\mathrm{lex}} \cmatrix{f_1\\ \vdots\\f_n}
\]
Lorsque les \coes de la première ligne de \(M\) sont égaux à \(1\), l'ordre monomial \(<_M\) est un ordre subordonné au degré total.
L'ordre monomial \(<_{\mathrm{grlex}}\) \gui{gradué lexicographique} est
celui défini par la matrice
\[\cmatrix{
1&1&&\dots&1\\
1&0&0&\dots&0\\
0&1&0& &0\\
\vdots & && &\vdots\\
0&0&\dots& 1&0
}\]

\smallskip Dans \cite{fLom06}, la dimension de Krull d'un anneau arbitraire \(\gR\) est caractérisée \cot par l'équivalence entre \(\Kdim\gR\leq n\) et le fait que pour tous \(x_0,\dots,x_n\in\gR\) on a un \pol \(P\in\gR[X_0,\dots,X_n]\) qui annule \((x_0,\dots,x_n)\) et dont le plus petit \coe pour l'ordre lexicographique est égal à \(1\). Par exemple  \(\Kdim\gR\leq 1\) \ssi pour tous
 \(x_0,x_1\in\gR\) on peut trouver une \egt \(0=x_0^{e_0}(x_1^{e_1}(1+c_1x_1)+c_0x_0)\): ici les \(c_i\) sont des \elts de \(\gR\) ou tout aussi bien des \elts de \(\gR[x_0,x_1]\).

\smallskip Dans \cite{fKV2014}, les auteurs ont montré, en \clama, que pour un anneau \noe, on peut caractériser la \ddk de la même manière en utilisant un ordre monomial arbitraire.

\smallskip Dans \cite{fKY2020}, les auteurs ont montré, en \clama, que pour un anneau arbitraire, on peut caractériser la dimension valuative de la même manière à condition d'utiliser un ordre monomial rationnel gradué à la place de l'ordre lexicographique.

Autrement dit, nous pouvons les paraphraser en définissant l'inégalité \gui{\(\dimv\gR\leq n\)} pour \(n\geq 0\) comme suit.

\begin{fdefinition}\label{fdefKY1} considérons un ordre monomial rationnel gradué
\(<_M\), on demande que pour tous  \(x_0,\dots,x_n\in\gR\) on ait un \pol \(P\in\gR[X_0,\dots,X_n]\) qui annule \((x_0,\dots,x_n)\) et dont le plus petit \coe pour l'ordre \(<_M\) est égal à \(1\).
\end{fdefinition}

Le résultat dans~\cite{fKY2020} est alors le suivant.
\begin{ftheorem}  \phantomsection\label{fthKY1}
\begin{enumerate}
\item[]
\item Cette \dfn de \(\dimv\gR\leq n\) ne dépend pas de la matrice \(M\) considérée.
\item Elle équivaut en \clama au fait que la dimension valuative de~\(\gR\)
est \(\leq n\).
\end{enumerate}
\end{ftheorem}

Par convention \(\dimv\gR=-1\) signifie que l'anneau est trivial.

En fait, la \demo du \thref{fthKY1} dans \cite{fKY2020} est clairement \cov pour le cas d'un anneau intègre~\(\gR\): pour tout \(n\geq 0\) on a \cot l'\eqvc
\begin{equation} \label{feqvdim-dimv}
\vdim\gR\leq n \quad \text{(\dfn \ref{fdefiValdim}.1)} \iff \dimv\gR\leq n \quad \text{(définition~\ref{fdefKY1})}
\end{equation}

\newpage
\section{Équivalence \cov des trois \dfns \covs}\label{fsecdival3}

\subsection{\(\vdim=\dimv\)}\label{fsubsecvdim=dimv}

On a un premier lemme qui prolonge \pref{feqvdim-dimv} au cas \qi.

\begin{flemma} \label{flemvdim-dimvQi}
Pour un anneau \qi \(\gR\) on a l'\eqvc
\begin{equation} \label{feqvdim-dimvQi}
\vdim\gR\leq d \iff \dimv\gR\leq d
\end{equation}
\end{flemma}
%
\begin{proof}
On reprend la \prco de \pref{feqvdim-dimv} donnée dans le cas intègre et on utilise la machinerie \lgbe \elr\ \num1.
\end{proof}

Pour étendre l'\eqvc \cov \pref{feqvdim-dimv} du cas d'un anneau intègre au cas d'un anneau arbitraire, étant donné que dans le cas \gnl on a défini \(\vdim\gR\leq d\)
comme signifiant \(\vdim\Rmin\leq d\), et vu le lemme \ref{flemvdim-dimvQi}, il nous suffit de démontrer \cot l'\eqvc
\begin{equation} \label{feqdimv-dimv}
\dimv\gR\leq d \iff \dimv\Rmin\leq d
\end{equation}
Cela implique en particulier \cot l'analogue de l'\eqvc \pref{feq2thValDim}
pour \(\dimv\).

Cette  \eqvc \pref{feqdimv-dimv} se ramène aux deux lemmes suivants.
\begin{flemma} \label{flem-dimv-dimvred} On a toujours
\begin{equation} \label{feq-dimv-dimvred}
\dimv\gR\leq d \iff \dimv\Rred\leq d
\end{equation}
\end{flemma}
\facile

\begin{flemma} \label{flem-dimv-dimvred2} Soient \(\gR\) un anneau réduit et \(a\in\gR\). Alors on a
\begin{equation} \label{feq-dimv-dimvred2}
\dimv\gR\leq d \iff \dimv\gR_{\so a}\leq d
\end{equation}
\end{flemma}
%
\begin{proof}
L'implication \(\Rightarrow\) est assez simple. D'une part, l'implication \[\dimv\gR\leq d \Rightarrow \dimv(\gR/\fa)\leq d\] est claire pour tout quotient \(\gR/\fa\). Et d'autre part on voit que \[(\dimv\gB\leq d ,\;\dimv\gC \leq d) \Rightarrow \dimv(\gB\times \gC)\leq d.\]
Voyons l'implication réciproque. On considère \(x_0,\dots,x_d\in\gR\). On a tout d'abord un \pol \(P_1(X_0,\dots,X_d)\in\gR[\uX]\) de \coe minimum égal à \(c_1=1+y_1\) et
tel que \(P_1(x_0,\dots,x_d)=z_1\) avec \(y_1,z_1\in a\epr\). On a par ailleurs un \pol \(P_2(\uX)\) de \coe minimum  \(c_2=1+y_2\) et tel que \(P_2(\ux)=z_2\), avec \(y_2,z_2\in (a\epr)\epr\).
On a alors \[y_1y_2=y_1z_2=z_1y_2=z_1z_2=0.\]
Quitte à  multiplier \(P_1\) et~\(P_2\) par des \moms convenables, on peut supposer que leurs \moms initiaux coïncident. Alors \(Q_2=P_2-y_2P_1\)
a pour \coe minimum \((1+y_2)-y_2(1+y_1)=1\) et vérifie \(Q_2(\ux)=z_2\).
De même \(Q_1=P_1-y_1P_2\)
a pour \coe minimum \((1+y_1)-y_1(1+y_2)=1\) et vérifie \(Q_1(\ux)=z_1\). Donc le \pol \(Q_1Q_2\) nous convient.
\end{proof}
%

\subsection{\(\vdim=\Vdim\) dans le cas intègre}\label{fsubsec-vdimv=Vdim1}

\begin{flemma} \label{flem-Vdim<=vdim}
Pour un anneau intègre \(\gR\), et pour un entier \(n\geq -1\)
\[\Vdim\gR\leq n\Rightarrow\vdim\gR\leq n\]
\end{flemma}
%
\begin{proof}
L'article \cite{fCoq2009} montre d'une part que
\[\Vdim\gR\leq n\Rightarrow\Kdim\gR\leq n\]
\vspace{-2em}

\noindent et d'autre part que
\[\Vdim\gR\leq n\Rightarrow\Vdim\RX\leq n+1.\]
Donc  \(\Vdim\gR\leq n\Rightarrow\Kdim\RXn\leq 2n\). Dans \cite{fCACM} il est démontré que
\(\Kdim\RXn\leq 2n\Rightarrow\vdim\gR\leq n\).
\end{proof}

\medskip \noindent \textbf{Nous allons maintenant démontrer l'inégalité opposée.}

\smallskip\noindent
Nous allons démontrer l'implication  $\vdim\leq n\implies\Vdim\leq n$  en construisant des suites \cops. Pour comprendre la \demo, 
nous traiterons d'abord les cas  \(n=2,3,4\): lorsque \(n=2\), 
la suite \cop est fabriquée avec des \elts de la forme~\(\Vp(y)\), tandis que les cas \(n=3\) et \(n=4\) donnent des idées successives pour construire une suite \cop dans le \trdi engendré par ce type d'\elts.

\medskip  \noindent \fbox{\(\vdim\leq 2\Rightarrow\Vdim\leq 2\)}

\smallskip \noindent Supposons que l'on a $x_0,x_1,x_2$ non nuls dans le corps de fractions $\gK$ de $\gR$. On cherche $y_0,y_1,y_2 \in \gK $ (on sait que
$y_2$ est nul) tels que:
\begin{align}
\Vp(x_0) \vee \Vp(y_0) &=1\text,\label{feq:1}\\
\Vp(x_0) \wedge \Vp(y_0) &\leq \Vp(x_1) \vee \Vp(y_1)\text,\label{feq:2}\\
\Vp(x_1) \wedge \Vp(y_1) &\leq \Vp(x_2) \vee \Vp(y_2) =\Vp(x_2)\text{ (car $y_2=0$),}\label{feq:3}\\
\Vp(x_2) \wedge \Vp(y_2) &=0\text{ (c'est certifié par $y_2=0$).}\notag
\end{align}
\smallskip \noindent \eqref{feq:1} revient à dire que $1=\gen{x_0,y_0}$ dans $\gR[x_0,y_0]$.

\noindent \eqref{feq:2} revient à dire que $1=\gen{x_1,y_1}$ dans $\gR[x_0^{-1}, y_0^{-1}, x_1,y_1]$.

\noindent \eqref{feq:3} revient à dire que $1=\gen{x_2}$ dans $\gR[x_1^{-1}, y_1^{-1}, x_2]$.

\smallskip \noindent On utilise le fait que $\Kdim (\gR[x_0,x_1,x_2]) \leq 2$. On a un polynôme $P \in \gR[X_0,X_1,X_2]$ de
plus petit coefficient $1$ (pour l'ordre monomial lexicographique avec $X_2>X_1 > X_0$) qui s'annule  en $(x_0,x_1,x_2)$. Soit $X_2^{n}X_1^mX_0^{\ell}$ le plus petit monôme de $P$.  En divisant $P(x_0,x_1,x_2)$ par $x_2^{n}x_1^mx_0^{\ell}$, on obtient une égalité
$$ 1+ x_0f_0(x_0) + x_1 f_1(x_1,x_0^{\pm}) + x_2 f_2(x_2,x_1^{\pm},x_0^{\pm})=0,$$

\noindent où $f_0\in \gR[X_0]$, $f_1 \in \gR[X_1,X_0^{\pm}]$ et $f_2 \in \gR[X_2,X_1^{\pm},X_0^{\pm}]$ (le $x_0f_0(x_0)$ provient des monômes de $P$ autres que $M$ dont le degré en $X_2$ est
égal à $n$ et le degré en $X_1$ est
égal à $m$,  $x_1 f_1(x_1,x_0^{\pm})$ provient des autres monômes de $P$ dont le degré en $X_2$ est
égal à $n$, alors que $x_2 f_2(x_2,x_1^{\pm},x_0^{\pm})$ provient des monômes de $P$ dont le degré en $X_2$ est
$>n$). Pour  certains $r_0,r_1 \in \mathbb{N}$, on a:
$$ 1+ x_0f_0(x_0) + x_0^{r_0} \bigg( x_1  g_1(x_1,x_0^{-1}) + \frac{x_2}{x_1^{r_1}}g_2(x_2,x_1,x_0^{-1})\bigg)=0 ,
$$
où $g_1 \in \gR[X_1,X_0^{-1}]$ et $g_2 \in \gR[X_2,X_1,X_0^{-1}]$. On voit que $y_0= \frac{1+ x_0f(x_0)}{x_0^{r_0}}$ et $y_1=\frac{x_2}{x_1^{r_1}}$  conviennent (notons que $x_2=x_1^{r_1}y_1 \in \gR[ x_1,y_1] \subseteq \gR[x_0^{-1}, y_0^{-1}, x_1,y_1]$).

\smallskip \noindent Prenons comme exemple d'effondrement:
$$x_0 x_1^2x_2^2 + 2 x_0^2 x_1^2x_2^2 +3  x_1^4x_2^2  + x_0^2 x_1^5x_2^2 + 3x_2^3 +2 x_1x_2^3 + x_0^2 x_1^3x_2^4 = 0.
$$
En divisant par $x_0 x_1^2x_2^2$, on obtient une égalité
$$ 1+ 2x_0^2 + x_0\Big(x_1 \big(\frac{3x_1}{x_0^2} +x_1^2\big) + \frac{x_2}{x_1^2}\big(\frac{3}{x_0^2}+ \frac{2x_1}{x_0^2}+ x_2x_1^3\big)\Big)=0.
$$
On voit que $y_0= \frac{1+ 2x_0^2}{x_0}$ et $y_1=\frac{x_2}{x_1^2}$ conviennent.

\medskip  \noindent \fbox{\(\vdim\leq 3\Rightarrow\Vdim\leq 3\)}

\smallskip \noindent Soient $x_0,x_1,x_2,x_3$ non nuls dans le corps de fractions $\gK$ de $\gR$. On cherche ${\mathfrak u}_0,{\mathfrak u}_1,{\mathfrak u}_2,{\mathfrak u}_3$ dans le treillis distributif engendré par les $V'(x)$ qui forment une suite complémentaire de  $x_0,x_1,x_2,x_3$ (voir les inégalités  $(\ref{feqC2G})$ dans le \thref{fth-dico-trdi-spec-dim1}). Autrement dit: 
{\small\[
1=V'(x_0) \vee {\mathfrak u}_0,\, V'(x_0) \wedge  {\mathfrak u}_0 \leq  V'(x_1) \vee {\mathfrak u}_1,\ldots,\, V'(x_2) \wedge  {\mathfrak u}_2 \leq  V'(x_3) \vee {\mathfrak u}_3,\,  V'(x_3) \wedge {\mathfrak u}_3=0.
\]}

\noindent On propose de trouver les ${\mathfrak u}_i$ sous la forme 
\[
{\mathfrak u}_0=\Vp(y_0),{\mathfrak u}_1=\Vp(y_1) \wedge \Vp(x_0),\,{\mathfrak u}_2= \Vp(y_2) ,\,{\mathfrak u}_3= 0
\]  avec
$y_0,y_1,y_2 \in \gK$, tels que
\begin{align}
\Vp(x_0) \vee \Vp(y_0) &=1,\label{feq:31}\\
  \Vp(x_0) \wedge \Vp(y_0) &\leq \Vp(x_1) \vee \mathfrak u_1= \big( \Vp(x_1) \vee \Vp(y_1) \big) \wedge \big(\Vp(x_1) \vee \Vp(x_0)\big)\text{ ou aussi}\notag\\
  \Vp(x_0) \wedge  \Vp(y_0) &\leq  \Vp(x_1) \vee  \Vp(y_1)\text{ et}\label{feq:32}\\
  \Vp(x_0) \wedge  \Vp(y_0) &\leq \Vp(x_1) \vee \Vp(x_0)\text,\label{feq:32bis}\\
\Vp(x_1) \wedge \mathfrak u_1 &\leq \Vp(x_2) \vee \Vp(y_2)\text,\label{feq:33}\\
\Vp(x_2) \wedge \Vp(y_2) &\leq \Vp(x_3) \vee 
\mathfrak u_3=\Vp(x_3),\label{feq:34}\\
\Vp(x_3) \wedge \mathfrak u_3 &=0\text{ (c'est certifié par $\mathfrak u_3=0$).}\notag
\end{align}
\noindent \eqref{feq:31} revient à dire que $1\in\gen{x_0,y_0}$ dans $\gR[x_0,y_0]$.

\noindent \eqref{feq:32} revient à dire que $1\in\gen{x_1,y_1}$ dans $\gR[x_0^{-1}, y_0^{-1}, x_1,y_1]$.

\noindent \eqref{feq:32bis} est toujours satisfaite car \(\Vp(x_0) \wedge \Vp(y_0) \leq \Vp(x_0) \leq \Vp(x_1) \vee \Vp(x_0)\).

\noindent \eqref{feq:33} revient à dire que $1\in\gen{x_2,y_2}$ dans $\gR[x_1^{-1}, y_1^{-1},x_0^{-1}, x_2,y_2]$.

\noindent \eqref{feq:34} revient à dire que $1\in\gen{x_3}$ dans $\gR[x_2^{-1}, y_2^{-1}, x_3]$.

\smallskip \noindent On utilise le fait que $\Kdim (\gR[x_0,x_1,x_2,x_3]) \leq 3$. \\
On a un polynôme $P \in \gR[X_0,X_1,X_2,X_3]$ de
plus petit coefficient $1$ (pour l'ordre monomial lexicographique avec $X_3>X_2>X_1 > X_0$) qui s'annule  en $(x_0,x_1,x_2,x_3)$. Soit $X_3^{n}X_2^{m}X_1^pX_0^{q}$ le plus petit monôme de $P$.  En divisant $P(x_0,x_1,x_2,x_3)$ par $x_3^{m}x_2^nx_1^{p}x_0^{q}$, on obtient une égalité
\[  
1+ x_0f_0(x_0) + x_1 f_1(x_1,x_0^{\pm1}) + x_2 f_2(x_2,x_1^{\pm1},x_0^{\pm1}) + x_3 f_3(x_3,x_2^{\pm1},x_1^{\pm1},x_0^{\pm1}) =0,
\]
où $f_0\in \gR[X_0]$, $f_1 \in \gR[X_1,X_0^{\pm1}]$, $f_2 \in \gR[X_2,X_1^{\pm1},X_0^{\pm1}]$ et $f_3 \in \gR[X_3,X_2^{\pm1},\allowbreak X_1^{\pm1},\allowbreak X_0^{\pm1}]$:
\begin{itemize}
\item $x_0f_0(x_0)$ provient des monômes de $P$ autres que $M$ dont le degré en $X_3$ est
égal à $m$, le degré en $X_2$ est
égal à $n$ et le degré en $X_1$ est
égal à $p$,
\item $x_1 f_1(x_1,x_0^{\pm})$ provient des autres monômes de $P$ dont le degré en $X_3$ est égal à $m$ et  le degré en $X_2$ est égal à $n$,
\item  $x_2 f_2(x_2,x_1^{\pm},x_0^{\pm})$ provient des autres monômes de $P$ dont le degré en $X_3$ est égal à
$m$,
\item $x_3 f_3(x_3,x_2^{\pm},x_1^{\pm},x_0^{\pm})$ provient des monômes de $P$ dont le degré en $X_3$ est
$>m$)
\end{itemize}
Pour  certains $r_0,r_1 \in \mathbb{N}$, on a:
{\small
\begin{equation} \label{feq17}
1+ x_0f_0(x_0) + x_0^{r_0} \bigl( x_1 g_1(x_1,x_0^{-1}) + {x_1^{r_1}}\bigl(x_2g_2(x_2,x_1^{-1},x_0^{-1})+\frac{x_3}{x_2^{r_2}}\,g_3(x_3,x_2,x_1^{-1},x_0^{-1})\bigr)\bigr)=0 ,
\end{equation}}
où $g_1 \in \gR[X_1,X_0^{-1}]$, $g_2 \in \gR[X_2,X_1^{-1},X_0^{-1}]$ et $g_3\in\gR[X_3,X_2,X_1^{-1},X_0^{-1}]$.
\\
On pose $y_0= \dfrac{1+x_0f_0(x_0)}{x_0^{r_0}}$,
$y_1=\dfrac{y_0+x_1 g_1(x_1,x_0^{-1})}{x_1^{r_1}}$, $y_2=\dfrac{x_3}{x_2^{r_2}}$.
\\
Condition \eqref{feq:31}: $1= x_0^{r_0}y_0-x_0f_0(x_0)\in\gen{x_0,y_0}$ dans $\gR[x_0,y_0]$ (m\^eme si \(r_0=0\)).
\\
Condition \eqref{feq:32}: $y_0=y_1x_1^{r_1}- x_1  g_1(x_1,x_0^{-1})$,
 si $y_0\neq 0$ on divise par $y_0$
et l'on obtient $1\in\gen{x_1,y_1}$
dans $\gR[x_1,y_1,x_0^{-1},y_0^{-1}]$.
\\
Condition \eqref{feq:34}: $ y_2x_2^{r_2}=x_3$, (et on divise par $y_2x_2^{r_2}$  si $y_2\neq 0$)
\\
Condition \eqref{feq:33}: l'égalité (\ref{feq17}) se relit comme suit
\[
x_2g_2(x_2,x_1^{-1},x_0^{-1})+y_2g_3(y_2x_2^{r_2},x_2,x_1^{-1},x_0^{-1})  = -y_1
\]
Si $y_1\neq 0$ on divise par $y_1$.
\\
Par ailleurs on a aussi un bon résultat si $y_0=0$ ou $y_1=0$ (voir le commentaire après l'\egt (\ref{feqV'})).

\medskip  \noindent \fbox{\(\vdim\leq 4\Rightarrow\Vdim\leq 4\)}

\smallskip \noindent  Soient $x_0,x_1,x_2,x_3,x_4$ non nuls dans le corps de fractions $\gK$ de $\gR$. On cherche ${\mathfrak u}_0,{\mathfrak u}_1,{\mathfrak u}_2,{\mathfrak u}_3,\allowbreak{\mathfrak u}_4$ dans le treillis distributif engendré par les $V'(x)$ qui forment une suite complémentaire de  $x_0,x_1,x_2,x_3,x_4$ (voir les inégalités $(\ref{feqC2G})$ dans le \thref{fth-dico-trdi-spec-dim1}). Autrement dit: 
\[
1=V'(x_0) \vee {\mathfrak u}_0,\, V'(x_0) \wedge  {\mathfrak u}_0 \leq  V'(x_1) \vee {\mathfrak u}_1,\ldots,\, V'(x_3) \wedge  {\mathfrak u}_3 \leq  V'(x_4) \vee {\mathfrak u}_4, ,\,  V'(x_4) \wedge {\mathfrak u}_4=0.
\]
On propose de trouver les ${\mathfrak u}_i$ sous la forme
{\small
\[{\mathfrak u}_0=\Vp(y_0),{\mathfrak u}_1=\Vp(y_1) \wedge \Vp(x_0),{\mathfrak u}_2= V'(y_2) \wedge V'(x_1) \wedge V'(x_0),{\mathfrak u}_3= \Vp(y_3),{\mathfrak u}_4= 0   
\]}

\noindent avec
$y_0,y_1,y_2,y_3 \in \gK$, tels que

\smallskip \noindent 
\begin{align}
V'(x_0) \vee V'(y_0) &=1\text,\label{feq:41}\\
V'(x_0) \wedge V'(y_0) &\leq V'(x_1) \vee \mathfrak u_1\text{, ou aussi}\notag\\
V'(x_0) \wedge V'(y_0) &\leq V'(x_1) \vee V'(y_1)\text{,}\label{feq:42}\\
V'(x_1) \wedge \mathfrak u_1 &\leq V'(x_2) \vee \mathfrak u_2\text{, ou aussi}\notag\\
V'(x_1) \wedge V'(y_1) \wedge V'(x_0) &\leq V'(x_2) \vee V'(y_2)\text{ et}\label{feq:43}\\
V'(x_2) \wedge \mathfrak u_2 &\leq V'(x_3) \vee V'(y_3)\text,\label{feq:44}\\
V'(x_3) \wedge V'(y_3) &\leq V'(x_4) \vee \mathfrak u_4= V'(x_4)\text,\label{feq:45}\\
V'(x_4) \wedge \mathfrak u_4 &=0\text{ (c'est certifié par $\mathfrak u_4=0$).}\notag
\end{align}
\noindent \eqref{feq:41} revient à dire que $1\in\gen{x_0,y_0}$ dans $\R[x_0,y_0]$.

\noindent \eqref{feq:42} revient à dire que $1\in\gen{x_1,y_1}$ dans $\R[y_0^{-1},x_0^{-1} , x_1,y_1]$.

\noindent \eqref{feq:43} revient à dire que $1\in\gen{x_2,y_2}$ dans $\R[y_1^{-1},x_1^{-1},x_0^{-1}, x_2,y_2]$.

\noindent \eqref{feq:44} revient à dire que $1\in\gen{x_3,y_3}$ dans $\R[y_2^{-1},x_2^{-1},x_1^{-1},x_0^{-1}, x_3,y_3]$.

\noindent \eqref{feq:45} revient à dire que $1\in\gen{x_4}$ dans $\R[x_3^{-1}, y_3^{-1}, x_4]$.

\medskip

\noindent On utilise le fait que ${\rm Kdim}\; \R[x_0,x_1,x_2,x_3,x_4] \leq 4$. On a un polynôme $P \in \alb\R[X_0,\allowbreak\dots,X_4]$ de
plus petit coefficient $1$ (pour l'ordre monomial lexicographique avec $X_4>\cdots > X_0$) qui s'annule  en $(x_0,x_1,x_2,x_3,x_4)$. \\
Soit $X_4^{\ell}X_3^{n}X_2^{m}X_1^pX_0^{q}$ le plus petit monôme de $P$.  En divisant $P(x_0,x_1,x_2,x_3,x_4)$ par $x_4^{\ell}x_3^{m}x_2^nx_1^{p}x_0^{q}$, on obtient une égalité
\begin{multline*}
  1+ x_0f_0(x_0) + x_1 f_1(x_1,x_0^{\pm1}) + x_2 f_2(x_2,x_1^{\pm1},x_0^{\pm1}) +x_3 f_3(x_3,x_2^{\pm1},x_1^{\pm1},x_0^{\pm1})\\
  + x_4 f_4(x_4,x_3^{\pm1},x_2^{\pm1},x_1^{\pm1},x_0^{\pm1})=0\text,
\end{multline*}
o\`u $f_0\in \R[X_0]$, $f_1 \in \R[X_1,X_0^{\pm}]$, $f_2 \in \R[X_2,X_1^{\pm},X_0^{\pm}]$, $f_3 \in \R[X_3,X_2^{\pm},X_1^{\pm},X_0^{\pm}]$ et $f_4 \in \R[X_4,X_3^{\pm},X_2^{\pm},X_1^{\pm},X_0^{\pm}]$:
\begin{itemize}
\item $x_0f_0(x_0)$ provient des monômes de $P$ autres que $M$ dont le degré en $X_4$ est
égal à~$\ell$, le degré en $X_3$ est
égal à $m$, le degré en $X_2$ est
égal à $n$ et le degré en $X_1$ est
égal à $p$,
\item $x_1 f_1(x_1,x_0^{\pm})$ provient des autres monômes de $P$ dont le degré en $X_4$ est
égal à $\ell$, le degré en $X_3$ est
égal à $m$, le degré en $X_2$ est
égal à $n$,
\item  $x_2 f_2(x_2,x_1^{\pm},x_0^{\pm})$ provient des autres monômes de $P$ dont le degré en $X_4$ est
égal à~$\ell$ et le degré en $X_3$ est
égal à $m$,

\item $x_3 f_3(x_3,x_2^{\pm},x_1^{\pm},x_0^{\pm})$ provient des monômes de $P$ dont le degré en $X_4$ est
égal à $\ell$,
\item $x_4 f_4(x_4,x_3^{\pm},x_2^{\pm},x_1^{\pm},x_0^{\pm})$ provient des monômes de $P$ dont le degré en $X_4$ est
$>\ell$.
\end{itemize}
Pour  certains $r_0,r_1,r_2,r_3 \in \mathbb{N}$, on a:
\begin{multline}
  1+ x_0f_0(x_0) + x_0^{r_0}\Bigl( x_1 g_1(x_1,x_0^{-1}) + {x_1^{r_1}}\Bigl(x_2g_2(x_2,x_1^{-1},x_0^{-1})\\
  +x_2^{r_2}\bigl(x_3g_3(x_3,x_2^{-1},x_1^{-1},x_0^{-1})+\frac{x_4}{x_3^{r_3}}\,g_4(x_4,x_3,x_2^{-1},x_1^{-1},x_0^{-1})\bigr)\Bigr)\Bigr)=0,\label{feq:18}
\end{multline}
o\`u $g_1 \in \R[X_1,X_0^{-1}]$, $g_2 \in \R[X_2,X_1^{-1},X_0^{-1}]$, $g_3\in\R[X_3,X_2^{-1},X_1^{-1},X_0^{-1}]$ et
$g_4\in\R[X_4,\allowbreak X_3,X_2^{-1},X_1^{-1},X_0^{-1}]$.
\\
On pose $y_0= \frac{1+x_0f_0(x_0)}{x_0^{r_0}}$,
$y_1=\frac{y_0+x_1  g_1(x_1,x_0^{-1})}{x_1^{r_1}}$, $y_2=\frac{y_1+x_2  g_2(x_2,x_1^{-1},x_0^{-1})}{x_2^{r_2}}$ et $y_3=\frac{x_4}{x_3^{r_3}}$.
\\
Condition \eqref{feq:41}: $1= x_0^{r_0}y_0-x_0f_0(x_0)\in\gen{x_0,y_0}$ dans $\R[x_0,y_0]$ (m\^eme si \(r_0=0\)).
\\
Condition \eqref{feq:42}: $y_0=y_1x_1^{r_1}- x_1  g_1(x_1,x_0^{-1})$,
 si $y_0\neq 0$ on divise par $y_0$
et l'on obtient $1\in\gen{x_1,y_1}$
dans $\R[x_1,y_1,x_0^{-1},y_0^{-1}]$.
\\
Condition \eqref{feq:43}: $y_1=y_2x_2^{r_2}- x_2  g_2(x_2,x_1^{-1},x_0^{-1})$,
 si $y_1\neq 0$ on divise par $y_1$
et l'on obtient $1\in\gen{x_2,y_2}$ dans $\R[y_1^{-1},x_1^{-1},x_0^{-1}, x_2,y_2]$.
\\
Condition \eqref{feq:44}: l'égalité \eqref{feq:18} se relit
\[
x_3g_3(x_3,x_2^{-1},x_1^{-1},x_0^{-1})+y_3g_4(y_3x_3^{r_3},x_3,x_2^{-1},x_1^{-1},x_0^{-1})  = -y_2
\]
Si $y_2\neq 0$ on divise par $y_2$ et l'on obtient $1\in\gen{x_3,y_3}$ dans $\R[y_2^{-1},x_2^{-1},x_1^{-1},x_0^{-1}, x_3,y_3]$.
\\
Condition \eqref{feq:45}: $ y_3x_3^{r_3}=x_4$, (et on divise par $y_3x_3^{r_3}$  si $y_3\neq 0$).

\medskip  \noindent \fbox{\(\vdim\leq n\Rightarrow\Vdim\leq n\)}

\smallskip \noindent  Soient $x_0,\dots,x_n$ non nuls dans le corps de fractions $\gK$ de $\gR$. 
On cherche ${\mathfrak u}_0,{\mathfrak u}_1,\dots,{\mathfrak u}_n$ dans le treillis distributif engendré par les $V'(x)$ qui forment une suite complémentaire de   $x_0,x_1,\dots,x_n$. Autrement dit: 
{\small\[
1=V'(x_0) \vee {\mathfrak u}_0,\, V'(x_0) \wedge  {\mathfrak u}_0 \leq  V'(x_1) \vee {\mathfrak u}_1,\ldots,\,V'(x_{n-1}) \wedge  {\mathfrak u}_{n-1} \leq  V'(x_n) \vee {\mathfrak u}_n ,\,  V'(x_n) \wedge {\mathfrak u}_n=0.
\]
}

\noindent On utilise le fait que ${\rm Kdim}\; \R[x_0,\ldots,x_n] \leq n$. On a un polynôme $P \in \R[X_0,\ldots,X_n]$ de
plus petit coefficient $1$ (pour l'ordre monomial lexicographique avec $X_n>X_{n-1}> \cdots > X_0$) qui s'annule  en $(x_0,\ldots,x_n)$. Soit $X_n^{q_n} \cdots X_0^{q_0}$ le plus petit monôme de~$P$.  En divisant $P(x_0,\ldots,x_n)$ par $x_n^{q_n} \cdots x_0^{q_0}$, on obtient une égalité
\[1+ x_0f_0(x_0) + x_1 f_1(x_1,x_0^{\pm}) + \ldots +
 x_n f_n(x_n,x_{n-1}^{\pm},\ldots,x_0^{\pm})=0,
\]
o\`u $f_0\in \R[X_0], \,f_1 \in \R[X_1,X_0^{\pm}],\ldots, \, f_n \in \R[X_n,X_{n-1}^{\pm},\ldots,X_0^{\pm}]$.

\medskip

\noindent Pour  certains $r_0,\ldots,r_n \in \mathbb{N}$, on a:
\[ 
\begin{array}{cl} 
1+ x_0f_0(x_0) + x_0^{r_0} \bigg( x_1  g_1(x_1,x_0^{-1}) +\cdots + x_{n-3}^{r_{n-3}}\Big( x_{n-2}g_{n-2}(x_{n-2},x_{n-3}^{-1},\ldots,x_0^{-1})
\\
+ x_{n-2}^{r_{n-2}}\big(x_{n-1}g_{n-1}(x_{n-1},x_{n-2}^{-1},\ldots,x_0^{-1})+
\frac{x_n}{x_{n-1}^{r_{n-1}}}\,g_n(x_n,x_{n-1},x_{n-2}^{-1},\ldots x_0^{-1}) \cdots \big) \Big)\bigg)    &=0, 
  \end{array}
\]
 o\`u $g_1 \in \R[X_1,X_0^{-1}],\ldots, \, g_{n-1}\in\R[X_{n-1},X_{n-2}^{-1},\ldots,X_0^r{-1}],$
$g_n\in\R[X_n,X_{n-1}, X_{n-2}^{-1},\ldots,\alb X_0^{-1}]$.
\\
On pose $y_0= \frac{1+x_0f_0(x_0)}{x_0^{r_0}},\,
y_1=\frac{y_0+x_1  g_1(x_1,x_0^{-1})}{x_1^{r_1}}, \ldots,\,y_{n-2}=\frac{y_{n-3}+x_{n-2}  g_{n-2}(x_{n-2},x_{n-3}^{-1},\ldots,x_0^{-1})}{x_{n-2}^{r_{n-2}}}$ et $y_{n-1}=\frac{x_n}{x_{n-1}^{r_{n-1}}}$.

\smallskip \noindent Il suffit alors de prendre ${\mathfrak u}_0=V'(y_0),\, {\mathfrak u}_1=V'(y_1)\wedge V'(x_0),{\mathfrak u}_2=V'(y_2)\wedge V'(x_1) \wedge V'(x_0),\ldots,
{\mathfrak u}_{n-2}=V'(y_{n-2})\wedge V'(x_{n-3}) \wedge \cdots \wedge V'(x_0),\, {\mathfrak u}_{n-1}=V'(y_{n-1})$ et  ${\mathfrak u}_n=0$.

\subsection{\(\vdim=\Vdim\) dans le cas général}\label{fsubsec-vdim=Vdim2}

On note que lorsque \(\gR\) est un anneau \qi, l'anneau \(\Frac(\gR)\)
est un anneau \zedr. Par ailleurs un \cdi est un anneau \zedr dans lequel tout \idm est égal à \(0\) ou \(1\).

Nous commençons alors par la remarque suivante qui découle de la machinerie \cov\ \elr\ \num 1.
\begin{flemma} \label{fVdimqi}
Soit \(\gR\) un anneau \qi. Définissons \(\Val(\gR):=\Val(\Frac(\gR),\gR)\)  et \(\Vdim(\gR)=\Kdim(\Val(\gR))\) comme dans \pref{feq01} en remplaçant \(\gK\) par \(\Frac(\gR)\).
Alors on obtient l'\egt \(\vdim(\gR)=\Vdim(\gR)\) comme dans le cas intègre.
\end{flemma}

Dans la suite de ce paragraphe, nous ne donnons pas de démonstration.

Nous renvoyons
à l'étude  \gnle \cite{fLM2022}.

\smallskip
Nous définissons une \tdy \sa{val} comme suit. On considère la signature
$$\Sigma_\mathrm{val}=(\,\cdot\di\cdot  \mathrel{;} \cdot+\cdot, \cdot\times\cdot,-\,\cdot,0,1\,)$$
Les axiomes sont les suivants.

\DeuxRegles
{
\lab{Col$_{\val}$} $\,\, 0\di 1 \vd \Bot$ \quad (effondrement)
\fLab{av1}  $\vd 1 \di  -1$
\fLab{av2} $\,\,a \di  b \Vd ac \di  bc$
\fLab{Av1} $\,\,a \di  b \vet b \di  c \Vd a \di  c$
}
{
\fLab{Av2} $\,\,a \di  b \vet a \di  c \Vd a \di  b + c$
\fLab{AV1} $\vd a \di  b \vou b\di a$
\fLab{AV2} $\,\,ax \di  bx  \Vd a \di  b \vou 0 \di x$
}

L'\egt \(x=0\) est définie comme une abrviation de \(x\di 0\).

\begin{fdefinition} \label{fdefivalkK}~
\begin{enumerate}
\item Si \(\gR\) est un anneau commutatif, la \sad $\sa{val}(\gR)$ est obtenue en prenant comme \pn le diagramme positif de l'anneau \(\gR\).
\item Si $\gk\subseteq \gR$ sont deux anneaux\footnote{Nous utilisons \(\gk\) comme petit anneau pour nous référer à l'intuition donnée dans la situation fréquente où \(\gk\) est un \cdi.
}, ou plus \gnlt si $\varphi:\gk\to\gR$ est une \alg, on note $\sa{val}(\gk,\gR)$ la \sad  dont la présentation est donnée par
\begin{itemize}
\item le diagramme positif de~$\gR$ comme anneau commutatif;
\item les axiomes
$\vd 1 \di \varphi(x)$ pour les \elts $x$ de $\gk$.
%
\end{itemize}

%
%
\end{enumerate}
\end{fdefinition}

Les deux \sads $\sa{val}(\gR)$ et $\sa{val}(\gZ,\gR)$, où \(\gZ\) est le plus petit sous-anneau de \(\gR\), sont canoniquement isomorphes.

\begin{fdefinition} \label{fdefivalkK0}
Soit $\gk$ un sous-anneau d'un anneau $\gR$.
On définit le \trdi $\val(\gR,\gk)$ par la \entrel 
 $\vdash_{\gR,\gk,\mathrm{val}}$ sur l'ensemble $\gR\times\gR$ donnée par l'\eqvc suivante. 
\begin{equation} \label {feqvalkK}
\begin{aligned} 
 (a_1,b_1),\dots,(a_n,b_n) &\,\vdash_{\gR,\gk,\mathrm{val}}\, (c_1,d_1),\dots,(c_m,d_m)\\ 
\equidef\quad  a_1\di b_1\tsbf,\dots\tsbf,\, a_n\di b_n &\Vdi{\sA{val}(\gR,\gk)} c_1\di d_1\tsbf{ or } \dots\tsbf{ or } c_m\di d_m\text.
 \end{aligned}
\end{equation}
\end{fdefinition}

On peut démontrer le résultat suivant.

\begin{flemma} \label{flemvalVal}
Soit \(\gR\)  un anneau intègre de corps de fraction \(\gK\). On a des morphismes naturels \(\Val(\gR,\gK)\to\val(\gR,\gK)\) et \(\val(\gR,\gR)\to\val(\gR,\gK)\). Ce sont des \isos.
\end{flemma}

La \dfn suivante est donc raisonnable. Nous allons voir qu'elle coïncide avec celle du lemme \ref{fVdimqi} dans le cas des anneaux \qis.

\begin{fdefinition} \label{fdefiVdimgnle}
Pour un anneau commutatif \(\gR\) arbitraire on définit \(\Vdim(\gR)\leq n\) par
\(\Kdim(\val(\gR,\gR))\leq n\).
\end{fdefinition}

Cette dimension coïncide avec celle déjà définie lorsque \(\gR\) est intègre. Mais elle n'est pas en \gnl égale à la dimension du treillis \(\val(\gR,\Frac(\gR))\).

De notre point de vue, cela signifie que \(\Frac(\Rmin)\) est un bien meilleur substitut que \(\Frac(\gR)\)  au corps de fractions lorsque \(\gR\)
n'est pas un anneau intègre. En fait \(\Rmin\) coïncide avec \(\Frac(\gR)\)
seulement pour les anneaux \qis.

Enfin on peut démontrer le \tho suivant.

\begin{ftheorem} \label{fthvdimAmin}
Les \trdis \[\val(\gR,\gR)  \hbox{ et }
 \val(\Rmin,\Rmin)\simeq\val(\Rmin,\Frac(\Rmin))\]
ont même dimension de Krull.
\end{ftheorem}

Avec le lemme \ref{fVdimqi} et les \dfns \ref{fdefivalkK} et \ref{fdefiVdimgnle} cela termine le travail.

\smallskip \noindent \textsl{Note}. On aurait pu définir directement \(\Vdim(\gR)=\Vdim(\Rmin)\) sans utiliser
la théorie \sa{val}, mais cela ressemblait un peu trop à priori à une \dfn ad hoc, car \(\Vdim(\gR)\) n'a pas de \dfn naturelle directe pour un anneau arbitraire si on utilise
uniquement la théorie \sa{Val}.

\bibliographystyle{plainnat-fr}

\endgroup
\end{document}